\documentclass{amsart}

\usepackage{graphicx}
\usepackage{caption}
\usepackage[utf8]{inputenc}
\usepackage{nicefrac}
\usepackage{mathabx}
\usepackage{amsrefs}
\usepackage{wrapfig}

\newtheorem{theorem}{Theorem}[section]
\newtheorem{lemma}[theorem]{Lemma}

\theoremstyle{definition}
\newtheorem{definition}[theorem]{Definition}

\newtheorem{prop}[theorem]{Proposition}
\newtheorem{cor}[theorem]{Corollary}
\newtheorem{ass}{Assumption}
\newtheorem{algo}[theorem]{Algorithm}

\allowdisplaybreaks
\theoremstyle{remark}
\newtheorem{remark}[theorem]{Remark}

\numberwithin{equation}{section}

\newcommand{\E}{\mathcal{E}}
\newcommand{\divi}{\mathrm{div}}
\newcommand{\dist}{\mathrm{dist}}

\begin{document}

\title[On Gradient Flows with Obstacles and Euler's Elastica]{On Gradient Flows with Obstacles and \\ Euler's Elastica}

\author{Marius Müller}
\address{Institut für Analysis, Universität Ulm, 89069 Ulm}
\email{marius.mueller@uni-ulm.de}
\thanks{The author is supported by the LGFG Grant (Grant no. 1705 LGFG-E) and would like to thank Anna Dall'Acqua and Fabian Rupp for helpful discussions.}
%





\begin{abstract}
We examine a steepest energy descent flow with obstacle constraint in higher order energy frameworks where the maximum principle is not available. We construct the flow under general assumptions using De Giorgi's minimizing movement scheme. Our main application will be the elastic flow of graph curves with Navier boundary conditions for which we study long-time existence and asymptotic behavior. 
\end{abstract}
\maketitle
\section{Introduction}
Under Obstacle Problems one understands the question whether a given energy functional attains a minimum in the set $C_\psi := \{ u \in V : u(x) \geq \psi(x) \, a.e. \}$, where $V$ is a space of Lebesgue-measurable functions and $\psi$ is a certain measurable function such that $C_\psi \neq \emptyset$, the so-called \emph{obstacle}. The question can be posed more generally as minimization of an energy functional in some convex closed subset of a Banach space $V$, but since $C_\psi$ has more structural properties than just convexity and closedness, one could also impose more conditions on $C$. Part of the goal of this article is to identify reasonable structural requirements in a more abstract framework, e.g. if $V$ is not required to be a space of functions. 

Such obstacle problems have various applications in physics and finance. A very important concept in this field is that of a variational inequality. This can be stated as follows: If $V$ is a Banach Space, $C \subset V$ is convex and closed and $\E \in C^1(V, \mathbb{R})$ is an energy functional then  
\begin{equation*}
u \in C \textrm{ is such that } \E(u) = \inf_{w \in C} \E(w)  \quad \Rightarrow \quad 0 \leq D\E(u)(v-u) \quad \forall v \in C , 
\end{equation*}
where $D\E: V \rightarrow V'$ denotes the Frechét derivative of $\E$. 

In this article we are especially interested in the case when $V$ is a Hilbert Space. The energy we mainly focus on is the elastic bending energy for graphs with fixed ends, i.e. 
\begin{equation}\label{eq:elasenerg}
\E: W^{2,2}(0,1) \cap W_0^{1,2}(0,1) \rightarrow \mathbb{R} , \quad \E(u) := \int_0^1 \frac{u''(x)^2}{(1+ u'(x)^2)^\frac{5}{2}} \; \mathrm{d}x .
\end{equation}
This energy is motivated as follows: If $\gamma \in C^\infty((a,b); \mathbb{R}^2)$ is an immersed plane curve, then one can define its one-dimensional Willmore energy to be 
\begin{equation}\label{eq:willmore}
\mathcal{W}(\gamma) := \int_\gamma \kappa^2(s) d\mathbf{s}, 
\end{equation} 
where $\kappa$ denotes the curvature of $\gamma$ and $d\mathbf{s}$ denotes the arclength element. If $\gamma : (0,1) \rightarrow \mathbb{R}^2$ is a \emph{graph curve}, i.e. $\gamma(x) = (x, u(x))$ for $x \in (0,1)$, then $\E(u) =\mathcal{W}(\gamma)$.  
The obstacle problem for elastic curves in the framework presented above has already been studied in \cite{Anna} and \cite{Mueller}. Elastic Curves with other obstacle-type confinements have also recently raised a lot of interest, see \cite{Novaga}. 

In this article we want to construct gradient flows respecting the obstacle condition. In this framework, a gradient flow should be understood as a flow that realizes the steepest possible energy descent in the class of admissible curves. The so-called steepest descent curve approach is also the starting point for \cite{Ambrosio} to define gradient flows in metric spaces $(X,d)$. Since each convex closed subset of a Banach space forms a metric space, one could think that the gradient flow we are interested in is already constructed in \cite{Ambrosio}. However, in \cite{Ambrosio} the authors impose a condition on the energy that the elastic energy does not satisfy. We want to comment on this condition shortly.  For a metric space $(X,d)$ one defines the metric slope of $\E: X \rightarrow \overline{\mathbb{R}}$ to be 
\begin{equation*}
|\partial \E| (u) := \limsup_{v \rightarrow u } \frac{(\E(u) - \E(v))^+}{d(u,v)}.
\end{equation*} 
If $X = C \subset H$ is a convex subset of a Hilbert space $H$ and $\E : H \rightarrow \mathbb{R}$,  $\E \in C^{1,1}_{loc}(\{ \E< \infty \} , \mathbb{R})$ , then one has (see Proposition \ref{prop:quantity} below) 
\begin{equation*}
| \partial \E | (u) = \sup_{ v \neq u, \E(v) < \infty  } \frac{(-D\E(u)(v-u))^+}{||v-u||}.
\end{equation*} 
In order to show existence of the flow, \cite[Equation (2.3.1)]{Ambrosio} requires that $|\partial \E|$ is weakly lower semicontinuous on $\{\E < \infty \} $ which is in general not true for nonlinear evolutions such as the evolution by elastic flow. If this condition is not satisfied, then one obtains the gradient flow only with a relaxed version of $|\partial \E|$, which is denoted by $|\partial^- \E |$ in \cite{Ambrosio}. Nevertheless, we can construct a steepest-descent flow  without relaxation for the elastic energy and even in a larger class of higher order energies. Existence and asymptotic behavior of this flow are proved in  Theorem \ref{thm:exisres}, Theorem \ref{thm:navier} and Theorem \ref{thm:subcon}, which are our main results.

Important progress on the field of higher order gradient flows with obstacle constraint has been made in \cite{Okabe1} and \cite{Okabe2}, where $H = L^2(\Omega)$ for some open bounded $\Omega \subset \mathbb{R}^n$ and 
\begin{equation*}
\E(u) := \begin{cases}
\int_\Omega  ( \Delta u)^2 \; \mathrm{d}x & u \in W_0^{2,2}(\Omega) \\ 
\infty  & u \not \in W_0^{2,2}(\Omega)
\end{cases} .
\end{equation*} 
Even though $|\partial^- \E| = |\partial \E|$ in this case, the article provides plenty of new insights in regularity of the obstacle gradient flow, which can hardly be discussed in the general framework in \cite{Ambrosio}. The reason why the energy is extended by infinity is, that one obtains an $L^2$-gradient flow in the end, which is actually more desireable than a gradient flow in $W^{2,2}$, since the PDE for the flow is usually easier. Unfortunately, we are unable to construct an $L^2$-gradient flow for the elastic energy respecting the obstacle condition at this point. In \cite{Okabe1}, the existence problem for this flow is presented as an open question and one of the main motivations of the authors to consider parabolic fourth order obstacle problems. The present article can be understood as some progress on this question.

The main method to show existence in \cite{Ambrosio}, \cite{Okabe1} and \cite{Okabe2} is the De Giorgi Minimizing Movement Scheme, which we will define later. For fundamental literature on the scheme see \cite[Chapter 3]{Ambrosio}. It serves as a discrete approximation of the flow and has been used by several authors to construct gradient flows.  Remarkable is an important application to optimal transport, the so-called JKO-scheme, see \cite{JKO}. Moreover it has already been applied to the elastic flow without obstacles, see \cite{Pozzi}.

The article is organized as follows: In Chapter 2 we clarify what we understand by an \emph{Obstacle Gradient Flow} in a Hilbert Space, show first properties of the defined flow and explain why the definition is consistent with the notion of a steepest energy descent flow in a metric space. Furthermore, we will discuss how the structure of an obstacle problem can be understood in a Hilbert Space which is not necessarily a space of functions. In Chapter 3 we examine a condition that is sufficient for the convergence of the De Giorgi Minimizing Movement Scheme to an Obstacle Gradient Flow that exists for all positive times. In Chapter 4 we find a class of higher order energies for which the aforementioned condition is satisfied and examine space and time regularity of the flow. The Elastic Energy defined in \eqref{eq:elasenerg} is a member of the considered class.  Chapter 5 deals with the elastic flow only and uses identities for higher derivatives from Chapter 4 to study long time behavior and asymptotics of the Obstacle Gradient Flow. The main result in Chapter 5 identifies two possible asymptotic behaviors, both of which may occur as we show in the end of the article. Fortunately, one of the two possible behaviors is convergence to a 'critical point', i.e. a solution of the variational inequality. 
\section{The Notion of Obstacle Gradient Flows}

In the following, $H$ denotes a real Hilbert Space with scalar product $(\cdot, \cdot)$ and norm $||\cdot||$. For a Frechét differentiable map $\E:H \rightarrow H$ and $u \in H$, $\nabla_H \E$ or just $\nabla \E$ denotes the $H-$Gradient of $\E$, i.e. the unique element $v \in H$ such that $D\E(u)(\phi) = (v, \phi)$ for all $\phi \in H$.  

\subsection{Definition and First Properties}

\begin{ass}[Assumptions on the Energy]\label{ass:main}
 
 We assume that  $\mathcal{E} \in C^{1,1}_{loc} (H ,\mathbb{R})$  is weakly lower semicontinuous and bounded from below by some $\alpha \in \mathbb{R}$.
 Furthermore, we assume that 
 \begin{equation*}
  \sup_{ w \in B_R(0) } || \nabla_H \E(w) || < \infty \quad \forall R > 0 . 
 \end{equation*}
 \end{ass}
 
\begin{ass} [Gradient Growth Condition]\label{ass:growth}

Another condition we will sometimes impose is that there is $\zeta: \mathbb{R}_+ \rightarrow \mathbb{R}_+$ nondecreasing and continuous such that for each $u,v \in H$ 
\begin{equation*}
||\nabla_H \mathcal{E}(u) - \nabla_H \mathcal{E}(v) || \leq \zeta(||u||+ ||v||) ||u-v|| .
\end{equation*}
\end{ass}

\begin{definition}[Obstacle Gradient Flow]\label{def:coneflow}

Let $H$ be a real Hilbert space and $C \subset H$ be a closed convex set. Suppose that $\E: H \rightarrow \mathbb{R}$ satisfies Assumption \ref{ass:main}. A curve $(u(t))_{t\geq 0 }$ in $H$ is called \emph{Obstacle Gradient Flow} with initial datum $u_0\in C$ if it satisfies
\begin{enumerate}
\item (Initial Datum) $u(0) = u_0$.
\item (Regularity) $u \in W^{1,2}((0,T),H)$ for each $T> 0$, $u(t) \in C$ for every $t\geq 0$.
\item (Flow Variational Inequality)  It holds for almost every $t > 0$ that
\begin{equation*}
(FVI)  \qquad (\dot{u}(t) , v- u(t)) + (\nabla_H \E(u(t)), v - u(t) ) \geq 0  \quad \forall v \in C.
\end{equation*}
\end{enumerate}
\end{definition}
 \begin{remark}
In the following, we need some more assumptions on $u_0$ to prove existence of the flow. These will be discussed in Section \ref{sec:existence}. 
\end{remark}
\begin{remark}
Requiring that $u(t) \in C$ for every $t\geq 0$ is only meaningful if we take the $C([0,\infty),H)$-representative of $u$, which we will always do unless otherwise specified. 
\end{remark}
\begin{remark}
From this point we will refer to the Flow Variational Inequality with $(FVI)$ as an abbreviation. Because of the $(FVI)$ requirement, our definition of the Obstacle Gradient Flow is slightly stronger than the definition of a weak solution in \cite[Definition 1.1]{Okabe1}.
\end{remark}
%

\begin{prop}[Energy Dissipation] \label{prop:enrggdisp}
Assume Assumption \ref{ass:main}. Let $(u(t))_{t \geq 0 }$ be an Obstacle Gradient Flow. Then $\E \circ u $ is noncreasing and lies in $W^{1,1}_{loc}(0,\infty)$. Furthermore 
$\dot{u} \in L^2(0,\infty)$ and  
\begin{equation}\label{eq:abl}
- ||\dot{u}(t) ||^2  = \frac{d}{dt} \mathcal{E}(u(t)) \quad a.e.   
\end{equation} 
\end{prop}
\begin{proof}
Recall from \cite[Theorem 1.4.35]{Cazenave} that $u(t)$ is almost everywhere Frechét differentiable.
Fix $t,s > 0 $ be such that $u$ is Frechét differentiable at $t$ and $(FVI)$ holds at $t$. Note that $(0,1) \ni r \mapsto \nabla \E(r u(t) + (1-r)u(s)) \in H $ is Bochner measurable as a continuous function. It is also bounded by Assumption \ref{ass:main}. One computes using \cite[Proposition 1.1.6]{Arendt}
\begin{align*}
\frac{\E(u(t)) - \E(u(s))}{t-s} & =  \frac{1}{t-s}\int_0^1\left( \nabla \E(r u(t) + (1-r) u(s))  , u(t)- u(s) \right)\; \mathrm{d}r \\ & = \left( \int_0^1 \nabla\E(ru(t) + (1-r)u(s))\; \mathrm{d}r , \frac{u(t) - u(s)}{t-s} \right)  . 
\end{align*}
Using that  
\begin{equation*}
\int_0^1 \nabla \E ( r u(t) + (1-r) u(s) ) \; \mathrm{d}r \rightarrow  \nabla \E(u(t))   \quad (s \rightarrow t) 
\end{equation*}
in $H$ we find 
\begin{equation}\label{eq:decay}
\frac{d}{dt} \E(u(t)) =  (\nabla\E(u(t), \dot{u}(t) )  \quad a.e. \; \; . 
\end{equation}
Let $|h| < t$ be arbitrary. Use $(FVI)$ from Definition \ref{def:coneflow} and plug in $v = u(t+h)$ to find 
\begin{equation*}
(\dot{u}(t) , u(t+h) - u(t)) + ( \nabla\E(u(t) , u(t+h) - u(t) ) \geq 0 
\end{equation*}
Recall also that $t$ was chosen such that $u$ is Frechét differentiable at $t$. If  $h > 0 $, dividing by $h$ and letting $h \rightarrow 0 + $ yields 
\begin{equation*}
||\dot{u}(t) ||^2 + (\nabla\E(u(t)) , \dot{u}(t) ) \geq  0 .
\end{equation*}
The same procedure as $h \rightarrow 0 - $ implies 
\begin{equation*}
||\dot{u}(t) ||^2 + ( \nabla \E(u(t)) , \dot{u}(t) ) \leq  0 .
\end{equation*}
We obtain together with \eqref{eq:decay}
\begin{equation*}
\frac{d}{dt} \E(u(t)) = - ||\dot{u}(t) ||^2 \quad a.e. 
\end{equation*}
This however does not prove that $\E \circ u$ is nonincreasing, since the fundamental theorem of calculus holds true iff $\E \circ u$ is locally absolutely continuous. To show this, we apply \cite[Remark 1.1.3]{Ambrosio} and verify \cite[Definition 1.1.1]{Ambrosio}. Fix $T> 0$.  Since $u$ is continuous, the set $u( [0,T] ) \subset H$ is compact. Since $\nabla \E$ is locally Lipschitz, it is Lipschitz on $u([0,T])$. Let $L_T > 0 $ denote the Lipschitz constant. 
Now for $t_1 , t_2 \in [0,T]$
\begin{equation}
||\E(u(t_1)) - \E(u(t_2))|| \leq L_T ||u(t_1) - u(t_2) || \leq L_T \int_{t_1}^{t_2} || \dot{u} (s)|| \; \mathrm{d}s
\end{equation}
which proves that $\E \circ u \in AC^2(0,T)$ in the sense of \cite[Definition 1.1.1]{Ambrosio}. Hence the monotonicity. We can use the fundamental theorem of calculus to obtain
\begin{equation}
\int_0^T ||\dot{u}(s) ||^2 \; \mathrm{d}s =  - \int_0^T \frac{d}{ds}\E(u(s)) \; \mathrm{d}s = \E(u_0) - \E(u(T)) \leq \E(u_0) - \alpha . 
\end{equation}
Letting $T \rightarrow \infty $ and using the monotone convergence theorem proves that $\dot{u}  \in L^2(0, \infty)$. The claim follows.
\end{proof}
\begin{cor}[Control of the Time Derivative]\label{cor:boundder}

Assume Assumption \ref{ass:main}. Let $(u(t))_{ t\geq 0 } $ be an Obstacle Gradient Flow. Then $||\dot{u}(t) || \leq || \nabla\E(u(t))|| $ for almost every $t> 0 $.
\end{cor}
\begin{proof}
Adopting the notation from the previous proposition and using \eqref{eq:decay} we obtain for a.e. $t \in (0,\infty)$ 
\begin{equation*}
|| \dot{u}(t) ||^2 = - ( \nabla \E(u(t)), \dot{u}(t) ) \leq || \nabla \E(u(t) ) || \;  || \dot{u}(t)|| . \qedhere
\end{equation*}
\end{proof}
\begin{prop}[Global Hölder Continuity and Growth]\label{prop:hoel}
Let $t_1,t_2 \geq 0 $ and assume Assumption \ref{ass:main}. Then 
\begin{equation*}
||u(t_1) - u(t_2) || \leq  \sqrt{|t_1-t_2|}\sqrt{\E(u_0)- \alpha}.
\end{equation*}
In particular, for each $t \geq 0 $ one has
\begin{equation}\label{eq:2.16}
||u(t) || \leq ||u_0|| + \sqrt{t}\sqrt{\E(u_0) - \alpha }.
\end{equation}
\end{prop}
\begin{proof}
From \eqref{eq:abl} follows immediately that for $ t_2 \geq t_1 \geq 0 $ 
\begin{align*}
||u(t_1) - u(t_2)|| & \leq \int_{t_1}^{t_2} ||\dot{u}(s)|| \; \mathrm{d}s  \leq \sqrt{|t_2 - t_1| } \sqrt{\int_{t_1}^{t_2} ||\dot{u}(s)||^2 \; \mathrm{d}s  } \\ & \leq \sqrt{|t_1 - t_2 |} \sqrt{\E(t_1) - \E(t_2)} 
\end{align*}
The rest of the claim follows easily choosing $t_1 = 0$ and applying the triangle inequality.
\end{proof}
\begin{prop}[Uniqueness and Continuous Dependence]\label{prop:Wellposed} 

Let $u_0 ,v_0 \in C$ be arbitrary and assume Assumptions \ref{ass:main} and  \ref{ass:growth}. Suppose $(u(t))_{t \geq 0 }$, $(v(t))_{t \geq 0} $ are Obstacle Gradient Flows with initial data $u_0$,$v_0$, respectively. Then 
\begin{equation*}
||u(t) - v(t) || \leq ||u_0-v_0|| e^{ t \zeta  ( ||u_0|| + ||v_0|| + 2 \sqrt{\max\{\mathcal{E}(u_0), \E(v_0)\} - \alpha} \sqrt{t})}.
\end{equation*}
In particular, if $u_0 = v_0 $ then $u \equiv v$. 
\end{prop}
\begin{proof}
Let $t > 0 $ be such that $u$ and $v$ are differentiable at $t$ and $(FVI)$ holds. Using $(FVI)$ and \eqref{eq:2.16} we find
\begin{align*}
\frac{d}{dt}||u(t) - v(t)||^2  & = 2(u(t)- v(t) , \dot{u}(t) - \dot{v}(t) )
\\ & = -2 ( \dot{u}(t) , v(t) - u(t) ) - 2 ( \dot{v}(t) , u(t) - v(t) ) 
\\ & \leq 2( \nabla\E(u(t)) , v(t) - u(t) ) + 2(\nabla\E(v(t)), u(t) - v(t)) 
\\ & \leq 2( \nabla \E(v(t))- \nabla \E(u(t)) , u(t) - v(t) )  
\\ &   \leq 2 ||\nabla \E(u(t))- \nabla \E(v(t)) || \; ||u(t) - v(t)||
\\ & \leq 2 \zeta(||u(t)||+ ||v(t)||) ||u(t) - v(t) ||^2
\\ &  \leq 2 \zeta( ||u_0|| + ||v_0|| + 2  \sqrt{\max\{\E(u_0),\E(v_0)\}- \alpha} \sqrt{t} )  ||u(t) - v(t)||^2. 
\end{align*} 
Since this holds true for almost every $t> 0 $, we can integrate over $t$ and find 
\begin{align*}
& ||u(t) - v(t) ||^2  \leq ||u_0 - v_0||^2 \\ &  \qquad   + \int_0^t  2 \zeta ( ||u_0|| + ||v_0|| + 2  \sqrt{\max\{\E(u_0),\E(v_0)\}- \alpha} \sqrt{s} )  ||u(s) - v(s)||^2 \; \mathrm{d}s.
\end{align*}
The Gronwall Lemma implies the desired estimate. 
\end{proof}
\begin{prop}[Approximation of Flows]\label{prop:limitflow}

Assume Assumptions \ref{ass:main} and \ref{ass:growth}. Assume further that $(u_0^m)_{m \in \mathbb{N}} \subset C$ is a $H$-convergent sequence and $u_0 \in C$ is such that $u_0^m \rightarrow u_0$ in $H$ as $m \rightarrow \infty$. Assume that for each $m \in \mathbb{N}$, $(u_m(t))_{t \geq 0 }$ is an Obstacle Gradient Flow with initial value $u_0^m$. Then there exists an Obstacle Gradient Flow $(u(t))_{t \geq 0 }$ with initial value $u_0$ and $u_m \rightarrow  u$ in $C([0,T],H)$  for all $T> 0$.
\end{prop} 
\begin{proof}
We show first that for each $T> 0$, $(u_m)_{m \in \mathbb{N}}$ defines a Cauchy sequence in $C([0,T],H)$. Indeed, Proposition \ref{prop:Wellposed} implies that 
\begin{align*}
\sup_ {0 \leq t \leq T} ||u_m(t) - u_{m'} (t)|| & \leq ||u_0^m - u_0^{m'} || e^{T \zeta(||u_0^m||+ ||u_0^{m'}|| + \sqrt{\max(\E(u_0^m) , \E(u_0^{m'}) -  \alpha }\sqrt{T}) }
\\ & \leq ||u_0^m - u_0^{m'} || e^{T \zeta(  \sup_{\mu \in\mathbb{N}} 2||u_0^\mu|| + \sqrt{\E(u_0^\mu)- \alpha} \sqrt{T} ) } 
\end{align*}
Therefore there exists $u \in C([0,\infty), H)$ such that for each $T>0$, $u$ is the limit of $(u_m)_{m \in \mathbb{N}}$ in $C([0,T],H)$. Since $C$ is closed, we obtain $u (t) \in C$ for each $t> 0$. Since uniform convergence implies pointwise convergence we conclude $u(0) = u_0$.  It remains to show that $u \in W^{1,2}_{loc}((0, \infty),H)$ and $(FVI)$ holds. For this we use \cite[Theorem 2.2]{Marcel}. Fix $T> 0 $ and $h > 0 $ . Using \eqref{eq:abl} we get
\begin{align*}
\int_0^T \frac{||u(t+h) -u(t) ||^2}{h^2} \; \mathrm{d}t   & = \lim_{m \rightarrow \infty} \int_0^T \frac{||u_m(t+h) - u_m(t)||^2 }{h^2} \; \mathrm{d}t \\
 & = \lim_{m \rightarrow \infty} \int_0^T \frac{1}{h^2} \left\Vert \int_0^1 \dot{u}_m(t+rh) h  \; \mathrm{d}r \right\Vert^2  \; \mathrm{d}t 
 \\ &  \leq \liminf_{m \rightarrow \infty} \int_0^T \int_0^1 || \dot{u}_m (t+rh) ||^2 \; \mathrm{d}r \; \mathrm{d}t
 \\ & \leq \liminf_{m \rightarrow \infty} \int_0^1 \int_0^{T+h} || \dot{u}_m(s)||^2 \; \mathrm{d}s \; \mathrm{d}r
\\ & = \liminf_{m \rightarrow \infty} \int_0^{T+h} - \frac{d}{ds} \E(u_m(s)) \; \mathrm{d}s 
\\ & = \liminf_{m \rightarrow \infty} \E(u_0^m)  - \E(u_m(T+h)) 
  \leq \sup_{e \in \mathbb{N}} \E(u_0^e) - \alpha  .
\end{align*}
Since the bound is independent of $h$ we obtain with \cite[Theorem 2.2]{Marcel} that $u \in W^{1,2}((0,T),H)$. This proves the regularity assertion. For the variational inequality let $t > 0$ be such that $u$ is differentiable at $t$. 
\begin{align*}
(\dot{u}(t) , v- u(t) ) & = \lim_{h \rightarrow 0+ } \frac{1}{h} (u(t+h)- u(t) , v - u(t) )
\\ & =  \lim_{h \rightarrow 0+ } \lim_{m \rightarrow \infty}  \frac{1}{h} (u_m(t+h)- u_m(t) , v - u_m(t) )
 \\ & = \lim_{h \rightarrow 0+ } \lim_{m \rightarrow \infty}  \frac{1}{h} \int_t^{t+h} (\dot{u}_m(s) , v -u_m(t) ) \; \mathrm{d}s
 \\ &  = \lim_{h \rightarrow 0+ }  \lim_{m \rightarrow \infty} \left( \frac{1}{h} \int_t^{t+h} (\dot{u}_m(s) , v -u_m(s) ) \; \mathrm{d}s \right. \\ & \quad \quad \qquad \qquad + \left. \frac{1}{h} \int_t^{t+h} (\dot{u}_m(s) , u_m(s) - u_m(t) )\; \mathrm{d}s \right) . 
\end{align*}  
Observe that since $h> 0 $
\begin{align*}
\left\vert \frac{1}{h} \int_t^{t+h} (\dot{u}_m(s) , u_m(s) - u_m(t) )\; \mathrm{d}s  \right\vert & \leq \frac{1}{h} \int_t^{t+h} ||\dot{u}_m(s) || \; ||u_m(s) - u_m(t) || \; \mathrm{d}s 
\\ & \leq \frac{1}{h} \int_t^{t+h} \int_t^s ||\dot{u}_m(s) ||\;  ||\dot{u}_m(r)|| \; \mathrm{d}r  \; \mathrm{d}s
\\ & \leq \frac{1}{h}  \left( \int_t^{t+h} ||\dot{u}_m(s)|| \; \mathrm{d}s \right)^2  \\ & \leq \int_t^{t+h} || \dot{u_m}(s) ||^2 \; \mathrm{d}s  = \E(u_m(t))-\E(u_m(t+h))  .  
\end{align*}
We obtain that 
\begin{align*}
\lim_{h \rightarrow 0 } \lim_{m \rightarrow \infty} \left\vert \frac{1}{h} \int_t^{t+h} (\dot{u}_m(s) , u_m(s) - u_m(t) )\; \mathrm{d}s  \right\vert & \leq  \liminf_{h \rightarrow 0 } \liminf_{m \rightarrow \infty} \E(u_m(t))-\E(u_m(t+h))\\ &  = \liminf_{h \rightarrow 0 } \E(u(t))-\E(u(t+h)) = 0    ,
\end{align*}
because of continuity of $\E$ and $u$. Therefore 
\begin{align}\label{eq:AQ}
(\dot{u}(t) ,v - u(t) ) &=  \lim_{h \rightarrow 0+ } \lim_{m \rightarrow \infty}  \frac{1}{h} \int_t^{t+h} (\dot{u}_m(s) , v -u_m(s) ) \; \mathrm{d}s  \nonumber \\ & \geq \liminf_{h \rightarrow 0 + } \liminf_{m \rightarrow \infty}  \frac{1}{h} \int_t^{t+h} (-\nabla\E(u_m(s)), v- u_m(s)) \; \mathrm{d}s .
\end{align}
Now recall that $u_m \rightarrow u $ uniformly on $[t,t+h]$ for each $h > 0 $ and $(u_m)_{m \in \mathbb{N}}$ is bounded in $C^\frac{1}{2}([t,t+h],H)$, see Proposition \ref{prop:hoel}. Since $\nabla \E$ is locally Lipschitz, $\nabla \E \circ u_m \rightarrow \nabla \E \circ u $  uniformly in $[t,t+h]$ provided that $h> 0 $ is appropriately small.
Passing to the limit  in \eqref{eq:AQ} we find 
\begin{equation*}
(\dot{u}(t) , v-u(t) ) \geq \liminf_{h \rightarrow 0 } \frac{1}{h} \int_t^{t+h} (- \nabla \E(u(s)), v- u(s) ) \; \mathrm{d}s  =- (\nabla \E(u(t)), v- u(t) ) 
\end{equation*} 
where we used the continuity of the integrand in the last step. We obtain $(FVI)$.
\end{proof}
\begin{cor}[Admissibility and Approximation]\label{cor:closure}
Under Assumptions \ref{ass:main} and \ref{ass:growth}, the set 
\begin{equation*}
\mathcal{G} := \{ u_0 \in C | \; \textrm{there exists an Obstacle Gradient Flow starting at $u_0$} \} 
\end{equation*}
is a closed subset of $C$.
\end{cor}
\begin{proof}
The proof is immediate by Proposition \ref{prop:limitflow}. 
\end{proof}

%
%
\subsection{The HPR-property and generalized Obstacle Problems}

In this article we focus on obstacle problems and not just on flows in some convex closed set. There is usually more structure in the admissible set for obstacle problems. In this section, we aim at understanding the obstacle structure geometrically. This will result in an estimate for the gradient along the flow and can be seen as an investigation of sharpness of the inequality in Corollary \ref{cor:boundder}. 

\begin{definition}[Half Plane Residuum Property]\label{def:halfi} 

Let $H$ be a Hilbert space. A closed, convex set $C$ is said to have the \emph{Half Plane Residuum Property} (for short: HPR-property)  if for each $u \in C$ and $v \in H$ one has $u + \pi_C(v) - v \in C$. Here $\pi_C$ denotes the nearest point projection on $C$ in $H$. 
\end{definition}

\begin{prop}[Justification of the Name]
 
A convex set $C \subset H$ has the HPR-property if and only if for each $w_1,w_2 \in H$
\begin{equation}\label{eq:halfp}
(w_1 - \pi_C(w_1) , w_2 - \pi_C(w_2) ) \geq 0.
\end{equation}
\end{prop} 
\begin{proof}
Suppose first that $C$ has the HPR-property. Fix $w_1,w_2 \in H$. Then, by \cite[Theorem 5.2]{Brezis}, 
\begin{equation*}
(w_1 - \pi_C(w_1) , u - \pi_C(w_1) ) \leq 0 \quad \forall u \in C.
\end{equation*}
Choosing $u:= \pi_C(w_1) + \pi_C(w_2) - w_2$, which lies in $C$ because of the HPR-property, implies \eqref{eq:halfp}.  Now suppose that \eqref{eq:halfp} holds and let $u \in C$ and $v \in H$ be arbitrary. Apply \eqref{eq:halfp} with $w_1 = v$ and $w_2 = u+ \pi_C(v) - v$ to find
\begin{align*}
0 & \leq ( v- \pi_C(v) , u+ \pi_C(v) -  v - \pi_C(w_2))  \\
& = -||v - \pi_C(v)||^2 + (v - \pi_C(v) , u - \pi_C(w_2) ) 
\\ & \leq - || v - \pi_C(v) ||^2 + ||v-\pi_C(v) || \;  || \pi_C(u) - \pi_C(w_2)|| 
\\ & \leq - ||v- \pi_C(v) ||^2  + ||v- \pi_C(v) ||  \; ||u- w_2 || \leq 0 
\end{align*}
since by construction $u - w_2 = v - \pi_C(v)$. Here we used in the last step that $\pi_C$ is Lipschitz continuous with Lipschitz constant $1$, see \cite[Proposition 5.3]{Brezis}. Therefore all inequalities have to be equalities. From equality in the Cauchy-Schwarz estimate one can infer that there is $\lambda > 0$ such that 
$u - \pi_C(w_2)  = \lambda (v- \pi_C(v) )$ .
Again because of equality in the above estimates one has $||v - \pi_C(v)|| = ||u- w_2 || = || \pi_C(u) - \pi_C(w_2)|| = ||u - \pi_C(w_2)||$. We obtain that $\lambda = 1$ and hence $u - \pi_C(w_2) = v -\pi_C(v)$.  
Rearranging we obtain that $u + \pi_C(v) - v  = \pi_C(w_2) \in C$. 
\end{proof}
\begin{prop}[A Gradient Estimate]
Suppose Assumption \ref{ass:main}. Let $u$ be an Obstacle Gradient Flow in $C$. Assume that $C$ satisfies the HPR-property, see Definition \ref{def:halfi}. Then $
\dist(\nabla\E(u(t)),C) \leq \dist(-\dot{u}(t),C)$ for each $t \geq 0$.
\end{prop}
\begin{proof}
Note that $u(t)+ \pi_C(\nabla \E(u(t)) - \nabla \E(u(t)) \in C$ because of the HPR property. Using this as a test function for $(FVI)$ in Definition \ref{def:coneflow} and the inequality in  \cite[Theorem 5.2]{Brezis} we find 
\begin{align*}
\mathrm{dist}(\nabla\E(u(t)), C)^2 & = || \pi_C(\nabla\E(u(t))) - \nabla\E(u(t))||^2  \\
& = (  \pi_C(\nabla\E(u(t))) - \nabla\E(u(t)) ,  \pi_C(\nabla\E(u(t))) - \nabla\E(u(t)) ) 
\\ & =  (  \pi_C(\nabla\E(u(t)))  ,  \pi_C(\nabla\E(u(t))) - \nabla\E(u(t)) ) \\ &  \quad - (  \nabla\E(u(t)) ,  \pi_C(\nabla\E(u(t))) - \nabla\E(u(t)) ) \\
& =   (  \pi_C(\nabla\E(u(t)))  ,  \pi_C(\nabla\E(u(t))) - \nabla\E(u(t)) ) \\ &  \quad - (  \nabla\E(u(t)) , u(t) +  \pi_C(\nabla\E(u(t))) - \nabla\E(u(t)) - u(t) ) \\
 & \leq  (  \pi_C(\nabla\E(u(t)))  ,  \pi_C(\nabla\E(u(t))) - \nabla\E(u(t)) ) \\ &  \quad + (  \dot{u}(t)  , u(t) +  \pi_C(\nabla\E(u(t))) - \nabla\E(u(t)) - u(t) ) \\
 &  = \inf_{\theta \in C} \left[   (  \pi_C(\nabla\E(u(t))) - \theta   ,  \pi_C(\nabla\E(u(t))) - \nabla\E(u(t)) ) \right.  \\ &  \quad  \qquad + \left. (  \dot{u}(t)  + \theta  ,   \pi_C(\nabla\E(u(t))) - \nabla\E(u(t))  ) \right] 
 \\ & \leq \inf_{\theta \in C} (  \dot{u}(t)  + \theta  , \pi_C(\nabla\E(u(t))) - \nabla\E(u(t))  )
 \\ & \leq \inf_{\theta \in C} ||  \dot{u}(t)  + \theta|| \;  || \pi_C(\nabla\E(u(t))) - \nabla\E(u(t)) || 
 \\ & = \mathrm{dist}( \nabla \E(u(t))  , C) \inf_{\theta \in C} ||  \dot{u}(t)  + \theta||. 
\end{align*}
Finally, the claim is shown once one observes
\begin{equation*}
\inf_{\theta \in C } || \dot{u}(t) + \theta || = \inf_{\theta \in C} || (- \dot{u}(t)) - \theta || = \mathrm{dist}(-\dot {u}(t), C). \qedhere
\end{equation*}
\end{proof}

\subsection{The Obstacle Gradient Flow in the Context of Gradient Flows in Metric Spaces}

The following Definition is a standard notion for Gradient Flows in Metric Spaces. The notions appear for example in \cite{Santambrogio} and \cite{Guide}.
\begin{definition}[Metric Slopes and Gradient Flows]
Let $(X,d)$ be a metric space and $\E:X \rightarrow \overline{\mathbb{R}}$ be a map. Let $u: [0,\infty) \rightarrow X$ be a curve in $X$ that is absolutely continuous in the sense of \cite[Definition 1.1.1]{Ambrosio}. Then we call for each $t \geq 0$ 
\begin{equation}\label{eq:metder}
|u'|(t) := \limsup_{ s \rightarrow t } \frac{d(u(s),u(t))}{|s-t|}
\end{equation}
the \emph{metric derivative} of $u$ at $t$. Moreover we define for $w \in X$ 
\begin{equation}\label{eq:2.75}
|\partial \E|(w) := \limsup_{v \rightarrow w, v \neq w} \frac{(\E(w) -\E(v))^+}{d(w,v)}
\end{equation}
the \emph{metric slope} of $\E$ at $w\in X$, where the $\limsup$ is taken with respect to convergence in $X$. The curve $u$ is called $EDE$-Gradient Flow for $\E$ if for each $T> 0$ 
\begin{equation*}
\E(u(T)) + \frac{1}{2}\int_0^T |u'|^2(s) \; \mathrm{d}s + \frac{1}{2} \int_0^T |\partial \E|^2(u(s)) \; \mathrm{d}s = \E(u(0)). 
\end{equation*} 
Here $EDE$ stands for \textit{`Energy Dissipation Equation'}. 
\end{definition}
\begin{remark}
The defining equation for an $EDE$-gradient flow is inspired by an equation that symbolizes the steepest possible energy descent in the setting of a smooth gradient flow in a Hilbert Space, for details see \cite{Santambrogio}.
\end{remark}
\begin{prop}[Metric Quantities for Convex Subsets of $H$] \label{prop:quantity}  Suppose $C \subset H$ is convex and closed. Set $X := C$ and let $d$ be the distance on $C$ that is induced by the norm on $H$. Then the following assertions hold true.
\begin{enumerate}
\item If $u: [0, \infty) \rightarrow X$ is absolutely continuous on $X$ then it is almost everywhere differentiable and 
\begin{equation*}
|u'|(t) = ||\dot{u}(t)||, \quad a.e. \; \;  t \geq 0 .
\end{equation*}
\item If $\E \in C^{1,1}_{loc}(H, \mathbb{R})$ then for each $w \in H$ we have
\begin{equation}\label{eq:dingenskircheni}
|\partial \E|(w) = \sup_{v \in C, v \neq w} \frac{(-\nabla\E(w), v-w)^+}{||v-w||}.
\end{equation}
\end{enumerate} 
\end{prop}
\begin{proof}
We show $(1)$ first. If $u$ is absolutely continuous on $C$ then it is also absolutely continuous on $H$. By \cite[Theorem 1.4.35]{Cazenave}, it also lies in $W^{1,1}_{loc}(0,\infty)$ and is almost everywhere differentiable. Using  \eqref{eq:metder} we find that for each point $t$ of differentiability of $u$ 
\begin{equation*}
|u'|(t) = \limsup_{s \rightarrow t } \frac{d(u(t),u(s))}{|t-s|} = \limsup_{s \rightarrow t} \left\Vert \frac{u(t)- u(s)}{t-s} \right\Vert  = ||\dot{u}(t)||. 
\end{equation*}
Claim $(1)$ follows. For Claim $(2)$ let $\E$ be as in the statement and fix $w \in C$. For $'\leq'$ in \eqref{eq:dingenskircheni} observe that 
\begin{align*}
|\partial \E| (w) & = \limsup_{ v\rightarrow w, v \in C } \frac{(\E(w)- \E(v))^+}{||w-v||}
\\ & = \limsup_{v \rightarrow w, v \in C} \frac{1}{||v-w||} \left( \int_0^1 \frac{d}{dt} \E(v+ t(w-v)) \; \mathrm{d}t \right)^+
\\ & = \limsup_{ v \rightarrow w, v \in C } \frac{1}{||w-v||}\left( \int_0^1 \left( \nabla \E(v+ t(w-v)), w-v \right) \; \mathrm{d}t \right)^+ 
\\ & \leq \limsup_{ v \rightarrow w, v \in C } \left[ \frac{1}{||w-v||} \left( \int_0^1 \left( \nabla \E(v+ t(w-v)) - \nabla \E(w),w - v\right) \; \mathrm{d}t \right)^+ \right. \\ &  \qquad \qquad \qquad  \left. + \frac{1}{||w-v||} (\nabla \E(w), w-v)^+ \right]
\\ & \leq \limsup_{ v \rightarrow w, v \in C }\sup_{t \in [0,1]} || \nabla \E(v+t(w-v))- \nabla \E(w)|| \;  \\ & \qquad \qquad \quad + \sup_{ v \neq w, v \in C } \frac{(-\nabla \E(w) ,v - w)^+}{||v-w||}
\\ & =  \sup_{ v \neq w,  v \in C } \frac{(-\nabla \E(w) ,v - w)^+}{||v-w||},
\end{align*}
where the last equality follows from the continuity of $\nabla \E$. Now we show $'\geq'$. Fix $v \in C$ such that $v \neq w$. Then by \eqref{eq:2.75} 
\begin{align*}
|\partial \E|(w) & \geq \limsup_{ t \rightarrow 0 + } \frac{(\E(w)- \E(tv+ (1-t)w))^+}{||w-(tv+ (1-t)w)||}
\\ & \geq  \limsup_{ t \rightarrow 0 + } \frac{1}{||v-w||} \left( \frac{\E(w)- \E(tv+ (1-t)w)}{t} \right)^+
\\ & = \frac{1}{||v-w||} \left( - \frac{d}{dt}_{\mid_{t=0}} \E(tv+ (1-t)w) \right)^+ 
\\ & = \frac{1}{||v-w||} (- \nabla \E(w), v-w )^+ .
\end{align*}
Taking the supremum over all $v \in C$ such that $v \neq w$, the claim follows. 
\end{proof}
\begin{prop}[$EDE$-Property of the Obstacle Gradient Flow]
Let $X= C$ be a convex subset of $H$ and $(u(t))_{t \geq 0 }$ be an Obstacle Gradient Flow in $C$. Then $u: [0, \infty) \rightarrow X$ is an $EDE$ gradient flow in $X$. 
\end{prop}
\begin{proof}
First, we show that 
\begin{equation*}
|\partial \E|(u(t)) \leq ||\dot{u}(t)|| \quad a.e. \; t \geq 0.
\end{equation*}
For this observe, using Proposition \ref{prop:quantity} and $(FVI)$ that 
\begin{align}\label{eq:opttr}
|\partial \E|(u(t)) & = \sup_{ v \neq u(t) } \frac{\big(-( \nabla \E(u(t)) , v-u(t) )\big)^+}{||v-u(t)||} \nonumber
\\ & \leq \sup_{v \neq u(t)} \frac{(\dot{u}(t), v- u(t))^+}{||v-u(t)||} \nonumber
\\ & \leq \sup_{v \neq u(t) } \frac{|( \dot{u}(t) , v- u(t) )|}{||v- u(t)||} \leq ||\dot{u}(t)|| .
\end{align}
Next, we show that 
\begin{equation}\label{eq:2.92}
|\partial \E|(u(t)) \geq ||\dot{u}(t)|| \quad a.e. \;  t \geq 0 . 
\end{equation}
Let $t$ be a point of Frechét differentiability of $u$. 
If $\dot{u}(t) = 0$ then the statement is trivial. If not then 
\begin{equation*}
0 < ||\dot{u}(t) || = \lim_{s \rightarrow t } \frac{||u(t)- u(s)||}{|t-s|}  
\end{equation*}
and therefore there exists $r > 0 $ such that $u(s) \neq u(t)$ for each $s \in (t-r,t+r) \setminus \{t\}$. Using Proposition \ref{prop:enrggdisp} and the fact that $\E \circ u $ is decreasing we find 
\begin{align*}
||\dot{u}(t)||^2 &  =- \frac{d}{dt} \E(u(t)) = \limsup_{s \rightarrow t+ } \frac{\E(u(t))- \E(u(s))}{s-t } =  \limsup_{s \rightarrow t+} \frac{(\E(u(t))- \E(u(s)))^+}{|t-s|} 
\\ & = \limsup_{s \rightarrow t } \frac{(\E(u(t))- \E(u(s)))^+}{||u(t)-u(s)||}\frac{||u(t)- u(s)||}{|t-s|}
\\ & \leq \limsup_{ w \rightarrow u(t), w \in C } \frac{(\E(u(t))- \E(w))^+}{||u(t)- w||} ||\dot{u}(t)||.
\end{align*} 
Equation \eqref{eq:2.92} follows immediately. 
Now we fix $T> 0$ and compute using \eqref{eq:opttr} and Propositon \ref{prop:enrggdisp}
\begin{align*}
 & \E(u(T)) + \frac{1}{2}\int_0^T |u'|(s)^2 \; \mathrm{d}s + \frac{1}{2} \int_0^T |\partial \E|(u(s))^2 \; \mathrm{d}s \\ & = \E(u(T)) +   \frac{1}{2}\int_0^T |u'|(s)^2 \; \mathrm{d}s + \frac{1}{2} \int_0^T ||\dot{u}(s)||^2 \; \mathrm{d}s 
 \\ & = \E(u(T)) + \int_0^T ||\dot{u}(s)||^2 \; \mathrm{d}s 
 \\ & = \E(u(T)) + \int_0^T - \frac{d}{ds}\E(u(s)) \; \mathrm{d}s = \E(u(0)). \qedhere
\end{align*}  
\end{proof}
\section{Existence}\label{sec:existence}
We will show the existence of Obstacle Gradient Flows, provided that the initial value satisfies a condition which we will formulate and motivate in this section. Throughout this section we assume Assumption \ref{ass:main}. Assumption \ref{ass:growth} will not be needed for the existence proof.

\subsection{Construction of the Flow Trajectory} 
 
\begin{prop}[Existence of Minimizing Movement Sequences, Proof in Appendix \ref{sec:App}]\label{prop:standard}
Let $v \in H$. Define for fixed $\tau > 0 $ 
\begin{equation*}
 \Phi_{v,\tau}: H \rightarrow \mathbb{R}, \; \;  u \mapsto \frac{||u-v||^2}{2\tau}+ \mathcal{E}(u).  
 \end{equation*}
Then there exists $w = w(v, \tau ,\E)  \in C$ such that 
\begin{equation}\label{eq:3.2}
\Phi_{v, \tau} (w) = \inf_{u \in C} \Phi_{v,\tau} (u) .
\end{equation} 
Moreover, each $w \in C$ that is a solution of \eqref{eq:3.2} satisfies
 \begin{equation}\label{eq:eulerlag}
 \frac{1}{\tau} (w-v, u-w) + (\nabla_H \mathcal{E}(w), u - w) \geq 0 \quad \forall u \in C. 
\end{equation}  
\end{prop}

\begin{algo}\label{algo:minimov} (Minimizing Movement Scheme)

Input: $u_0 \in C$, $\tau > 0 $. 

Output: A sequence $(u_{k \tau})_{k \in \mathbb{N}} $. 
\vspace{0.2cm}

\hspace{2.3cm} Set $ \quad u_{0\tau} := u_0$.

\hspace{2.3cm} For $ \quad k \in \mathbb{N}: $
\begin{equation*}
 \textrm{Choose} \quad  u_{k \tau} \in \arg \min_{u \in C} \Phi_{u_{(k-1)\tau} , \tau }. 
\end{equation*}
\end{algo}
\begin{remark}
It may happen that minimizers of the variational problems are not unique. In this case, the algorithm fixes some choice of minimizers. Whenever we refer to a \emph{minimizing movement sequence} $(u_{k \tau})_{k \in \mathbb{N}} $ from now on, we mean an arbitrary but fixed choice of minimizers unless otherwise specified.
\end{remark}
\begin{remark}
To keep notation consistent later when we consider different time stepwidths $\tau$, one should actually write $u_{k, \tau}$ instead of $u_{k\tau}$. We will nevertheless use the other notation for the sake of simplicity. 
\end{remark}
\begin{remark}(Discrete Energy Descent) \label{rem:energdec}

Note that $\mathcal{E}(u_{k\tau}) \leq \mathcal{E}(u_{(k-1)\tau})$. More exactly, 
\begin{equation}\label{eq:energdec}
\frac{||u_{k\tau} - u_{(k-1)\tau}||^2}{2\tau} \leq   \E(u_{(k-1)\tau} ) - \E(u_{k\tau}) ,
\end{equation} 
which is easy to see rearranging the inequality $\Phi_{(k-1) \tau} (u_{k\tau}) \leq \Phi_{(k-1) \tau} (u_{(k-1)\tau} )$.
\end{remark}

\begin{definition}[Interpolations]
Let $u_0 \in C, \tau > 0 $ and $(u_{k\tau})_{k \in \mathbb{N}} \subset C$ be a minimizing movement sequence generated by Algorithm \ref{algo:minimov}. Then we define the \emph{energy interpolation} of $(u_{k\tau})_{k \in \mathbb{N}}$ to be 
\begin{equation*}
\mathcal{E}_\tau(t) := \sum_{k = 1}^\infty \chi_{[(k-1) \tau , k \tau)}(t) \mathcal{E}(u_{(k-1) \tau}) \quad (t \geq 0 ),
\end{equation*}
where $\chi_E$ denotes the characteristic function of a set $E \subset \mathbb{R}$.
Moreover, we define 
\begin{align}\label{eq:lininter}
\overline{u}^\tau(t) & := \sum_{k = 1}^\infty \chi_{[(k-1) \tau , k \tau)}(t) u_{(k-1) \tau} \nonumber ,  \\
u^\tau(t) & := \sum_{k = 1}^\infty \chi_{[(k-1) \tau , k \tau)}(t)\left(  \left( 1 -   \frac{t- (k-1) \tau}{\tau} \right) u_{(k-1) \tau} + \frac{t- (k-1) \tau}{\tau } u_{k\tau} \right).  
\end{align}
\end{definition}

\begin{definition}[Preconditioned Initial Values] \label{def:precon}

We call $u_0 \in C$ a preconditioned initial value, if there exists a family of minimizing movement sequences $(u_{k\tau})_{k \in \mathbb{N}, \tau \in (0,1)}$ starting at $u_0$ such that for each $T>0$ the set $\{u_{k \tau} | \; k \in \mathbb{N}, \tau \in ( 0, 1) , k  \tau \leq T \} $  is precompact in $H$. We denote by $\mathcal{P}$ the set of preconditioned initial values.
\end{definition}

\begin{remark}
Of course it is possible that $\mathcal{P}= \emptyset$. We will analyze $\mathcal{P}$ for some special problems in Section \ref{Sec:Navier}. A more general statement about $\mathcal{P}$ would be very desirable, but is at this point out of reach for this article. 
\end{remark}

 \begin{prop}[Uniform Hölder Bound] \label{prop:bddc1/2} Let $t,s \geq  0 $. Then for each $\tau \in (0,1)$, 
 \begin{equation}\label{eq:hoelder}
 ||u^\tau(t) - u^\tau(s) || \leq 3 \sqrt{2} (\mathcal{E}(u_0) - \alpha)^\frac{1}{2} \sqrt{|t-s|},
 \end{equation}
 where  $u^\tau$ is defined as in \eqref{eq:lininter}. 
\end{prop}
\begin{proof}
First assume that $t,s \in [(k-1)\tau,k \tau] $. Then $|t-s| < \tau $ and hence using \eqref{eq:energdec}
\begin{align}\label{eq:locholder}
||u^\tau(t) - u^\tau(s) || & = \frac{|t-s|}{\tau} ||u_{(k-1) \tau } - u_{k \tau } || 
\nonumber \\ & \leq \frac{|t-s|}{\tau} \sqrt{2\tau } \sqrt{\mathcal{E}(u_{(k-1) \tau }) - \mathcal{E}(u_{k\tau})} \nonumber \\ &  \leq \sqrt{2 }\frac{|t-s|}{\sqrt{\tau} }\sqrt{ \mathcal{E}(u_0) - \alpha } \leq \sqrt{2} \frac{|t-s|}{\sqrt{|t-s|}} \sqrt{\mathcal{E}(u_0) - \alpha } \nonumber \\ & = \sqrt{2}\sqrt{|t-s|} \sqrt{\mathcal{E}(u_0) - \alpha}.  
\end{align} 
Now assume that there is $k,l \in \mathbb{N}$ such that $ (l-1) \tau \leq t \leq l \tau \leq (k-1) \tau \leq s \leq k \tau $. Using \eqref{eq:locholder} we obtain 
\begin{align*}
||u^\tau(t) - u^\tau(s) || & \leq ||u^\tau (s) - u^\tau((k-1) \tau) || + || u^\tau((k-1) \tau) - u^\tau(l\tau) || \\ & \; \; \; \;  + ||u^\tau(l \tau) - u^\tau(t) || 
 \nonumber \\ & \leq \sqrt{2} \sqrt{|s-(k-1)\tau |} \sqrt{\mathcal{E}(u_0)- \alpha } + \sqrt{2} \sqrt{|l \tau - t|} \sqrt{\mathcal{E}(u_0)- \alpha}  \\ & \; \; \; \;\qquad  + \sum_{m = l}^{k-2} ||u_{(m + 1) \tau } - u_{m \tau } || 
\nonumber \\ &  \leq 2 \sqrt{2} \sqrt{|t-s|} \sqrt{\mathcal{E}(u_0)- \alpha}  \\ & \; \; \; \; \qquad  + \sqrt{(k-1-l)} \sqrt{\sum_{m=l}^{k-2} ||u_{(m+1) \tau} - u_{m \tau } ||^2 }
\nonumber \\ & \leq 2 \sqrt{2} \sqrt{|t-s|} \sqrt{\mathcal{E}(u_0)- \alpha} \\ &  \; \; \; \;\qquad  + \sqrt{k-1-l} \sqrt{\sum_{ m = l}^{k-2} 2 \tau (  \E(u_{m\tau}))- \mathcal{E}(u_{(m+1)\tau}))   } \nonumber \\ & 
\leq 2 \sqrt{2} \sqrt{|t-s|} \sqrt{\mathcal{E}(u_0)- \alpha}  + \sqrt{2}\sqrt{(k-1) \tau - l \tau } \sqrt{  \mathcal{E}(u_{l \tau} ) - \mathcal{E}(u_{(k-1) \tau })  } \nonumber \\ & \leq 3 \sqrt{2} \sqrt{\mathcal{E}(u_0)- \alpha}\sqrt{|t-s|}. \qedhere
\end{align*}
\end{proof}
\begin{cor}[A Growth Rate]\label{cor:growthest}
For each $T> 0$ the restrictions of the linear interpolations $(u^{\tau}(t))_{ \tau > 0, t \in [0,T] } $ are locally uniformly bounded, more precisely 
\begin{equation}\label{eq:unifest}
\sup_{t \in [0,T], \tau > 0  } || u^{\tau}(t) || \leq ||u_0|| + 3 \sqrt{2} \sqrt{\E(u_0) - \alpha } \sqrt{T}.
\end{equation}
\end{cor}
\begin{proof}
Immediate by the triangle inequality and Proposition \ref{prop:bddc1/2} with $s = 0$. 
\end{proof}

\begin{prop}[Rate of Energy Descent]\label{prop:energdec} 
Let $(u_{k \tau})_{k \in \mathbb{N}, \tau \in (0,1)}$  be a family of sequences generated by Algorithm \ref{algo:minimov}.
Then for each $T> 0$, there exists a constant $C_T= C_T(\E, u_0 , T) > 0 $ such that  for all $\tau \in ( 0 , 1) $ and all $ k \in \mathbb{N}$ with $(k+1) \tau \leq T $ 
\begin{equation*}
\frac{\E(u_{k \tau}) - \E(u_{(k+1) \tau}) }{\tau} \leq C_T.
\end{equation*}
\end{prop}
\begin{proof}
Define 
$
R_T := ||u_0|| + 3\sqrt{2} \sqrt{\E(u_0) - \alpha} \sqrt{T}
$
and compute using \eqref{eq:unifest}
\begin{align*}
\frac{\E(u_{k \tau}) - \E(u_{(k+1) \tau}) }{\tau} &  = \int_0^1 \left(\nabla \E(s u_{k \tau} + (1-s)u_{(k+1) \tau})  , \frac{u_{k \tau} - u_{(k+1) \tau}}{\tau } \right)  \; \mathrm{d}s \\   &  \leq \sup_{w \in B_{R_T}(0)}  ||\nabla \E(w)||  \frac{||u_{k \tau } - u_{(k+1)\tau}||}{\tau } \\ &  \leq \sqrt{2} \sup_{w \in B_{R_T}(0)}  ||\nabla \E(w)|| \sqrt{\frac{\E(u_{k \tau}) - \E(u_{(k+1) \tau}) }{\tau}}. 
\end{align*}
The claim follows choosing $C_T := 2 \sup_{w \in B_{R_T}(0)} || \nabla \E(w)||^2$.
\end{proof}
\begin{prop}[Limit Trajectories]\label{prop:limitra}
Let $u_0$ be a preconditioned initial value. Then there exists a sequence $(\tau_n)_{n \in \mathbb{N}}$ converging to zero and $\phi: [0,\infty) \rightarrow \mathbb{R}$ nonincreasing as well as  $u \in C^\frac{1}{2}([0,\infty),H)$ such that $\E_{\tau_n} \rightarrow \phi $ pointwise and for each $T>0$ : $u^{\tau_n} \rightarrow u $ in $C([0,T],H)$.
\end{prop}

\begin{proof}
First fix $T \in \mathbb{N}$. 
Since $\E_{\tau_n}$ is nonincreasing for each $n  \in \mathbb{N}$ we can conclude from Helly's Theorem, see \cite[Lemma 3.3.3]{Ambrosio}, that there exists a sequence $\tau_n \rightarrow 0$ and $\phi_T :[0, T] \rightarrow \mathbb{R}$ nonincreasing such that $\E_{\tau_n} \rightarrow \phi_T$ pointwise on $[0,T]$. We show the convergence of a subsequence of $u^{\tau_n}$ in $C([0,T],H)$ using an appropriate version of the Arzelá-Ascoli Theorem, namely \cite[4.12, "Warning" on p.106]{Alt}. For this, note that boundedness in $C^\frac{1}{2}([0,T],H)$ implies uniform equicontinuity. It remains to show that for each $t \in [0,T]$ the set  $\{u^{\tau_n}(t)| n \in \mathbb{N} \} $ is precompact. This however is immediate if we write 
\begin{equation*}
u^{\tau_n}(t) = \left(  \left( 1 -   \frac{t- (k_n-1) \tau_n}{\tau_n} \right) u_{(k_n-1) \tau_n} + \frac{t- (k_n-1) \tau_n}{\tau_n} u_{k_n\tau_n} \right)    
\end{equation*}
where $k_n=k_n(t)$ is chosen such that $(k_n-1)\tau_n \leq t < k_n\tau_n $. Since $0 \leq \frac{t- (k_n-1) \tau_n}{\tau_n} \leq 1 $ and  $u_{k_n \tau_n}, u_{(k_n-1) \tau_n}$ lie in a precompact set (as $u_0 \in \mathcal{P}$), one can extract a subsequence such that every summand converges and this gives rise to a convergent subsequence of $u^{\tau_n}(t)$. Therefore the Arzela-Ascoli Theorem applies and we obtain a further subsequence of $(\tau_n)$ which we do not relabel, such that $u^{\tau_n} \rightarrow u_T$ uniformly in $[0,T]$ for some $u_T \in C([0,T],H)$. It remains to show that $u_T \in C^\frac{1}{2}([0,T],H),$ but this follows passing to the limit in \eqref{eq:hoelder}. Since $T \in \mathbb{N}$ was arbitrary we can - using a standard diagonal sequence argument - find a subsequence (again called) $\tau_n $ such that for each $T \in \mathbb{N}$, $(u^{\tau_n})_{n \in \mathbb{N}} $ is uniformly convergent to some $u_T$  and $(\E_{\tau_n})_{n \in \mathbb{N}}$ converges pointwise to a nonincreasing function $\phi_T$ on $[0,T]$. Since pointwise limits are unique, we get $(u_{T+1})_{\mid_{[0,T]}} = u_T$ and $(\phi_{T+1})_{\mid_{[0,T]}} = \phi_T$. Therefore setting $u(t) := u_T(t)$ if $t \leq T $ and $\phi(t) := \phi_T(t)$  if $t \leq T$ is well-defined and gives the desired functions. 
\end{proof}
\begin{remark}
Note that $\phi$ is a bounded function since for each $t \geq 0 $ one has $\alpha \leq \E_{\tau}(t) \leq \E(u_0)$ and the inequality carries over to the pointwise limit $\phi$. Moreover, \eqref{eq:unifest} carries over to the limit, so 
\begin{equation*}
 \sup_{t \in [0,T]} ||u(t)|| \leq ||u_0|| + 3 \sqrt{2} \sqrt{\E(u_0) - \alpha} \sqrt{T} \quad \forall T > 0 . 
\end{equation*}
\end{remark}
\begin{cor}[Limit of Constant Interpolations]\label{cor:const}

Let $u \in C^\frac{1}{2}([0,\infty), H)$ be a limit trajectory and $(\tau_n)_{n \in \mathbb{N}}$ be chosen as in Proposition \ref{prop:limitra}. Then the constant interpolations $(\overline{u}^{\tau_n})_{n \in \mathbb{N}}$ also converge to $u$ pointwise in $[0,\infty)$ and uniformly on $[0,T]$ for each $T> 0$. 
\end{cor}
\begin{proof}
Fix $T > 1$. Assume without loss of generality that $\tau_n \leq 1$ for each $n \in \mathbb{N}$. For $t \in [0,T-1] $, let $k_n(t)$ be chosen such that $(k_n(t)- 1)\tau_n \leq t \leq k_n(t)\tau_n$.  Then by \eqref{eq:lininter} and \eqref{eq:energdec} 
\begin{align*}
||\overline{u}^{\tau_n}(t) - u(t) ||  & \leq || \overline{u}^{\tau_n}(t) - u^{\tau_n}(t) || + ||u^{\tau_n}(t) - u(t) || 
\\ & \leq \frac{t- (k_n(t)-1) \tau_n}{\tau_n}  ||u_{(k_n(t) - 1) \tau_n} - u_{k_n(t) \tau_n} || + ||u^{\tau_n} - u||_{C([0,T], H) }
\\ & \leq ||u_{(k_n(t) - 1) \tau_n} - u_{k_n(t) \tau_n} || + ||u^{\tau_n} - u||_{C([0,T], H) } 
\\ & \leq \sqrt{2 \tau_n} \sqrt{\mathcal{E}(u_0) - \alpha } +  ||u^{\tau_n} - u||_{C([0,T], H) }  \rightarrow 0 \quad (n \rightarrow \infty) .
\end{align*}
Therefore $|| \overline{u}^{\tau_n} - u||_{C([0,T-1],H)} \rightarrow 0$. Since $T>1$ was arbitrary, the claim follows. 
\end{proof}

\subsection{Regularity of the limit trajectories}

So far, we have shown that the minizing movement schemes approximate a limit trajectory as the stepwidth becomes small. However, it is unclear whether this limit is an Obstacle Gradient Flow. Here we examine the regularity of the limit trajectory and discover that it meets the regularity requirements of the flow. The argument will use difference quotient techniques, which are also available for vector-valued functions. For a detailed overview about these methods see \cite{Marcel}. 

\begin{lemma}[Difference Quotients of BV-Functions] \label{lem:bdddiffquot} 

Let $d\geq 1$ and  $\Omega, \widetilde{\Omega} \subset  \mathbb{R}^d$ be open and bounded, $\Omega \subset \subset \widetilde{\Omega}$. Then there exists a constant $C=C(\Omega, \widetilde{\Omega} ) > 0$ such that for each $ h < \frac{1}{4} \mathrm{dist}(\Omega, \widetilde{\Omega}^C) $ and $f \in BV(\widetilde{\Omega}) $ 
\begin{equation*}
||D^h f||_{L^1(\Omega)} \leq C |[Df]|(\widetilde{\Omega}) ,
\end{equation*}
where $D^h_if := \frac{f(x+he_i) - f(x)}{h}$ for $i = 1,...,d$  and $|[Df]|$ denotes the total variation measure. 
\end{lemma}

\begin{proof}
For $\epsilon \in (0, \tfrac{1}{2}\mathrm{dist}(\Omega, \widetilde{\Omega}^C ))$ let  $\phi_\epsilon$ be the standard mollifier supported in $\overline{B_\epsilon(0)}$. Set 
\begin{equation*}
g_i(x) := \frac{f \chi_{\widetilde{\Omega}}(x+ h e_i) - f \chi_{\widetilde{\Omega}}(x) }{h }, \quad i = 1,...,d \quad x \in \mathbb{R}^d.
\end{equation*}
Observe that $g* \phi_\epsilon \rightarrow g $ in $L^1(\mathbb{R}^n)$.  Since $g \equiv D^h f$ on $\Omega$ for $h < \frac{1}{4} \mathrm{dist}(\Omega, \widetilde{\Omega})^C$ we find that 
that $D^h f * \phi_\epsilon \rightarrow D^hf $ in $L^1(\Omega)$ as $\epsilon \rightarrow 0+$. Define $U := \{x \in \widetilde{\Omega} : \mathrm{dist}(x, \widetilde{\Omega}^C) > \frac{1}{2} \mathrm{dist}(\Omega, \widetilde{\Omega}^C) \} $ and observe that using the constant $C(\Omega, U)$ from \cite[Theorem 5.8.3 (i)]{Evans}, \cite[Example 1, Section 5.1]{EvansGariepy} and Fubini's Theorem
\begin{align*}
||D^h f ||_{L^1(\Omega)} &  = \lim_{\epsilon \rightarrow 0 }  ||D^hf * \phi_\epsilon||_{L^1(\Omega)} = \lim_{\epsilon \rightarrow 0 } || D^h(f*\phi_\epsilon)||_{L^1(\Omega)}
\\ & \leq  \liminf_{\epsilon \rightarrow 0 } C(\Omega, U) \int_U |D(f* \phi_\epsilon) | \; \mathrm{d}x
\\ &  = C(\Omega, U)\liminf_{\epsilon \rightarrow 0 } \sup_{ \phi \in C_0^1(U), ||\phi||_\infty \leq 1  } \int_U  f* \phi_\epsilon \divi (\phi) \; \mathrm{d}x 
\\ & = C(\Omega, U)  \liminf_{\epsilon \rightarrow 0 } \sup_{ \phi \in C_0^1(U),  ||\phi||_\infty \leq 1  } \int_U   \int_{\widetilde{\Omega}}  f(y) \phi_\epsilon(x-y) \divi (\phi(x))  \; \mathrm{d}y \; \mathrm{d}x
\\ & = C(\Omega, U)  \liminf_{\epsilon \rightarrow 0 } \sup_{ \phi \in C_0^1(U), ||\phi||_\infty \leq 1  }   \int_{\widetilde{\Omega}}  f(y)   \int_U \phi_\epsilon(x-y) \divi (\phi(x)) \; \mathrm{d}x  \; \mathrm{d}y
\\ & =C(\Omega, U) \liminf_{\epsilon \rightarrow 0 } \sup_{ \phi \in C_0^1(U) ,  ||\phi||_\infty \leq 1 }   \int_{\widetilde{\Omega}}  f(y)   \divi(\phi * \phi_\epsilon)   \; \mathrm{d}y\\ & \leq C(\Omega, U) |[Df]|(\widetilde{\Omega}), 
\end{align*}
since $\phi* \phi_\epsilon \in C_0^1(\widetilde{\Omega})$ for $\epsilon$ small enough. Here we used that $\phi_\epsilon * \mathrm{div}(\phi) = \mathrm{div}( \phi_\epsilon * \phi).$ on $U$.
\end{proof}
\begin{prop}[$L^2$-Bounds for the Difference Quotients] \label{prop:bdondif}
Let $u \in C^\frac{1}{2}([0,\infty), H) $ be as in Proposition \ref{prop:limitra}. Then for each $T > 0 $ there is $C =C(T) > 0 $ such that for $|h| < \frac{1}{4}$
\begin{equation}
||D^hu||^2_{L^2((0,T),H)} \leq C (\mathcal{E}(u_0) - \alpha).
\end{equation}
\end{prop}
\begin{proof} We only show the claim for $0< h <\frac{1}{4}$. The case $h < 0 $ is very similar. Let $(\tau_n)_{n\in \mathbb{N}}$ be chosen as in Proposition \ref{prop:limitra}. For $s \in [0,\infty)$ we define $k_n(s) $ to be the unique natural number such that $(k_n(s) - 1)\tau_n \leq s < k_n(s) \tau_n $. 
Using Fatou's Lemma, Remark \ref{rem:energdec} and the Cauchy-Schwarz inequality we find 
\begin{align*}
 &  \int_0^T \frac{||u(t+h) - u(t)||^2}{h^2} \; \mathrm{d}t  \leq  \liminf_{n \rightarrow \infty} \int_0^T \frac{||\overline{u}^{\tau_n}(t+h) - \overline{u}^{\tau_n} (t)||^2 }{h^2} 
\\ & \qquad  \leq \liminf_{n \rightarrow \infty} \int_0^T \frac{1}{h^2} \left( \sum_{ l = k_n(t) - 1}^{k_n(t+h)- 2} ||u_{(l+1)\tau_n} - u_{l \tau_n} || \right)^2 \; \mathrm{d}t
\\ & \qquad  \leq \liminf_{n \rightarrow \infty} \int_0^T \frac{1}{h^2}\left( \sum_{l= k_n(t)-1}^{k_n(t+h) - 2} \sqrt{2\tau_n} \sqrt{\E(u_{(l+ 1) \tau_n}) - \E(u_{l \tau_n}) } \right)^2 \; \mathrm{d}t  
\\ & \qquad \leq \liminf_{n \rightarrow \infty} \int_0^T \frac{2 \tau_n (  k_n(t+h) - k_n(t)) }{h^2} \left\vert \sum_{l= k_n(t)-1}^{k_n(t+h) - 2} {\E(u_{(l+ 1) \tau_n}) - \E(u_{l \tau_n}) } \right\vert \; \mathrm{d}t
\\ &  \qquad \leq \liminf_{n \rightarrow \infty} \int_0^T \frac{2 (h+\tau_n) }{h^2} ( \E(u_{(k_n(t)-1)\tau_n})-\E(u_{(k_n(t+h)-1)\tau_n}) ) \; \mathrm{d}t, 
\end{align*}
where we used in the last step that 
\begin{equation}\label{eq:discretjump}
\tau_n ( k_n(t+h) - k_n(t)) = (k_n(t+h)-1) \tau_n - (k_n(t)-1) \tau_n \leq t+ h - (t- \tau_n)  = h + \tau_n .
\end{equation}
The integrand converges pointwise to $\frac{\phi(t+h)- \phi(t)}{h} $ as $n \rightarrow \infty$, see Proposition \ref{prop:limitra}, and is dominated by $\frac{h+1 }{h^2} (\E(u_0)- \alpha)$. So Lebesgue's Theorem yields 
\begin{equation*}
\int_0^T \frac{||u(t+h) - u(t)||^2}{h^2}  \leq \int_0^T \frac{\phi(t)- \phi(t+h)}{h} \; \mathrm{d}t. 
\end{equation*}
Now $\phi$ can be extended to a nonincreasing function on $[-1,T+1]$ by a constant. As a nonincreasing and bounded function, the extended version is of bounded variation and $|[D\phi]|((-1, T+1)) = \phi(T+1) - \phi(-1)  \leq \mathcal{E}(u_0) - \alpha$, see \cite[Theorem 1 of Section 5.10]{EvansGariepy}. 
Using Lemma \ref{lem:bdddiffquot} with $\Omega = (0,T)$ and $\widetilde{\Omega} = (-1,T+1) $  we find, since $h < \frac{1}{4}$,
\begin{equation}
 ||D^hu ||_{L^2((0,T),H)} \leq ||D^h\phi||_{L^1(0,T)} \leq C|[D\phi]| ( -1, T+1)   \leq C ( \E(u_0) - \alpha) .
\end{equation}
\end{proof}
\begin{cor}[Sobolev Regularity in Time] \label{cor:regu} Let $u \in C^\frac{1}{2}([0,\infty), H) $ be as in Proposition \ref{prop:limitra} and $T> 0 $. Then $u \in W^{1,2}((0,T),H )$. In particular, $u \in W^{1,2}_{loc}((0, \infty), H)$. 
\end{cor}
\begin{proof}
This follows immediately from \cite[Lemma 2.1]{Marcel} and Proposition \ref{prop:bdondif}.
\end{proof}

\subsection{Conclusion of the Existence Proof}

We have already shown regularity of the limit trajectory of the minimizing movement scheme. What is still missing is the $(FVI)$-property, which we will verify immediately using the variational inequalities for the discrete approximations.

\begin{lemma}[Flow Variational Inequality] \label{lem:fvi} Let $u$ be as in Proposition \ref{prop:limitra}. Then $u$ satisfies $(FVI)$, i.e. for almost every $t> 0 $,
\begin{equation*}
(\dot{u}(t), v- u(t)) + (\nabla \E (u(t)) , v- u(t) ) \geq 0 \quad \forall v \in C.
\end{equation*}
\end{lemma}

\begin{proof} 
Choose $N \subset (0,\infty)  $  to be a set of measure zero such that for each $t \in (0,\infty)\setminus N$ we have $\dot{u}(t) = \lim_{h \rightarrow 0 } \frac{u(t+h) - u(t)}{h} $ and $\phi$, chosen as in Proposition \ref{prop:limitra}, is continuous in $(0, \infty)\setminus N$ (Recall that nonincreasing functions have at most countably many points of discontinuity and therefore $N$ can indeed be chosen to be a null set). Fix $t \in [0,\infty) \setminus N $ and $T> t$. Let $\epsilon > 0 $ be such that $\nabla\mathcal{E}$ is Lipschitz continuous in $B_\epsilon(u(t))$, see Assumption \ref{ass:main}. Denote the Lipschitz constant on $B_\epsilon(u(t))$ by $L$. Let $h_0 > 0$ be such that $3 \sqrt{2} \sqrt{|h_0|} (\mathcal{E}(u_0) - \alpha)^\frac{1}{2} < \frac{\epsilon}{3}$ and $n_0 \in \mathbb{N}$ be such that $|| u^{\tau_n} - u||_{C([0,T+h_0],H)}< \frac{\epsilon}{3}$ for each $n \geq n_0$. Whenever we write $\lim_{h \rightarrow 0+ }$ in the following, we shall actually mean $\lim_{h \rightarrow 0, 0 < h < h_0}$. We compute for $t \in [0,T] \setminus N$

\begin{align}\label{eq:varinn}
(\dot{u}(t) , v- u(t)  )  &=  \lim_{h \rightarrow 0 +} \left( \frac{u(t+h) - u(t) }{h} , v - u(t)  \right) \nonumber \\
 & = \lim_{h \rightarrow 0+ } \lim_{n \rightarrow \infty} \frac{1}{h} \left( \overline{u}^{\tau_n}(t+h) - \overline{u}^{\tau_n}(t) , v - u^{\tau_n}(t)  \right) 
\nonumber  \\ &  = \lim_{h \rightarrow 0 +} \lim_{n \rightarrow \infty} \frac{1}{h} \sum_{l = k_n(t) - 1}^{k_n(t+h)-2} ( u_{(l+1)\tau_n} - u_{l \tau_n} , v- u^{\tau_n}(t) ) 
\nonumber \\ & = \lim_{h \rightarrow 0 +} \lim_{n \rightarrow \infty} \frac{1}{h} \sum_{l = k_n(t) - 1}^{k_n(t+h)-2} ( u_{(l+1)\tau_n} - u_{l \tau_n} , v- u_{(l+1)\tau_n} ) 
\nonumber   \\ & \quad \quad \quad \quad \quad  + \frac{1}{h} \sum_{l = k_n(t) - 1}^{k_n(t+h)-2} (u_{(l+1) \tau_n} - u_{l \tau_n} , u_{(l+1) \tau_n} - u^{\tau_n} (t) ) .
 \end{align}
 We estimate the second summand using Remark \ref{rem:energdec} and that by  \eqref{eq:discretjump} $u_{(l+1)\tau_n} = u^{\tau_n} ((l+1) \tau_n)$ we obtain
 \begin{align*}
 &  \left\vert  \frac{1}{h} \sum_{l = k_n(t) - 1}^{k_n(t+h)-2} (u_{(l+1) \tau_n} - u_{l \tau_n} , u_{(l+1) \tau_n} - u^{\tau_n} (t) ) \right\vert\\ 
 & \leq \frac{1}{h}\sum_{l = k_n(t) - 1}^{k_n(t+h)-2} ||u_{(l+1) \tau_n } - u_{l \tau_n} || \; || u_{(l+1)\tau_n} - u^{\tau_n} (t) || 
 \\ & \leq\frac{1}{h} \sum_{l = k_n(t) - 1}^{k_n(t+h)-2}  \sqrt{2\tau_n}\sqrt{ \E(u_{l\tau_n}) - \E(u_{(l+1) \tau_n})   } ||u^{\tau_n} ((l+1) \tau_n ) - u^{\tau_n} (t) || 
 \end{align*}
 Using the Hölder estimate from Proposition \ref{prop:bddc1/2} and the fact that $|t- (l+1) \tau_n| \leq h $ for each $l \in \{ k_n(t) -1, ... , k_n(t+h) - 2 \}$ we find  
 \begin{align*}
  &  \left\vert  \frac{1}{h} \sum_{l = k_n(t) - 1}^{k_n(t+h)-2} (u_{(l+1) \tau_n} - u_{l \tau_n} , u_{(l+1) \tau_n} - u^{\tau_n} (t) ) \right\vert
 \\ & \leq \frac{3 \sqrt{2} \sqrt{\E(u_0) - \alpha }\sqrt{2\tau_n} }{h}\sum_{l = k_n(t) - 1}^{k_n(t+h)-2} \sqrt{ \E(u_{l\tau_n}) - \E(u_{(l+1) \tau_n})   } \sqrt{|t -(l+1) \tau_n|}
 \\ & \leq \frac{3 \sqrt{2} \sqrt{\E(u_0) - \alpha }\sqrt{2\tau_n} }{\sqrt{h}} \sum_{l = k_n(t) - 1}^{k_n(t+h)-2}\sqrt{ \E(u_{l\tau_n}) - \E(u_{(l+1) \tau_n})   } 
 \\ &  \leq \frac{3 \sqrt{2} \sqrt{\E(u_0) - \alpha }\sqrt{2\tau_n}\sqrt{k_n(t+h)- k_n(t)} }{\sqrt{h}} \left( \sum_{l = k_n(t) - 1}^{k_n(t+h)-2}(\E(u_{l \tau_n} )- \E(u_{(l+1) \tau_n}) )\right)^\frac{1}{2}
 \\ & \leq  6 \sqrt{\E(u_0) - \alpha}\frac{\sqrt{h + \tau_n}}{\sqrt{h}} ( \E(\overline{u}^{\tau_n}(t) ) - \E(\overline{u}^{\tau_n} (t+h) ))^\frac{1}{2} 
 \\ & =   6 \sqrt{\E(u_0) - \alpha}\frac{\sqrt{h + \tau_n}}{\sqrt{h}} ( \E_{\tau_n}(t)  - \E_{\tau_n} (t+h) )^\frac{1}{2} .
 \end{align*}
 Letting $n \rightarrow \infty$ in this expression, we obtain that 
\begin{equation*}
\limsup_{n \rightarrow \infty} \left\vert  \frac{1}{h} \sum_{l = k_n(t) - 1}^{k_n(t+h)-2} (u_{(l+1) \tau_n} - u_{l \tau_n} , u_{(l+1) \tau_n} - u^{\tau_n} (t) ) \right\vert \leq  6 \sqrt{\E(u_0)- \alpha } ( \phi(t) - \phi(t+h) )^\frac{1}{2},
\end{equation*}
which goes to zero as $h \rightarrow 0 +$ since $t$ was chosen to be a point of continuity of $\phi$. Therefore \eqref{eq:varinn} simplifies to
\begin{align*}
(\dot{u}(t) , v- u(t) ) = \lim_{h \rightarrow 0 + } \lim_{n \rightarrow \infty} \frac{1}{h } \sum_{l = k_n(t) - 1}^{k_n(t+h)-2} ( u_{(l+1)\tau_n} - u_{l \tau_n} , v- u_{(l+1) \tau_n} ) . 
\end{align*}
Using the variational inequalities for the minimizing movement elements, see Proposition \ref{prop:standard}, we can compute 
\begin{align*}
(\dot{u}(t) , v- u(t) )  & \geq \liminf_{h \rightarrow 0 + } \liminf_{n \rightarrow \infty} \frac{1}{h } \sum_{l = k_n(t) - 1}^{k_n(t+h)-2}  -\tau_n ( \nabla \E(u_{(l+1) \tau_n}) , v-u_{(l+1)\tau_n} ) \\ & 
\geq  \liminf_{h \rightarrow 0 + } \liminf_{n \rightarrow \infty} \frac{1}{h } \sum_{l = k_n(t) - 1}^{k_n(t+h)-2} -  \tau_n ( \nabla \E(u((l+1) \tau_n) , v-u((l+1) \tau_n) )  - J_{n,h}      
\end{align*}
where 
\begin{equation*}
J_{n,h} := \frac{1}{h} \sum_{l = k_n(t) - 1}^{k_n(t+h)-2} \tau_n [( \nabla \E(u_{(l+1) \tau_n}) , v-u_{(l+1)\tau_n} ) -  ( \nabla \E(u((l+1) \tau_n)) , v-u((l+1) \tau_n) ))]  .
\end{equation*}
Further, define  $I_{n,h} := [k_n(t) - 1, k_n(t+h) - 2]\cap \mathbb{N}$ and
\begin{equation*}
S_{n, h} := \sup_{l \in I_{n,h}}  |( \nabla \E(u_{(l+1) \tau_n}) , v-u_{(l+1)\tau_n} ) -  ( \nabla \E(u((l+1) \tau_n)) , v-u((l+1) \tau_n) )|.
\end{equation*}
According to \eqref{eq:discretjump}
\begin{equation*}
|J_{n,h} | \leq \frac{1}{h} \tau_n (k_n(t+h)- k_n(t) ) S_{n,h}  \leq \frac{h+\tau_n}{h} S_{n,h} .
\end{equation*}
Using that by \eqref{eq:lininter} $u_{(l+1) \tau_n } = u^{\tau_n} ( (l+1) \tau_n)$ we obtain with the Lipschitz estimate (see Assumption \ref{ass:main} and definition of $L$ at the beginning of the proof)
\begin{align}\label{eq:Snh}
S_{n,h} \leq \sup_{l \in I_{n,h} } & |(\nabla \E(u^{\tau_n} ( (l +1) \tau_n ) - \nabla \E(u((l+1) \tau_n ) ), v- u^{\tau_n} ( (l+1) \tau_n ) )|  \nonumber \\ & + |\nabla \E( (u((l+1)\tau_n) ) , u((l+1) \tau_n) - u^{\tau_n} ((l+1)\tau_n)) |
 \nonumber \\ \leq   \sup_{l \in I_{n,h} } &  (  L  (||v|| + 3 \sqrt{2} \sqrt{\E(u_0) - \alpha} \sqrt{(l+1) \tau_n })+ ||\nabla \E(u((l+1) \tau_n) )||)\nonumber \\  & \quad \quad \quad \quad \quad \cdot  || u^{\tau_n} ( (l+1) \tau_n ) - u((l+1) \tau_n )  || .
\end{align}
We need for the last step, in order to use the Lipschitz continuity from Assumption \ref{ass:main}, that for each $l \in I_{n,h}$,  $u^{\tau_n}((l+1) \tau_n) , u((l+1)\tau_n ) \in B_\epsilon(u(t))$. This is ensured by $h < h_0$ and the choice of $h_0$ in the beginning of the proof. Indeed,
\begin{align}\label{eq:esti1}
||u((l+1)\tau_n )- u(t)||  < 3 \sqrt{2} \sqrt{|(l+1) \tau_n - t|} \sqrt{\mathcal{E}(u_0) - \alpha} < \frac{\epsilon}{2} 
\end{align}
 since $(l+1) \tau_n \in [t,t+h]$ for all $l \in [k_n(t)-1, k_n(t+h)- 2]$. Also, 
 for each $n \geq n_0$ as defined in the beginning of the proof one has
\begin{align*}
 ||u^{\tau_n}((l+1) \tau_n ) - u(t) || & \leq ||u^{\tau_n}((l+1) \tau_n ) - u((l+1) \tau_n)|| + || u((l+1) \tau_n ) - u(t) || 
 \\ & \leq ||u^{\tau_n} - u||_{C([0,T],H)} + 3 \sqrt{2} \sqrt{\E(u_0) - \alpha }\sqrt{t- (l+1) \tau_n }
 \\  &  \leq || u^{\tau_n} - u||_{C([0,T],H)} + \frac{\epsilon}{2} < \epsilon  ,
\end{align*}  
where we used the same estimate as in \eqref{eq:esti1} in the last step. 
All in all, for $R_T := ||u_0|| + 3 \sqrt{2} \sqrt{\E(u_0) -\alpha} \sqrt{T} $ one obtains with \eqref{eq:Snh} that for $n \geq n_0$ 
\begin{equation*}
|S_{n,h}| \leq ( L(||v|| + 3\sqrt{2(\E(u_0) - \alpha)}\sqrt{T} + \sup_{w \in B_{R_T}(0)}||\nabla\E(w)|| ) ||u^{\tau_n} - u||_{C([0,T],H)}
\end{equation*}
and therefore $|J_{n,h}| \rightarrow 0 $ as $n \rightarrow \infty$ for each fixed $h \in (0, h_0)$. 

The original computation can now be simplified to 
\begin{equation*}
(\dot{u}(t) , v- u(t)) \geq   \liminf_{h \rightarrow 0 } \liminf_{n \rightarrow \infty} \frac{1}{h} \sum_{l = k_n(t) - 1}^{k_n(t+h)-2}  ( -\tau_n \nabla \mathcal{E}(u((l+1) \tau_n) ) , v- u((l+1) \tau_n))).
\end{equation*}
Notice that $\{k_n(t) \tau_n , ... , (k_n(t+h)- 1) \tau_n \}$ is a partition of $[t,t+h]$ and $t \mapsto (\nabla\mathcal{E}(u(t)) ,v- u(t) )$ is Riemann integrable as continuous map, see Proposition \ref{prop:limitra} and Assumption \ref{ass:main}. 
Therefore 
\begin{align*}
(\dot{u}(t),v- u(t) ) & \geq  \liminf_{h \rightarrow 0 } \liminf_{n \rightarrow \infty} \frac{1}{h} \sum_{l = k_n(t) - 1}^{k_n(t+h)-2}  \tau_n ( - \nabla \mathcal{E}(u((l+1) \tau_n) ) , v- u((l+1) \tau_n) ))
\\ &=  \liminf_{h \rightarrow 0} \frac{1}{h} \int_{t}^{t+h} (-\nabla \mathcal{E}(u(s)) , v- u(s) ) \; \mathrm{d}s = - ( \nabla \mathcal{E}(u(t)), v- u(t)) 
\end{align*}
where the last equality holds because of continuity. This proves the claim. 
\end{proof}
\begin{theorem}[Existence Result] \label{thm:exisres}

Suppose Assumption \ref{ass:main} and \ref{ass:growth}. Let $\mathcal{G}$ be the set of Corollary \ref{cor:closure} and $\mathcal{P}$ be the set in Definition \ref{def:precon}. Then 
$
\mathcal{G} \supset \overline{\mathcal{P} },
$ 
i.e. for each $u_0 \in \overline{\mathcal{P}}$ there exists an Obstacle Gradient Flow starting at $u_0$. 
\end{theorem}
\begin{proof}
Because of Corollary \ref{cor:closure} it is enough to show that $ \mathcal{G} \supset \mathcal{P}$. But this follows from Proposition \ref{prop:limitra}, Corollary \ref{cor:regu} and Lemma \ref{lem:fvi}.
\end{proof}

\begin{remark}
If one imposes Assumption \ref{ass:main} only one has still $\mathcal{G} \supset \mathcal{P}$. 
In the section to follow, we will outline a situation where $\mathcal{P}$ is dense in $C$ and Assumption \ref{ass:growth} holds true. Then, the previous theorem implies that $\mathcal{G} =C $, a desirable situation. 
\end{remark}

\section{Fourth Order Flows with Obstacles and Navier Boundary Conditions}\label{Sec:Navier}
For this section we fix an \emph{obstacle function} $\psi \in C^0([0,1])$ such that $\psi(0), \psi(1) < 0 $. We are considering gradient flows for the following class of fourth order obstacle problems

\begin{definition}[The Considered Framework]\label{def:navier}
Set $H := W^{2,2}(0,1) \cap W_0^{1,2}(0,1)$  
with scalar product 
$
(u,v) := \int_0^1 u'' v'' \; \mathrm{d}x 
$
and 
$
C := \{ u \in H | u \geq \psi \; a.e. \} .
$
Additionally, we suppose that 
\begin{equation}\label{eq:energ}
\E(u) := \int_0^1 \left( \frac{d}{dx} (G \circ u')(x) \right)^2  \; \mathrm{d}x+ 2 \int_0^1 K(u'(x))   \; \mathrm{d}x
\end{equation}
for functions $G, K \in C^3(\mathbb{R})$ that satisfy $G' > 0$ and $ K \geq 0 $.   
\end{definition}
First, we will check the conditions that are necessary to guarantee existence of the flow and then show further properties. An important property will be that the flow preserves $W^{3,2}$-space regularity and Navier Boundary conditions. The reason why this is so important is that both properties hold for minimizers of the static problem, as a close examination of \cite[Corollary 3.3]{Anna} and \cite[Theorem 5.1]{Anna} should reveal.
\begin{remark}
The scalar product given in Definition \ref{def:navier} is actually a scalar product and $H$ is actually a Hilbert Space because of \cite[Theorem 2.31]{Sweers}.
\end{remark}
\subsection{Verification of Existence Conditions}
\begin{remark}
Notice that $C^\infty([0,1]) \cap W_0^{1,2}(0,1)$ is dense in $H$. Indeed, if $\phi \in H$ is arbitrary but fixed then there exists $(\phi_n)_{n \in \mathbb{N}} \subset C^\infty([0,1])$ such that $\phi_n \rightarrow \phi $ in $W^{2,2}(0,1)$. As an easy computation shows
\begin{equation*}
\widetilde{\phi}_n(x) := \phi_n(x)  - (\phi_n(1)- \phi_n(0)) x - \phi_n(0)  
\end{equation*}
defines a sequence in $C^\infty([0,1]) \cap W_0^{1,2}(0,1) $ that converges to $\phi$ in $H$ .   
\end{remark}
\begin{prop}[Verification of the Assumptions, Proof in Appendix \ref{sec:App}]\label{prop:3.4}
The energy $\E$ defined in \eqref{eq:energ} satisfies Assumption \ref{ass:main} and \ref{ass:growth} and 
\begin{equation}\label{eq:frech}
 D\E(u) (\phi) =  \int_0^1 2 u'' \phi'' G'(u')^2 \; \mathrm{d}x  + 2 \int_0^1 K'(u') \phi' \; \mathrm{d}x + 2 \int_0^1 G'(u') G''(u') u''^2 \phi' \; \mathrm{d}x
\end{equation}
for all $u,\phi \in H$.  
\end{prop} 
\begin{prop}[Verification of HPR Property]
The set $C$  given in Definition \ref{def:navier} is convex, closed and has the HPR property as defined in Definition \ref{def:halfi}. 
\end{prop}
\begin{proof} Convexity and closedness are easy to show. We only show the HPR property.
We have to show that for $u \in C$ and $v \in H$, $u + \pi_C(v) - v \in C$. We do this by showing that $v -\pi_C(v) \leq 0 $ for each $v \in H$. If this is shown $u \geq \psi$ implies certainly $u + \pi_C(v) - v \geq \psi$, which proves the claim. Now let $v \in H$ arbitrary but fixed. 
For each $ w \in C$  we infer from \cite[Theorem 5.2]{Brezis} that 
$
(v- \pi_C(v), w - \pi_C(v)) \leq 0 .
$
Let $\phi \in W_0^{1,2}(0,1) \cap W^{2,2}(0,1)$ be arbitrary such that $\phi \geq 0 $ almost everywhere. Choosing $w = \pi_C(v)+ \phi$  we find that  $w \in C$ and therefore 
\begin{equation*}
\int_0^1 (v - \pi_C(v)) '' \phi'' \; \mathrm{d}x = (v - \pi_C(v) , \phi) \leq 0 .
\end{equation*}
Fix $f \in L^2(0,1)$ be such that $f \leq 0 $ a.e.. Then, according to elliptic regularity and the weak maximum principle, there exists $\phi \in W_0^{1,2}(0,1)\cap W^{2,2}(0,1)$ such that $\phi \geq 0 $ and $\phi'' = f $ a.e. This implies that 
\begin{equation*}
\int_0^1 (v - \pi_C(v)) '' f \; \mathrm{d}x \leq 0 \quad \forall f \in L^2(0,1) : f \leq 0  \; a.e.  
\end{equation*}
However this results in $(v- \pi_C(v))'' \geq 0 $ almost everywhere. The weak maximum principle implies that $v- \pi_C(v) \leq 0 $ almost everywhere which proves the claim. 
\end{proof}
\begin{lemma}[A Recursion Identity, Proof in Appendix \ref{sec:App}]\label{lem:ind:rec}
Suppose that $(a_k)_{k \in \mathbb{N}_0} \subset \mathbb{R}$ is a sequence such that $a_0 \geq 0 $ and there are $\tau, Z_1, Z_2 > 0 $ such that 
$
a_{k+1} \leq (1+ Z_1 \tau) a_k + Z_2 \tau$ for each $k \in \mathbb{N}$.
Then
\begin{equation*}
a_l \leq ( 1 + Z_1 \tau)^l a_0  + \frac{Z_2}{Z_1} ( (1+ Z_1 \tau)^l - 1) \quad \forall l \in \mathbb{N}. 
\end{equation*}
\end{lemma}

\begin{prop}[Precompactness of the Discrete Trajectories]\label{prop:existence}

Suppose that $u_0 \in W^{3,\infty }(0,1) \cap H$ and $u_0''(0)= u_0''(1) = 0 $. Let $(u_{k \tau})_{k \in \mathbb{N}, \tau \in (0,1)}$ be a family of minimizing movement sequences generated by Algorithm \ref{algo:minimov} for \eqref{eq:energ}.  Then for each $T> 0 $ the set $(u_{k \tau } )_{ \tau \in (0,1), k \tau \leq T } $ is bounded in $W^{3,\infty }(0,1)$ and therefore precompact in $H$.   
\end{prop}
\begin{proof}
Fix $T> 0$. Let $\widetilde{\delta} > 0 $ be such that $\psi_{ \mid_{[0, \widetilde{\delta}]}} \leq \frac{\psi(0)}{2}$ and $\psi_{\mid_{[1-\widetilde{\delta}, 1]}} \leq \frac{\psi(1)}{2}$.
Define 
\begin{equation*}
\delta := \min \left\lbrace \widetilde{\delta} , \frac{\min(|\psi(0)|, |\psi(1)|)}{4 C_P ( ||u_0|| + 3 \sqrt{2} \sqrt{\E(u_0) - \alpha } \sqrt{T}) } \right\rbrace 
\end{equation*}
where $C_P$ denotes the operator norm of the embedding $H \hookrightarrow W^{1,\infty}(0,1) \cap W_0^{1,2}(0,1).$ 
Let $\tau \in (0,1)$ be arbitrary. We first claim that $u_{k \tau} > \psi$ on $[0, \delta] \cup [1-\delta ,1]$. Indeed, for arbitrary $\theta \in [0, \delta]$ we find using \eqref{eq:unifest}
\begin{align*}
|u_{k\tau} (\theta)| & \leq \delta ||u_{k \tau}'||_\infty \leq \delta C_P ||u_{k\tau}||\\ & \leq \delta C_P(||u_0|| + 3 \sqrt{2} \sqrt{\E(u_0) - \alpha}\sqrt{T}) < \frac{|\psi(0)|}{2} \leq |\psi(\theta)|.
\end{align*} 
Since $\psi(\theta)$ is negative, the claim follows. For $\theta \in [1- \delta, 1]$ the estimate can be shown similarly.

 Recall from Propositon \ref{prop:standard} that $u_{(k + 1) \tau} \in \arg \min_{w \in C} \E(w) + \frac{1}{2\tau } ||w - u_{k \tau}||^2$ implies that 
\begin{equation*}
\frac{1}{\tau} (u_{(k+1) \tau} - u_{k \tau} , v - u_{ (k+1) \tau }) +( \nabla \E(u_{(k+1) \tau}) , v- u_{(k+1) \tau } ) \geq 0  \quad \forall v \in C. 
\end{equation*} 
For the sake of simplicity of notation we define $u := u_{ (k+1) \tau } $ and leave out the integration indicators $'\mathrm{d}x'$. Note that for each $ \phi \in C_0^\infty(0,1)$ such that $\phi \geq 0 $ one can use $u+ \phi$ as test function in \eqref{eq:eulerlag}. Using this and Proposition \ref{prop:3.4} we find
\begin{equation}\label{eq:eli}
\int_0^1 \frac{u''- u_{k \tau }'' }{\tau} \phi'' + \int_0^1 2 u'' \phi'' G'(u')^2 + 2 \int_0^1 K'(u') \phi' + 2 \int_0^1 G'(u') G''(u') u''^2 \phi' \geq 0 .
\end{equation}
Additionally, for each $\phi \in C_0^\infty(\{ u > \psi \}) $ and also for each $\phi \in C_0^\infty([0, \delta)) $ such that $\phi(0) = 0 $ 
\begin{equation}\label{eq:elk}
\int_0^1 \frac{u''- u_{k \tau }'' }{\tau} \phi'' + \int_0^1 2 u'' \phi'' G'(u')^2 + 2 \int_0^1 K'(u') \phi' + 2 \int_0^1 G'(u') G''(u') u''^2 \phi'  = 0 .
\end{equation}
\textbf{Our first claim is that $u \in W^{3,\infty}(0, \delta)$ and $u''(0) = 0$}.  For this we proceed by induction over $k$. The first step of the induction is $k = 0$ so $u = u_{1,\tau} \in W^{3,\infty} ( 0,\delta)$. For this we observe with \eqref{eq:elk} that for all $\phi \in C_0^\infty(0,\delta) $ 
\begin{equation} \label{eq:=0}
\int_0^1 \frac{u''- u_{0 }'' }{\tau} \phi'' + \int_0^1 2 u'' \phi'' G'(u')^2 + 2 \int_0^1 K'(u') \phi' + 2 \int_0^1 G'(u') G''(u') u''^2 \phi'  = 0  .
\end{equation}
This implies (see \cite[Lemma 13.1]{Hestenes}) that
$
\frac{u'' - u_0''}{\tau} +2  u'' G'(u')^2
$
is weakly differentiable and there is a constant $Z \in \mathbb{R}$ such that 
\begin{equation}\label{eq:3.56}
\left( \frac{u'' - u_0''}{\tau} +2  u'' G'(u')^2 \right)' = Z + 2 K'(u') + 2 G'(u')G''(u') u''^2 \in L^1 (0 ,\delta).
\end{equation}
Since $u_0'' \in W^{1,\infty}(0, \delta)$ and $\frac{1}{G' \circ u'} \in W^{1,2}(0,1) $ as $\frac{1}{G'}$ is locally Lipschitz, we obtain by the product rule that $u'' \in W^{1,2}(0, \delta) $ and 
\begin{equation}\label{eq:bootstr}
 u''' = \frac{1}{\frac{1}{\tau} + 2 G'(u')^2} \left( \frac{1}{\tau} u_0'''  + Z +  2K'(u')  - 2 G'(u') G''(u') u''^2  \right) .
\end{equation}  
We infer from the equation that $u''' \in L^1(0,\delta)$. From this however, one concludes that $u'' \in L^\infty(0,\delta)$ and as a result of that, \eqref{eq:bootstr} yields that $u''' \in L^\infty(0,\delta)$. The induction step is now very similar since we can assume that $u_{k \tau } \in W^{3,\infty}$ and therefore we generate the same equations for $u$ as before for with $u_0$ replaced by $u_{k \tau }$. We leave the details to the reader. 
To show that $u''(0) = 0 $ we proceed again by induction. Here we just show that $u_{k \tau } '' (0) = 0 $ implies that $u''(0) = u_{ (k+1) \tau }''(0)  = 0$. For this we conclude just like in \eqref{eq:=0} and \eqref{eq:3.56} that 
\begin{equation}\label{eq:neufuer}
\left( \frac{u'' - u_{k \tau}''}{\tau} + 2 u'' G(u')^2 \right)' = \widetilde{Z} + 2 K'(u') + 2 G'(u') G''(u') u''^2
\end{equation}
for some $\widetilde{Z} \in \mathbb{R}$. Now plug some $\phi\in C_0^\infty([0, \delta)) $ into \eqref{eq:elk} that satisfies $\phi(0) =0 $ and $\phi'(0) \neq 0 $. Integrating by parts in \eqref{eq:elk} we obtain 
\begin{align}\label{eq:3.60}
0 & = \left[ \left( \frac{u'' - u_{k \tau }'' }{\tau } + 2 u'' G'(u')^2 \right) \phi' \right]_0^\delta \nonumber \\ &  -  \int_0^\delta \phi' \left(
\frac{u'''- u_{k \tau}'''}{\tau }+ \frac{d}{dx} ( 2 u'' G'(u')^2) - 2 K'(u') - 2 G'(u') G''(u') u''^2 \right) \; \mathrm{d}x .
\end{align}
Evaluating in the first summand  and using \eqref{eq:neufuer} we obtain
\begin{equation}\label{eq:4.25}
0  = \phi'(0) \left(  \frac{u''(0) - u_{k \tau } ''(0) }{\tau } + 2 u''(0) G'(u'(0))^2  \right) - \int_0^\delta \phi' \widetilde{Z} \; \mathrm{d}x .
\end{equation}
 The last integral evaluates to $ \widetilde{Z} ( \phi(\delta)- \phi(0) ) = 0 $.  
Using the induction hypothesis that  $u_{k \tau } '' (0) = 0 $ we find $u''(0) =0$. We can show very similarly that $u \in W^{3,\infty}((1-\delta, 1))$ and $u''(1) = 0$. \textbf{Next, we show that $u \in W^{3,\infty}(0,1)$.}
  We conclude from \eqref{eq:elk} and \eqref{eq:eli} and \cite[Corollary 1, Section 1.8]{EvansGariepy} that there exists a Radon measure $\mu_{k \tau}$ on $(0,1)$ supported on $\{ u = \psi \} $ such that for all $\phi \in C_0^\infty(0,1)$: 
\begin{equation}\label{eq:muktau}
\int_0^1 \frac{u''- u_{k \tau }'' }{\tau} \phi'' + \int_0^1 2 u'' \phi'' G'(u')^2 + 2 \int_0^1 K'(u') \phi' + 2 \int_0^1 G'(u') G''(u') u''^2 \phi' = \int \phi \; \mathrm{d}\mu_{k \tau}.
\end{equation}
The measure $\mu_{k \tau}$ is finite since $\{u = \psi\} \subset [\delta, 1- \delta] $ is compactly contained in $(0,1)$.  
Using Fubini's Theorem, we find for $\phi \in C_0^\infty(0,1) $
\begin{equation}\label{eq:vertfkt}
\int_0^1 \phi \; \mathrm{d}\mu_{k \tau}  = \int_0^1 \int_0^x \phi'(s) \; \mathrm{d}s \; \mathrm{d}\mu_{k \tau}(x) = \int_0^1 \mu_{k\tau}([ s,1]) \phi'(s) \; \mathrm{d}s . 
\end{equation}
Therefore 
\begin{align}\label{eq:ABCD}
\int_0^1 \frac{u''- u_{k \tau }'' }{\tau} \phi'' & + \int_0^1 2 u'' \phi'' G'(u')^2 + 2 \int_0^1 K'(u') \phi' \nonumber \\ & + 2 \int_0^1 G'(u') G''(u') u''^2 \phi' - \int \phi' \mu_{k \tau}([x,1]) \; \mathrm{d}x = 0 .
\end{align}
From here we can proceed like in \eqref{eq:=0} and \eqref{eq:3.56}. Define $m_{k\tau}(x):= \mu_{k\tau}( [x,1])$ and observe, according to \eqref{eq:ABCD}, that there is $C_{k \tau } \in \mathbb{R}$ such that 
\begin{equation}\label{eq:mkatau}
\left( \frac{u''- u_{k\tau}''}{\tau } + 2 G'(u') u'' \right)' =- m_{k \tau}(x) - C_{k\tau} + 2G'(u') G''(u') u''^2 + 2K'(u')  \in L^1(0,1) 
\end{equation} 
Using the product rule again we obtain that $u \in W^{3,1}(0,1)$ and 
\begin{equation}\label{eq:3.71}
\frac{u''' - u_{k\tau} '''}{\tau} + 2 G'(u')^2 u''' = - m_{k\tau}(x) - C_{k\tau} - 2 G'(u') G''(u') u''^2 + 2 K'(u')  
\end{equation}
and rearrranged 
\begin{equation}\label{eq:u'''}
u''' = \frac{1}{\frac{1}{\tau } + 2 G'(u')^2} \left( \frac{u_{k\tau}'''}{\tau } - m_{k \tau} -C_{k\tau}  - 2 G'(u') G''(u') u''^2 + 2 K'(u') \right) .
\end{equation}
Note that this already implies that $u \in W^{3, \infty}(0,1)$ since $u \in W^{3,1}(0,1) $ implies that $u'' \in L^\infty(0,1) $ and given this information, we conclude from \eqref{eq:u'''} that $u''' \in W^{3,\infty}(0,1)$.

\textbf{We claim that there is $D = D(\delta, \E(u_0), ||u_0|| , T )$ independent of $k, \tau$ such that} 
$
 \mu_{k \tau} (0,1) \leq D .
$
 To show this, choose a fixed $\overline{\varphi} \in C_0^\infty(0,1)$ such that $0 \leq \overline{\varphi} \leq 1 $ and $\phi_{\delta} \equiv 1 $ on $[\delta, 1- \delta]$. Set $R_T :=  (||u_0|| + 3\sqrt{2}\sqrt{\E(u_0) - \alpha } \sqrt{T})$ and $B(T) := B_{C_P R_T} (0)$, where - recall - $C_P$ is the operator norm of the embedding operator $\iota : H \hookrightarrow W^{1,\infty}(0,1)$. 
 Also define $||F||_{\infty,B(T)} := \sup_{z \in B(T)} |F(z)|$ for given $F \in C^0(\mathbb{R})$. We can estimate using Remark  \ref{rem:energdec} and Corollary \ref{cor:growthest} 
\begin{align*}
\mu_{k\tau} (0,1) & = \mu_{k \tau } ( [\delta , 1- \delta] ) \leq \int_0^1 \overline{\varphi} \; \mathrm{d}\mu_{k \tau } 
\\ &= \int_0^1 \frac{u''- u_{k \tau }'' }{\tau} \overline{\varphi}'' + \int_0^1 2 u'' \overline{\varphi}'' G'(u')^2 \\ & \quad + 2 \int_0^1 K'(u') \overline{\varphi}' + 2 \int_0^1 G'(u') G''(u') u''^2 \overline{\varphi}'  
\\ & \leq \frac{||u - u_{k \tau}||}{\tau } ||\overline{\varphi}|| + 2 \left( \int_0^1 u''^2 G'(u')^4 \right)^\frac{1}{2} ||\overline{\varphi}|| \\ & \quad + 2||K'||_{\infty,B(T)} ||\overline{\varphi}'||_\infty + 2 ||G'||_{\infty, B(T)} ||G''||_{\infty, B(T)}  \int_0^1 u''^2 \; \mathrm{d}x  ||\overline{\varphi}'||_\infty 
\\ & \leq ||\overline{\varphi}|| \left(  \sqrt{2} \left(\frac{ \E(u_{k\tau})- \E(u)}{\tau} \right)^\frac{1}{2} + 2 ||G'||_{\infty,B(T)}^2R_T  \right) \\ & \quad + ||\overline{\varphi}||  \left( 2 C_P ||K'||_{\infty,B(T)}^2 +2 ||G'||_{\infty, B(T)} ||G''||_{B(T)}C_P R_T^2 \right) .
\end{align*} 
Note  that $\frac{\E(u_{k\tau})-\E(u) }{\tau} \leq C_T $ where $C_T> 0$ is the constant from Proposition \ref{prop:energdec}. All in All we obtain $D$ independent of $k, \tau$ such that 
$
\mu_{k\tau}(0,1) \leq D,  
$ as claimed.

\textbf{We continue showing that $|u'''|$ is bounded independently of $k , \tau$}.
We first prove an estimate for $|C_{k\tau}|$, defined as in \eqref{eq:mkatau}. For this observe that, using $u''(0) = u''(1)= 0$ and \eqref{eq:u'''}, 
\begin{align*}
0 & =  \int_0^1 u'''  =  \int_0^1 \frac{\frac{1}{\tau} u_{k\tau} '''}{\frac{1}{\tau } + 2 G'(u')^2 } - \int_0^1 \frac{m_{k\tau}}{{\frac{1}{\tau } + 2 G'(u')^2 }} + C_{k \tau } \int_0^1 \frac{1}{\frac{1}{\tau } + 2 G'(u')^2 } \\ & \qquad \qquad \qquad \qquad \qquad \qquad   - \int_0^1 \frac{2 G'(u') G''(u') u''^2}{\frac{1}{\tau} + 2 G'(u')^2} + 2 \int_0^1 \frac{K'(u')}{\frac{1}{\tau }+ 2 G'(u')^2}.
\end{align*}
Notice that $\int_0^1 u_{k\tau}''' =u_{k\tau}''(1)- u_{k\tau}''(0)  = 0$ and therefore  
\begin{equation}\label{eq:4.44}
 \int_0^1 \frac{\frac{1}{\tau} u_{k\tau} '''}{\frac{1}{\tau } + 2 G'(u')^2 } =- \int_0^1 \frac{2 G'(u')^2u_{k \tau}'''}{\frac{1}{\tau} +2 G'(u')^2}.
\end{equation}
Hence
\begin{align*}
|C_{k\tau}| \int_0^1 \frac{1}{{\frac{1}{\tau } + 2 G'(u')^2 }} & \leq \tau ||m_{k \tau}||_\infty + 2||G'||_{\infty,B(T)}^2 \tau \int_0^1 |u_{k\tau}'''|  \\ & \quad +2 \tau ||K'||_{\infty,B(T)} + 2\tau ||G'||_{\infty,B(T)}  ||G''||_{\infty, B(T)} R_T^2 .  
\end{align*}
Now observe that  since $\tau < 1 $
\begin{equation*}
|C_{k\tau}| \int_0^1 \frac{1}{{\frac{1}{\tau } + 2 G'(u')^2 }} \geq |C_{k\tau} | \frac{1}{\frac{1}{\tau }+ 2 ||G'||^2_{\infty, B(T)} } \geq \frac{|C_{k \tau} | \tau}{1+ 2 ||G'||^2_{\infty,B(T)}} .
\end{equation*}
All in all we get the crucial estimate for $C_{k\tau}$: 
\begin{align}\label{eq:Cktau}
|C_{k\tau} | & \leq (1 + 2 ||G'||^2_{\infty,B(T)}) \bigg( ||m_{k\tau}||_\infty + 2 ||G'||_{\infty,B(T)} ||G''||_{\infty,B(T)}R_T^2 \nonumber  \\ & \qquad \qquad  \qquad \qquad \qquad \qquad  + 2||K'||_{\infty, B(T)} +  2||G'||^2_{\infty,B(T)}  \int_0^1 |u_{k\tau}'''| \bigg) \nonumber
\\ &   \leq (1 + 2 ||G'||^2_{\infty,B(T)}) \bigg( D +2 ||G'||_{\infty,B(T)} ||G''||_{\infty,B(T)}R_T^2 \nonumber \\ &  \qquad \qquad  \qquad \qquad \qquad \qquad  + 2||K'||_{\infty, B(T)} +  2||G'||_{\infty,B(T)}^2 \int_0^1 |u_{k\tau}'''| \bigg) \nonumber
\\ & \leq  Z_1 \int_0^1 |u_{k\tau}'''|  + Z_2 ,
\end{align}
where $Z_1>0$, $Z_2 \in \mathbb{R}$ are constants only depending on $T$ but independent of $k , \tau $. Taking absolute values and integrating in $\eqref{eq:u'''}$, then using the derived estimate for $|C_{k \tau}|$ we obtain that 
\begin{align*}
\int_0^1 |u'''| & \leq \int_0^1 |u_{k\tau}'''| + \tau D +  \tau |C_{k\tau}|  + 2 \tau \int_0^1 |G'||G''| u''^2 + 2 \tau ||K'||_{\infty,B(T)} 
\\ & \leq  (1 + Z_1 \tau) \int_0^1 |u_{k \tau} ''' | + \tau ( Z_2  + 2 ||G'||_{\infty,B(T)} ||G''||_{\infty, B(T)} R_T^2 + 2 ||K'||_{\infty,B(T)}+D ) 
\\ & \leq (1+ Z_1 \tau ) \int_0^1 |u_{k\tau}'''| + \tau Z_3 
\end{align*}
for some $Z_3$ independent of $k, \tau$. 
If we set $a_{k \tau} := \int |u_{k\tau}'''| $ we obtain the recursive inequality 
\begin{equation}
a_{(k+1)\tau} \leq (1+ Z_1\tau) a_{k \tau} + \tau Z_3.
\end{equation}
This given, Lemma \ref{lem:ind:rec} shows that 
\begin{equation*}
a_{k\tau} \leq (1+Z_1 \tau)^k a_{0\tau} + \frac{Z_3}{Z_1} ( (1+ Z_1\tau)^{k} - 1 ). 
\end{equation*}
We observe that $\tau \leq \frac{T}{k}$ and therefore 
\begin{equation*}
a_{k\tau} \leq \left(1+Z_1 \frac{T}{k}\right)^k a_{0\tau} + \frac{Z_3}{Z_1} \left( \left(1+ Z_1\frac{T}{k}\right)^{k} - 1 \right) \leq e^{Z_1T} a_{0\tau} + \frac{Z_3}{Z_1} (e^{Z_1T} - 1) 
\end{equation*}
since $\left(\left(1 + \frac{Z}{k} \right)^k \right)_{k\in \mathbb{N}}$ is bounded from above by $e^Z$ for all $Z> 0$.  
Hence and since $a_{0\tau }  = \int_0^1 |u_0'''|  $ for each $\tau \in (0,1)$ we have
\begin{equation*}
 \int |u'''| , \int |u_{k\tau}'''| \leq  e^{Z_1T} \int |u_0'''| + \frac{Z_3}{Z_1} (e^{Z_1T} - 1).
\end{equation*}
Uniform Boundedness of $||u_{k\tau}'''||_{L^1}$ independent of $k, \tau$ implies also uniform boundedness of $C_{k\tau}$ independent of $k, \tau$, see \eqref{eq:Cktau}.  Together with $u_{k\tau}''(0) = 0 $ we obtain uniform boundedness of $||u_{k\tau}''||_\infty$ independent of $k, \tau$. Using all these uniform bounds in \eqref{eq:u'''} we obtain that there exists some $Z_4$ independent of $k,\tau$ such that
\begin{equation*}
|u'''(x)| \leq |u_{k \tau} '''(x) | + Z_4 \tau \quad a.e. x \in (0,1).
\end{equation*}
If we understand this again as a recursion formula (recall $u =u_{(k+1)\tau}$ and $k$ was arbitrary), we obtain that for each $k \in \mathbb{N}$ such that $k \tau \leq T$
\begin{equation}\label{eq:3.109}
|u_{k\tau}'''(x)| \leq |u_0'''(x) | + Z_4 k \tau \leq ||u_0'''||_\infty + Z_4 T .
\end{equation}
The claim follows. 
\end{proof}
\begin{remark}\label{rem:navidisrc}
The proof of Proposition \ref{prop:existence} reveals that for $u_0$ as in the statement of the Proposition we have that $u_{k \tau}''(0) = u_{k \tau}''(1)= 0 $ for each $k\in \mathbb{N}$, and $\tau \in (0,1)$. This property is noteworthy and also carries over to the Obstacle Gradient Flow, as discussed in Theorem \ref{thm:navier}. 
\end{remark}
\begin{remark}
If we omit the assumption $u_0''(0) = u_0''(1) = 0 $, the argument of the previous proof can not be repeated. The reason for that is that  \eqref{eq:4.25} shows that the second derivatives at the boundary are not necessarily maintained and this might generate contributuions to the term in \eqref{eq:4.44} which are not necessarily elements of $o(\tau)$. This however was crucial for the proof.
\end{remark}

\subsection{Some Qualitative Properties}

\begin{theorem}[Existence, Space Regularity and Navier Boundary Conditions] \label{thm:navier}

Let $\E,C,H$ be as in Definition \ref{def:navier}. Then for each $u_0 \in W^{3,\infty}(0,1) \cap W_0^{1,2}(0,1)$ such that $u_0''(0) = u_0''(1)=0$ there exists an Obstacle Gradient Flow $(u(t))_{t \geq 0} $ starting at $u_0$.  Additionally, for every $t> 0 $, $u(t), \dot{u}(t)  \in  W^{3, 2}(0,1)$  and $u(t)''(0) = u(t)''(1)= 0$. Furthermore,  $ \dot{u}(t)''(0) = \dot{u}(t)''(1)= 0$ for almost every $t> 0 $. Moreover, for each $T> 0 $
\begin{equation*}
\sup_{ t\in [0,T] } ||u(t)||_{W^{3,2}} < \infty , \quad  \mathrm{esssup}_{ t \in [0,T]} || \dot{u}(t) ||_{W^{3,2}} < \infty .  
\end{equation*}
\end{theorem} 
\begin{proof}
The existence follows from Proposition \ref{prop:3.4}, Proposition \ref{prop:existence} and Theorem \ref{thm:exisres}. Fix $t> 0$. By Corollary \ref{cor:const} there exists $\tau_n \rightarrow 0 $  such that $\overline{u}^{\tau_n} ( t) \rightarrow u(t)$ in $H$. Because of the boundedness of $\overline{u}^{\tau_n} (t) $ in $W^{3,2}(0,1)$, see Proposition \ref{prop:existence}, we can extract a subsequence $\tau_{l_n} \rightarrow 0 $ such that $\overline{u}^{\tau_{l_n}}(t)  \rightharpoonup u(t) $ in $W^{3,2}(0,1)$. From this follows the $W^{3,2}$-regularity of $u(t)$. Since  $\sup_{t \in [0,T], n \in \mathbb{N}} ||\overline{u}^{\tau_n}(t)||_{W^{3,2}} < \infty$ by Proposition \ref{prop:existence} and the norm is weakly lower semicontinuous we find that $\sup_{t \in [0,T]} || u(t)||_{W^{3,2}} < \infty $.  Because of the compact embedding $W^{3,2}(0,1) \hookrightarrow C^2[0,1]$ and Remark \ref{rem:navidisrc} we obtain that 
\begin{equation}\label{eq:navierrand}
u(t) '' (0) = \lim_{n \rightarrow \infty} \overline{u}^{\tau_{l_n}}(t)''(0) = 0
\end{equation}
Analogously one shows that $u(t) ''(1) = 0 $. Now fix $t> 0 $  such that  $(FVI)$ in Definition \ref{def:coneflow} holds true. The $(FVI)$ and \cite[Corollary 1, Section 1.8]{EvansGariepy} imply that there is a Radon measure $\mu_t$ supported on $ \{ u(t) =  \psi \}$ such that for all $\phi \in C_0^\infty(0,1)$ 
\begin{align*}
& \int_0^1 \dot{u}(t)'' \phi'' \; \mathrm{d}x + 2 \int_0^1 G'(u(t)')^2 u(t) '' \phi'' \; \mathrm{d}x  \\   & \qquad + 2 \int_0^1 G'(u(t)')G''(u(t)') u(t) ''^2 \phi' \; \mathrm{d}x +2 \int_0^1 K'(u(t)') \phi' \; \mathrm{d}x = \int \phi \; \mathrm{d}\mu_t. 
\end{align*} 
Since $\{ u(t) = \psi \}$ is compactly contained in $(0,1)$ (as $\psi(0), \psi(1) < 0 $ and $u(t)(0) = u(t) (1) = 0$), we obtain that $\mu_t$ is finite and therefore, we can derive like in \eqref{eq:vertfkt} that there is $m_t \in L^\infty(0,1)$  such that for all $\phi \in C_0^\infty(0,1)$ 
\begin{align*}
& \int_0^1 \dot{u}(t)'' \phi'' \; \mathrm{d}x + 2 \int_0^1 G'(u(t)')^2 u(t) '' \phi'' \; \mathrm{d}x  \\   & \qquad + 2 \int_0^1 G'(u(t)')G''(u(t)') u(t) ''^2 \phi' \; \mathrm{d}x + 2 \int_0^1 K'(u(t)') \phi' \; \mathrm{d}x = \int_0^1 m_t \phi' \; \mathrm{d}x .
\end{align*} 
Proceeding similar to \eqref{eq:mkatau}, \eqref{eq:3.71} and \eqref{eq:u'''} one can derive that
\begin{equation*}
\left( \dot{u}(t)'' + 2 G(u(t)')^2 u(t) '' \right)' \in L^\infty(0,1)
\end{equation*}
and 
\begin{equation}\label{eq:esssup}
\mathrm{esssup}_{ t \in [0,T] } || \left( \dot{u}(t)'' + 2 G(u(t)')^2 u(t) '' \right)'  ||_{L^\infty(0,1)} < \infty . 
\end{equation}
Since $\sup_{t \in [0,T]} ||u(t)||_{W^{3,2}} < \infty$ we find from \eqref{eq:esssup}  that $\mathrm{esssup}_{t \in [0,T]} ||\dot{u}(t)'''||_{L^2} < \infty $. To bound $\mathrm{esssup}_{ t \in [0,T] }||\dot{u}(t)||_{W^{3,2}} $ it remains to bound $\mathrm{esssup}_{t\in[0,T]}||\dot{u}(t)||_{W^{2,2}}$. For this, one observes using the fact that  $\dot{u}(t) \in W_0^{1,2}(0,1)$ for each $t > 0 $ and Corollary \ref{cor:boundder}  
\begin{align*}
\mathrm{esssup}_{t \in [0,T] } ||\dot{u}(t)||_{W^{2,2}} & \leq  C \;  \mathrm{esssup}_{t \in [0,T]} || \dot{u}(t)||_H \leq C \sup_{t \in [0,T]}||\nabla\E(u(t))||_H \\ & \leq C \sup_{w \in B_{R_T}(0)} ||\nabla\E(w)|| < \infty, 
\end{align*}
where $R_T := ||u_0|| + \sqrt{2(\E(u_0) - \alpha ) } \sqrt{T}$, see \eqref{eq:2.16}. 
 Now let $\delta > 0 $ be such that $[0 ,\delta] \cap \mathrm{supp} (\mu_t)  = \emptyset$. Then the flow variational inequality implies that 
\begin{align*}
& \int_0^1 \dot{u}(t)'' \phi'' \; \mathrm{d}x + 2 \int_0^1 G'(u(t)')^2 u(t) '' \phi'' \; \mathrm{d}x  \\   & \qquad + 2 \int_0^1 G'(u(t)')G''(u(t)') u(t) ''^2 \phi' \; \mathrm{d}x + 2 \int_0^1 K'(u(t)') \phi' \; \mathrm{d}x = 0
\end{align*}
 for all $\phi \in C_0^\infty([0, \delta)) $ such that $\phi(0)= 0, \phi'(0) \neq 0 $. Similar to \eqref{eq:=0}, \eqref{eq:3.56}  and \eqref{eq:3.60} one can derive that 
 \begin{equation*}
 \dot{u}(t)''(0) + 2 G'(u(t)')(0))^2 u(t)''(0) = 0 , 
\end{equation*}  
and therefore using \eqref{eq:navierrand} one finds that $\dot{u}(t)''(0) = 0$. Similarly one can compute $\dot{u}(t)''(1) = 0 $, which finishes the proof.
\end{proof}
\begin{remark}\label{rem:secder}
For the time-independent problem, \cite[Theorem 5.1]{Anna} shows $W^{3,\infty}$-regularity of each solution of the time independent variational inequality. For the discrete trajectories, we have also shown $W^{3,\infty}$-regularity and $W^{3,\infty}$ boundedness independent of the stepwidth $k,\tau$, see \eqref{eq:3.109}. Since $W^{3,\infty}$ is not reflexive, the regularity does not immediately carry over to the limit. However note that $C^{2,1}([0,1]) = W^{3,\infty}(0,1) $ and $||f||_{C^{2,1}} =  || f||_{W^{3,\infty}} $  for all $f \in C^{2,1}([0,1])$, see \cite[Theorem 6.12 and Exercise 6.14]{Heinonen} for $'\subset', '\geq'$ and $'\supset', '\leq'$ are immediate using the fundamental theorem of calculus. Now fix $T > 0$ and $t \in [0,T]$. Then, according to the proof of Theorem \ref{thm:navier}, there is a sequence $\tau_n \rightarrow 0 $ such that $\overline{u}^{\tau_n}(t) \rightharpoonup u(t)$ in $W^{3,2}(0,1)$. Now observe that for arbitrary but fixed $x,y \in [0,1]$, since $W^{3,2}(0,1)$ embeds compactly in $C^2([0,1])$
\begin{align*}
|u(t)''(x) - u(t)''(y) | & =  \lim_{n \rightarrow \infty} | \overline{u}^{\tau_n} (t) ''(x) - \overline{u}^{\tau_n}(t) ''(y) | \leq \sup_{n \in \mathbb{N}} || \overline{u}^{\tau_n}(t) ||_{C^{2,1}} |x-y| \\ & = \sup_{n \in \mathbb{N}} || \overline{u}^{\tau_n}(t) ||_{W^{3,\infty}} |x-y| = \sup_{ \tau \in (0,1), k \in \mathbb{N}: k \tau \leq T } ||u_{k \tau}||_{W^{3,\infty}} |x-y| 
\end{align*}
and the supremum in the last step is finite because of \eqref{eq:3.109}. Therefore $u(t) \in C^{2,1} = W^{3, \infty}$ for each $t > 0 $. 
\end{remark} 
Later we will be interested in asymptotic behavior of solutions. For this it is vital to obtain estimates for the third derivative uniformly in time. The rest of this section will be dedicated to such estimates.

\begin{prop}[Measurability in Time]\label{prop:meas}

Suppose that $u_0 \in W^{3,\infty}(0,1) \cap H$ is such that $u_0''(0) = u_0''(1) = 0$. Let $(u(t))_{t \geq 0} $ be the Obstacle Gradient Flow with initial datum $u_0$. 
Let $T> 0$ be arbitrary. Then $(0, T ) \ni t \mapsto u(t)''' \in L^2(0,1)$ and $(0,T) \ni t \mapsto \dot{u}(t)''' \in L^2(0,1)$ are Bochner measurable maps. Moreover 
\begin{equation}
(0,T) \times (0,1) \ni (t,x) \mapsto u(t)'''(x) 
\end{equation}   
is measurable with respect to the product Lebesgue measure on $(0,T) \times (0,1)$. In particular, there is a set $N \subset (0,1)$ of measure zero such that for each $x \in (0,1) \setminus N$, $ (0,T) \ni t \mapsto u(t)'''(x) \in \mathbb{R}$ is measurable. Analogous statements hold for $\dot{u}(t)'''$. 
\end{prop}
\begin{proof}
For the Bochner measurability, we use the Pettis measurability Theorem, see \cite[Section V.4]{Yosida}. The map is separably valued since $L^2(0,1)$ is separable. We claim that it is weakly measurable, i.e. for each $g \in L^2(0,1)$ 
\begin{equation*}
(0,T) \ni t \mapsto \int_0^1 u(t)'''(x) g(x) \; \mathrm{d}x \in \mathbb{R}
\end{equation*}
is Lebesgue measurable. For this we assume first that $g \in C_0^\infty(0,1)$. In this case 
\begin{equation*}
\int_0^1 u(t)'''(x) g(x) \; \mathrm{d}x = - \int_0^1 u(t)(x) g'''(x) \; \mathrm{d}x , 
\end{equation*}
which is  measurable in $t$ as a continuous function in $t$. For arbitrary $g \in L^2(0,1)$ there is $(g_n) \subset C_0^\infty(0,1)$ such that  $g_n \rightarrow g $ in $L^2(0,1)$. 
An easy computation shows that 
\begin{equation}\label{eq:approx}
\int_0^1 u(t)'''(x) g_n(x) \; \mathrm{d}x \rightarrow \int_0^1 u(t)'''(x) g(x) \; \mathrm{d}x \quad (n  \rightarrow \infty) 
\end{equation}
pointwise in $(0,T)$. The claim follows since pointwise limits of measurable functions are measurable. Hence Pettis measurability Theorem applies and Bochner measurability follows. The fact that $(t,x ) \mapsto u(t)'''(x)$ is also product measurable follows from the fact that $L^2((0,T), L^2(0,1)) $ and $L^2((0,T) \times (0,1))$ are isomorphic, see \cite[bottom of p.14]{Arendt}. The rest of the claim follows using Fubini's Theorem.  For $\dot{u}(t)'''$ everything is analogous except for the weak measurability, on which we will elaborate shortly. Observe that for $g \in C_0^\infty(0,1)$ we find 
\begin{equation*}
\int_0^1 \dot{u}(t)'''(x) g(x) \; \mathrm{d}x = - \int_0^1 \dot{u}(t) '' (x) g'(x) \; \mathrm{d}x = - ( \dot{u}(t), g' ) .
\end{equation*}
This however is the weak derivative of $t \mapsto - (u(t) ,g')$ and as such measurable. For arbitrary $g \in L^2(0,1)$ an approximation argument similar to \eqref{eq:approx} applies. 
\end{proof}

\begin{remark}
Definition \ref{def:coneflow}, Theorem  \ref{thm:navier} and Proposition \ref{prop:meas} together imply that the constructed Obstacle Gradient Flow $u$ lies in $W^{1,2}((0,T), W^{2,2}(0,1)) \cap L^\infty((0,T), W^{3,2}(0,1))$ for each $t> 0 $ which embeds compactly in $C([0,T], W^{2,\infty}(0,1))$ because of the Aubin-Lions-Simon Lemma, see \cite[Corollary 5]{Simon}. Because of Remark \ref{rem:secder}, we also have that $u(t) \in C^2([0,1])$ for each $t > 0 $. Since the norms on $W^{2,\infty}(0,1)$ and $C^2([0,1])$ coincide on $C^2([0,1])$, we obtain that $u \in C([0,T], C^2([0,1]))$ for each $T > 0$, i.e. $C([0, \infty),C^2([0,1]))$. 
\end{remark}

\begin{cor}[Regularity in Space-Time] \label{cor:nullset}

Let $u_0 \in W^{3,\infty}(0,1) \cap H$ be such that $u_0''(0) = u_0''(1) = 0 $. Then there is a set $N \subset (0,1)$ of Lebesgue measure zero such that for each $T > 0$ and $x \in (0,1) \setminus N$,   $(0,T) \ni t \mapsto u(t)'''(x)$ is weakly differentiable on $(0,T)$ and  $\partial_t (u(t) '''(x)) = \dot{u}(t) '''(x) $ for almost every $t \in (0,T)$. 
\end{cor}
\begin{proof}First fix $T \in \mathbb{N}$. 
Let $\eta \in C^\infty(0,1)$ and define 
\begin{equation*}
h_\eta(t) := \int_0^1 u(t)'''(r) \eta(r) \; \mathrm{d}r .
\end{equation*}  
We show that $h_\eta$ is weakly differentiable in $(0,T)$ with derivative 
\begin{equation*}
\partial_t h_\eta(t) := \int_0^1 \dot{u}(t)'''(r) \eta(r) \; \mathrm{d}r.
\end{equation*}
Indeed, fix $\phi \in C_0^\infty(0,T)$. Using Fubini's Theorem - which is applicable since $u, \dot{u} \in L^\infty((0,T),W^{3,2}(0,1))$ - we find  
\begin{align*}
\int_0^T h_\eta(t) \partial_t \phi(t) \; \mathrm{d}t & = \int_0^T  \left( \int_0^1 u(t) '''(r) \eta(r) \; \mathrm{d}r \right)  \partial_t\phi(t)  \; \mathrm{d}t 
\\ & = - \int_0^T \left( \int_0^1 u(t)(r) \partial_r^3\eta(r) \; \mathrm{d}r \right) \partial_t\phi(t) \; \mathrm{d}t
\\ & =  - \int_0^1  \left( \int_0^T u(t)(r) \partial_t \phi(t) \; \mathrm{d}t \right)  \partial_r^3 \eta(r) \; \mathrm{d}r.
\end{align*}
Observe that by the Riesz-Frechét theorem for each $r  \in (0,1) $ there is $\delta_r \in H $ such that $v(r) = (v, \delta_r)$ for each $v \in H$. Therefore 
\begin{align*}
 \int_0^T u(t)(r) \partial_t \phi(t) \; \mathrm{d}t  & = \int_0^T (u(t), \delta_r) \partial_t \phi(t)\; \mathrm{d}t 
 = \int_0^T (u(t) , \partial_t (\phi(t) \delta_r ) )  \; \mathrm{d}t 
 \\ & = - \int_0^T ( \dot{u }(t) ,  \phi(t) \delta_r) \; \mathrm{d}t 
 = - \int_0^T \phi(t ) ( \dot{u}(t) , \delta_r ) \; \mathrm{d}t 
\\&  = - \int_0^T \phi(t) \dot{u}(t)(r) \; \mathrm{d}t.   
\end{align*}
We conclude using Fubini's Theorem again
\begin{align}\label{eq:haeta}
\int_0^T h_\eta(t) \partial_t \phi(t) \; \mathrm{d}t & = \int_0^1 \left( \int_0^T \dot{u}(t)(r) \phi(t) \; \mathrm{d}t \right) \partial_r^3 \eta(r) \; \mathrm{d}r \nonumber
\\ & = \int_0^T \phi(t) \left(  \int_0^1 \dot{u}(t)(r) \partial_r^3 \eta(r) \; \mathrm{d}r \right) \; \mathrm{d}t \nonumber
\\ & = - \int_0^T \phi(t) \left( \int_0^1 \dot{u}(t)'''(r) \eta(r) \; \mathrm{d}r \right) \; \mathrm{d}t.  
\end{align}
The intermediate claim follows. Now let $(\eta_n)_{n \in \mathbb{N}}$ be a sequence of standard mollifiers.
 We claim that 
\begin{equation}\label{eq:molli1}
\int_0^1 \int_0^T |h_{\eta_n(x- \cdot)} (t) - u(t)'''(x) |^2 \; \mathrm{d}t \; \mathrm{d}x  \rightarrow 0 \quad (n \rightarrow \infty) 
\end{equation}
and 
\begin{equation}\label{eq:molli2}
\int_0^1 \int_0^T |\partial_t h_{\eta_n(x- \cdot)} (t) - \dot{u}(t)'''(x) |^2 \; \mathrm{d}t \; \mathrm{d}x  \rightarrow 0 \quad (n \rightarrow \infty). 
\end{equation}
From this follows by \cite[Section 1.3, Theorem 5]{EvansGariepy} that there is a subsequence $(\eta_{l_n})_{n \in \mathbb{N}}$ such that for almost every $x \in (0,1)$ 
\begin{equation*}
  \int_0^T |h_{\eta_{l_n}(x- \cdot)} (t) - u(t)'''(x) |^2 \; \mathrm{d}t  \rightarrow 0 \quad (n \rightarrow \infty) 
\end{equation*}
and 
\begin{equation*}
\int_0^T |\partial_t h_{\eta_{l_n}(x- \cdot)} (t) - \dot{u}(t)'''(x) |^2 \; \mathrm{d}t  \rightarrow 0 \quad (n \rightarrow \infty).
\end{equation*}
Once this is shown, we can infer that there exists a null set $N_T$ such that for $x \in (0,1) \setminus N_T$, $h_{\eta_{l_n}(x- \cdot)}$ is convergent in $W^{1,2}(0,T)$ and the limit coincides with $u(\cdot)'''(x)$. Moreover, the weak derivative of the limit corresponds to the $L^2(0,T)$-limits of $\partial_t h_{\eta_{l_n}(x-\cdot)}$, i.e. $\partial_t [u(t)'''(x)] = \dot{u}(t)'''(x)$ for almost every $t \in (0,T)$. Choosing $N := \bigcup_{T \in \mathbb{N}} N_T$ we find that the claim is really shown once \eqref{eq:molli1} and \eqref{eq:molli2} are verified. To verify $\eqref{eq:molli1}$ we use Fubini-Tonelli's Theorem: 
\begin{align}\label{eq:gegennull}
\int_0^1 \int_0^T | h_{\eta_n(x- \cdot)} (t ) - u(t) ''' (x) &|^2 \; \mathrm{d}t \; \mathrm{d}x \nonumber \\ & = \int_0^T \int_0^1 \left\vert \int_0^1 \eta_n(x-r) u(t)'''(r)  \; \mathrm{d}r- u(t) '''(x) \right\vert^2 \; \mathrm{d}x \; \mathrm{d}t \nonumber 
\\ & = \int_0^T || u(t) ''' * \eta_n - u(t)'''||^2_{L^2(0,1)} \; \mathrm{d}t 
\end{align}
Since $u(t) ''' $ is an $L^2$-function and $(\eta_n)_{n \in \mathbb{N}}$ is a sequence of standard mollifiers, the integrand goes to zero pointwise as $n \rightarrow \infty$. Note also that the integrand is uniformly bounded in $t$ since for each $t \in (0,T)$ we can estimate, using \cite[Proof of Theorem 6, Appendix C]{Evans},
\begin{equation*}
||u(t)'''* \eta_n - u(t) '''||_{L^2} \leq 2 ||u(t)'''||_{L^2} \leq 2 \; \mathrm{esssup}_{a \in [0,T]} ||u(a)'''||_{L^2} < \infty,
\end{equation*}
since $u(t) \in L^\infty((0,T), W^{3,2}(0,1))$ by Proposition \ref{prop:existence}. By the dominated convergence theorem, we conclude that the expression in \eqref{eq:gegennull} tends to zero as $n \rightarrow \infty$. Finally, \eqref{eq:molli2} can be verified with similar techniques applying the formula for $\partial_t h_\eta$  found in \eqref{eq:haeta}. 
\end{proof}
\begin{prop}[Regularity of the Gradient and Dynamics in $H'$]\label{prop:dynH'}

Let $u_0  \in W^{3,\infty}([0,1]) \cap H$ be such that $u_0''(0) = u_0''(1)= 0$. Let $(u(t))_{t \geq 0 }$ be the Obstacle Gradient Flow with initial data $u_0$. Then 
\begin{equation*}
\nabla \E(u(t) ) \in W^{3,2}(0,1) , \nabla\E(u(t))''(0) = \nabla \E(u(t))''(1) = 0 \quad \mathrm{a.e.} \; t > 0 
\end{equation*} 
and for almost every $t > 0$ there exists  a finite Radon measure $\mu_t $ on $(0,1)$ such that   
 for all $\phi \in W_0^{1,2}(0,1) \cap W^{2,2}(0,1)$ 
\begin{equation*}
(\dot{u}(t) , \phi) + ( \nabla \E(u(t)), \phi ) = \int \phi \; \mathrm{d}\mu_t  .
\end{equation*} 
Additionally, for each $t> 0 $ there exists a null set $N_t$ such that for each $x \in (0,1)\setminus N_t$    
\begin{equation}\label{eq:gradientdarst}
\nabla \E(u(t)) '''(x) = - \dot{u}(t)'''(x) + \int_0^1 \mu_t([s,1]) \; \mathrm{d}s - \mu_t([x,1]). 
\end{equation}
\end{prop}
\begin{proof} Let $t > 0$ be such that $(FVI)$ holds.  
 First one can infer from the $(FVI)$ and \cite[Corollary 1, Section 1.8]{EvansGariepy} that there exists a Radon measure $\mu_t$ supported on $\{u(t) = \psi \}$ on $(0,1)$ such that for all $\phi \in C_0^\infty(0,1)$  
 \begin{equation}\label{eq:mass}
 (\dot{u}(t) , \phi) + ( \nabla \E(u(t)), \phi) = \int \phi \; \mathrm{d}\mu_t .  
 \end{equation}
 Notice that $\mu_t $ is finite since $u(t) (0) = 0 > \psi(0)$ and $u(t)(1) = 0 > \psi(1)$ and therefore $\mathrm{supp}(\mu_t)$ is compactly contained in $(0,1)$ because of continuity of $u(t) - \psi$. Hence we can derive like in \eqref{eq:vertfkt} that there is $m_t := \mu_t([\cdot,1]) \in L^\infty(0,1)$ such that for all $\phi \in C_0^\infty(0,1)$ 
 \begin{equation*}
(\dot{u}(t) , \phi) + ( \nabla \E(u(t)), \phi)  = \int \phi' m_t \; \mathrm{d}x , 
 \end{equation*} more explicitly 
 \begin{equation*}
\int_0^1 \dot{u}(t)'' \phi'' \; \mathrm{d}x + \int_0^1 \nabla \E(u(t)) '' \phi'' \; \mathrm{d}x =  \int \phi' m_t \; \mathrm{d}x . 
 \end{equation*}
 Integrating by parts (which is justified by Proposition \ref{prop:existence}) we obtain for all $\phi \in C_0^\infty(0,1)$ 
 \begin{equation*}
 \int_0^1 \nabla \E(u(t))'' \phi'' \; \mathrm{d}x = \int_0^1 \phi' m_t \; \mathrm{d}x + \int_0^1 \dot{u}(t)''' \phi' \; \mathrm{d}x . 
 \end{equation*}
This implies, see \cite[Lemma 13.1]{Hestenes} that $\nabla \E(u(t)) \in W^{3,2}(0,1)$ and 
\begin{equation}\label{eq:Ct} 
\nabla\E(u(t)) ''' = C_t - m_t - \dot{u}(t)''' 
\end{equation}
for some $C_t \in \mathbb{R}$.
Because of continuity there exists  $\delta > 0 $ such that $u(t) > \psi$ on $[0, \delta]$. Therefore,  one has $\nabla \E(u(t))''' = C_t - \mu_t((0,1)) - \dot{u}(t) '''$ on $[0 , \delta ]$. The $(FVI)$ implies that for each $\phi \in C_0^\infty([0, \delta))$ such that $\phi(0) = 0, \phi'(0) \neq 0 $ 
 \begin{equation*}
 \int_0^1 \dot{u}(t)'' \phi'' \; \mathrm{d}x + \int_0^1 \nabla \E(u(t))'' \phi'' \; \mathrm{d}x = 0.
 \end{equation*}
 Integrating by parts  and using that $\dot{u}(t)''(0) =0 $ one finds that 
 \begin{equation*}
 \phi'(0) \nabla \E(u(t))''(0) =  \int_0^1 ( \dot{u}(t)''' + \nabla\E(u(t))''') \phi' \; \mathrm{d}x = \int_0^1 (C_t - \mu_t(0,1)) \phi' \; \mathrm{d}x = 0 
 \end{equation*}
We infer that $\nabla\E(u(t))''(0) = 0$. Similarly one proves that $\nabla \E(u(t))''(1) =0$. Using this and integrating over \eqref{eq:Ct} one obtains 
\begin{equation*}
 0 = C_t - \int_0^1 m_t \; \mathrm{d}s \quad \Rightarrow \quad 
C_t = \int_0^1 \mu_t([s,1] ) \; \mathrm{d}s.
\end{equation*}
This proves \eqref{eq:gradientdarst}. Integrating by parts in \eqref{eq:mass} one finds that for all $\phi \in C_0^\infty(0,1)$ 
\begin{equation}\label{eq:formel}
 - \int_0^1 \dot{u}(t)''' \phi' - \int_0^1 \nabla\E(u(t))''' \phi' = \int \phi \; \mathrm{d}\mu_t.
\end{equation}
By density, the same formula holds true for $\phi \in W_0^{1,2}(0,1)$. If we fix $\phi \in W^{2,2}(0,1) \cap W_0^{1,2}(0,1)$ then we can use integration by parts again in \eqref{eq:formel} and  obtain with $\dot{u}(t)''(0) = \dot{u}(t)''(1) = \nabla \E(u(t))''(0) = \nabla \E(u(t))''(1) = 0$ that
\begin{equation*}
\int_0^1 \dot{u}(t)'' \phi'' + \int_0^1 \nabla \E(u(t)) '' \phi'' = \int \phi \; \mathrm{d}\mu_t,
\end{equation*}
as claimed in the statement. 
\end{proof}

\section{Application: Elastic Flow with obstacle constraint} 
In this section we examine long-time behavior of the gradient flow of the elastic energy given in \eqref{eq:elasenerg} in the same framework as in Definition \ref{def:navier}. To do so, we will need to assume slightly stronger conditions on the obstacle, namely that it is $C^1$ in neighborhoods of $0$ and $1$. The main theorem predicts two possible behaviors for the flow for large times: Either one has convergence to a critical point in the sense of a solution of the time-independent variational inequality or one has blow-up of the $L^\infty$-norm of the first derivative. If the obstacle is too large the second case can actually occur, as we will show.

\subsection{Classification of the Asymptotic Behavior}

%
%
\begin{remark} 
The elastic energy, which will be denoted by $\E$ in this section, is one of the energies studied in Section \ref{Sec:Navier}, see Definition \ref{def:navier} with 
\begin{equation}\label{eq:G}
K \equiv  0, \quad G(z) := \int_0^z \frac{1}{(1+w^2)^\frac{5}{4}} \; \mathrm{d}w.
\end{equation}
Clearly, $K$ and $G$ satisfy the assumptions in Definition \ref{def:navier} and therefore we can use the results of Section \ref{Sec:Navier}. Recall in particular that $H$ and $C$ are given by Definition \ref{def:navier}.
\end{remark}
\begin{prop}[Gradient Formula and Estimate]
For each $u, \phi \in H$  one has  
\begin{equation}\label{eq:gradela}
(\nabla\E(u), \phi) = 2 \int_0^1 \frac{ u''\phi''}{(1+ u'^2)^\frac{5}{2}} \; \mathrm{d}x - 5 \int_0^1 \frac{u''^2 u' \phi'}{(1+ u'^2)^\frac{7}{2}} \; \mathrm{d}x.
\end{equation}
In particular, 
\begin{equation}\label{eq:gradnorm}
||\nabla \E(u)|| \leq \left(1 + \frac{5}{2}C_P \right) \E(u) + 1,
\end{equation}
where $C_P$ denotes the operator norm of the embedding $ W^{2,2}(0,1) \cap W_0^{1,2}(0,1) \hookrightarrow W^{1,\infty}(0,1) \cap W_0^{1,2}(0,1)$. 
\end{prop}
\begin{proof}
The formula for the gradient \eqref{eq:gradela} can be shown by a very easy computation, see \cite[Equation 1.5]{Anna}. Using $2ab \leq a^2 + b^2$ for each $a,b \in \mathbb{R}$ we find
\begin{align*}
||\nabla\E(u)|| & = \sup_{ \phi \in H, ||\phi|| \leq 1 } ( \nabla \E(u) , \phi) \\ 
& = \sup_{||\phi||\leq 1} \int_0^1 \frac{2 u'' \phi''}{(1+u'^2)^\frac{5}{2}} \; \mathrm{d}x - 5 \int_0^1 \frac{u''^2 u' \phi'}{(1 + u'^2)^\frac{7}{2}} \; \mathrm{d}x
\\ & \leq \sup_{|| \phi|| \leq 1 } \int_0^1 \frac{u''^2 + \phi''^2}{(1+ u'^2)^\frac{5}{2}} \; \mathrm{d}x + \frac{5}{2}||\phi'||_\infty \int_0^1 \frac{u''^2 (1+u'^2)}{(1+ u'^2)^\frac{7}{2}}  \; \mathrm{d}x
\\ & \leq \sup_{||\phi|| \leq 1 }  \left( 1 +  \frac{5}{2} ||\phi'||_\infty \right) \int_0^1 \frac{u''^2}{(1+ u'^2)^\frac{5}{2}} \; \mathrm{d}x  + \int_0^1 \phi''^2 \; \mathrm{d}x  
\end{align*}
Estimating $||\phi'||_\infty $ by $C_P || \phi ||$ and using $\int_0^1 \phi''^2 \; \mathrm{d}x = ||\phi||^2 \leq 1$ one obtains the claim.  
\end{proof}
\begin{prop}[Evolution of Third Derivatives]
Let $H,C, \psi$ be as in Definition \ref{def:navier} and $\E$ be given by \eqref{eq:elasenerg}. Let $u_0 \in W^{3,\infty}(0,1) \cap H$ be such that $u_0''(0) = u_0''(1) = 0 $ and let $(u(t))_{ t \geq 0 } $ be the Obstacle Gradient Flow for $C$ starting at $u_0$. Then 
\begin{equation}\label{eq:thirdder}
\nabla \E(u(t) ) '''(x) = \frac{2 u(t) '''(x) }{(1 + u(t)'(x)^2)^\frac{5}{2}} - 5 \frac{u(t)''(x)^2 u(t)'(x)}{(1+ u(t)'(x)^2)^\frac{7}{2}} - 5 \int_0^1 \frac{u(t)''(r)^2 u(t)'(r)}{(1+ u(t)'(r)^2)^\frac{7}{2}} \; \mathrm{d}r 
\end{equation}  
for almost every $t>0 $. 
\end{prop} 
\begin{proof}

If $\phi \in C_0^\infty(0,1)$ and  $u \in W^{3,2}(0,1)\cap H $ is fixed such that $u''(0) = u''(1) = 0 $ one can integrate by parts in \eqref{eq:gradela}
\begin{equation*}
\int_0^1 \nabla \E(u)'' \phi'' \; \mathrm{d}x = - \int_0^1 \frac{2u''' \phi'}{(1+u'^2)^\frac{5}{2}} \; \mathrm{d}x + 5 \int_0^1 \frac{u''^2 u' \phi'}{(1+ u'^2)^\frac{7}{2}} \; \mathrm{d}x.
\end{equation*} 
Therefore, for each $t> 0$ there exists some $C_t \in \mathbb{R}$
\begin{equation*}
\nabla \E(u(t))''' = C_t + 2 \frac{u(t)'''}{(1+u(t)'^2)^\frac{5}{2}}- 5 \frac{u(t)''^2 u(t)' }{(1+ u(t)'^2)^\frac{7}{2}}  .
\end{equation*}
Now note that $\nabla \E(u(t))'' (0) = \nabla \E(u(t)) '' (1) = 0 $ for almost every $t>0$ and for such $t$ one has 
\begin{align*}
0 & = C_t + 2 \int_0^1 \frac{u(t)'''(r) }{(1 + u(t)'(r)^2)^\frac{5}{2}} \; \mathrm{d}r- 5 \int_0^1 \frac{u(t)''(r)^2 u(t)'(r)}{(1+ u(t)'(r)^2)^\frac{7}{2}}\; \mathrm{d}r
\\ & = C_t + \int_0^1 \frac{d}{dr} \frac{2 u(t)''(r)}{(1+ u(t)'(r)^2)^\frac{5}{2}} +  \frac{5 u(t)''(r)^2 u(t)'(r)}{(1+ u(t)'(r)^2)^\frac{7}{2}} .
\end{align*}
Integrating and using that $u(t) '' (0) = u(t)''(1) = 0 $ one obtains
\begin{equation*}
C_t = - 5 \int_0^1 \frac{u(t)''(r)^2 u(t)'(r)}{(1+ u'(t) (r)^2)^\frac{7}{2}}\; \mathrm{d}r
\end{equation*}
and \eqref{eq:thirdder} follows.  
\end{proof}
\begin{lemma}[A linear ODE of weakly differentiable functions, Proof in Appendix \ref{sec:App}] \label{lem:ODE}
Suppose that $f \in W^{1,1}(a,b)$ and $\alpha \in C([a,b])$, $\beta \in L^1(a,b)$ are such that 
\begin{equation}\label{eq:ODE}
f'(t) =\alpha(t) f(t) + \beta(t) \quad a.e. \; t \in (a,b).
\end{equation}
Then 
\begin{equation*}
f(t) = \exp \left( \int_a^t \alpha(s) \; \mathrm{d}s \right)  f(a) + \int_a^t \exp  \left(  \int_s^t \alpha(w) \; \mathrm{d}w \right) \beta(s) \; \mathrm{d}s, \quad \mathrm{a.e.} \; t \in (a,b) . 
\end{equation*}
\end{lemma}

\begin{prop}[Global $W^{3,\infty}$-estimates] \label{prop:linf}

Let $u_0 \in  W^{3,\infty}(0,1) \cap H$ be such that $u_0''(0) = u_0''(1) = 0 $ and $(u(t))_{t \geq 0}$ be an Obstacle Gradient Flow for $\E$ with inital datum $u_0$. Assume further that there is $\delta > 0 $ such that for all $t \in (0,\infty)$, $\{ u(t) = \psi\} \subset (\delta, 1- \delta)$. Then there is $C = C(\delta, ||u_0'''||_\infty, \E(u_0) ) $ such that  
\begin{equation*}
||u(t)'''||_{L^\infty} \leq C ( 1 + \sup_{s \in [0,t]} ||u(s)'||_\infty^2)^\frac{15}{2}  \quad \mathrm{a.e. } \;  t\in (0, \infty) .
\end{equation*}
\end{prop}
\begin{proof}
Adopting the notation from Proposition \ref{prop:dynH'} we define for $x \in (0,1)$ 
\begin{equation*}
M_t(x) := \int_0^1 \mu_t([s,1]) \; \mathrm{d}s - \mu_t([x,1]) .
\end{equation*}
Fix $T> 0 $. \textbf{ We show that for each $x \in (0,1)$, $t \mapsto M_t(x) \in L^\infty(0,\infty)$ and $\sup_{x \in (0,1)} ||M_{(\cdot)}(x)||_{L^\infty(0,\infty)} < \infty$}. For this we have to show Lebesgue measurability and boundedness. For the measurability observe that for each $\phi \in C_0^\infty(0,1)$ 
\begin{equation*}
(0,T) \ni t \mapsto \int_0^1 \phi \; \mathrm{d}\mu_t 
\end{equation*}
is measurable. Note that this map is only defined for almost every $t> 0$, but this does not affect the measurability claim as null sets are Lebesgue measurable. Indeed, 
\begin{equation*}
\int_0^1 \phi \; \mathrm{d}\mu_t = (\dot{u}(t) , \phi) + ( \nabla \E(u(t)), \phi ) 
\end{equation*}
is measurable in $t$ as $(0, \infty) \ni t \mapsto \dot{u}(t) \in H $ is Bochner measurable and hence also weakly measurable. The second summand is continuous and therefore also measurable. Now fix $x \in (0,1)$ and choose $\phi_n \subset C_0^\infty(0,1)$ such that $\phi_n \rightarrow \chi_{[x,1-\frac{\delta}{2}]}$ in $C^0([0,1])$. Here we use the convention that intervals are by definition empty if the lower bound exceeds the upper bound. Since $\mathrm{supp}(\mu_t) \subset [\delta, 1-\delta]$   
\begin{equation*}
\mu_t([x,1]) = \mu_t([x, 1- \tfrac{\delta}{2}]) = \lim_{n \rightarrow \infty} \int_0^1 \phi_n \; \mathrm{d}\mu_t
\end{equation*}
and so $t \mapsto \mu_t([x,1])$ is measurable as pointwise limit of measurable functions. It remains to show that 
\begin{equation*}
t \mapsto \int_0^1 \mu_t([s,1]) \; \mathrm{d}s  
\end{equation*}
is measurable. But this follows from $\mu_t([\cdot, 1])$ being monotone  for each $t$ and therefore Riemann integrable. Hence
 \begin{equation}\label{eq:mear}
 \int_0^1 \mu_t([s,1]) \; \mathrm{d}s = \lim_{k \rightarrow \infty} \frac{1}{k} \sum_{l = 1}^{k} \mu_t \left( \left[\nicefrac{l}{k},1\right] \right) 
\end{equation}   
and therefore it is measurable in $t$ as pointwise limit of linear combinations of measurable functions. \textbf{We continue showing that this map is essentially uniformly bounded in $t$}. For this fix $\eta \in C_0^\infty(0,1)$ such that $ \eta \equiv 1$ on $[\delta, 1- \delta]$. Then using Corollary \ref{cor:boundder} and \eqref{eq:gradnorm} we obtain 
\begin{align}\label{eq:defA}
\mu_t((0,1)) & = \int_0^1 \eta \; \mathrm{d}\mu_t = ( \dot{u}(t) , \eta ) + ( \nabla \E(u(t)), \eta) \leq ||\eta|| ( ||\dot{u}(t) || + || \nabla \E(u(t)) || ) \nonumber  \\ &  \leq 2 ||\eta || \; || \nabla \E(u(t)) || \leq 2||\eta||\left(  \left( 1 + \frac{5}{2}C_p \right) \E(u(t) ) + 1 \right) \nonumber \\ & \leq 2 || \eta ||  \left( \left( 1 + \frac{5}{2}C_p \right) \E(u_0 ) + 1\right) =: J.
\end{align}  
We conclude that $t \mapsto M_t(x) $ defines an $L^\infty(0, \infty)$ function and an upper bound for the $L^\infty$ norms can be chosen independently of $x$. The intermediate claim follows.
Let $N \subset (0,1)$ be the null set of Corollary \ref{cor:nullset}. Define for $x \in (0,1) $ and $h > 0 $ 
\begin{align*}
f_{h,x}(t) & := \frac{1}{h} \int_x^{x+h} u(t)'''(r) \; \mathrm{d}r, \\
A(s)  &:= 5 \int_0^1 \frac{u(s)''(l)^2u(s)'(l)}{(1 + u(s)'(l)^2)^\frac{7}{2}} \; \mathrm{d}l.
\end{align*}
Estimating $|u(s)'(l)| \leq \frac{1}{2} (1 + u(s)'(l)^2)$ in the numerator one obtains that $|A(s)| \leq 5 \E(u(s)) \leq 5 \E(u_0)$. 
Now we compute using the definition of $N$, Fubini's Theorem (which can easily be justified) and \eqref{eq:gradientdarst} as well as \eqref{eq:thirdder} 
\begin{align*}
f_{h,x}(t)  & = \frac{1}{h} \int_x^{x+h} u(t)'''(r) \; \mathrm{d}r  = \frac{1}{h} \int_{[x,x+h] \setminus N } u(t)'''(r) \; \mathrm{d}r 
\\ & = \frac{1}{h} \int_{[x,x+h] \setminus N } \left( \int_0^t \dot{u}(s)'''(r) \; \mathrm{d}s + u_0'''(r) \right) \; \mathrm{d}r 
\\ & = \frac{1}{h} \int_{x}^{x+h} u_0'''(r) \; \mathrm{d}r + \int_0^t  \frac{1}{h} \int_x^{x+h} \dot{u}(s)'''(r) \; \mathrm{d}r \; \mathrm{d}s 
\\ & = \frac{1}{h} \int_x^{x+h} u_0'''(r) \; \mathrm{d}r +\frac{1}{h} \int_0^t \int_{x}^{x+h} - \nabla \E(u(s))'''(r) + M_s(r) \; \mathrm{d}r \; \mathrm{d}s  
\\ & = \frac{1}{h} \int_x^{x+h} u_0'''(r) \; \mathrm{d}r \\ & \quad  + \frac{1}{h}\int_0^t \int_{x}^{x+h} \frac{-2 u(s)'''(r)}{(1 + u(s)'(r)^2)^\frac{5}{2}} + 5 \frac{u(s)''(r)^2 u(s)'(r)}{(1+ u(s)'(r)^2)^\frac{7}{2}}+ (A(s) +M_s(r) )  \; \mathrm{d}r \; \mathrm{d}s. 
\end{align*}
Define for $s \in [0,T]$ 
\begin{align*}
\theta_{h,x}^1 (s )  & := A(s) + \frac{1}{h}\int_x^{x+h} M_s(r) \; \mathrm{d}r, \\
\theta_{h,x}^2(s)  & := \frac{1}{h} \int_x^{x+h} 2 u(s)'''(r) \left( \frac{1}{(1+u'(s)(r)^2)^\frac{5}{2}} - \frac{1}{(1+ u'(s)(x)^2)^\frac{5}{2}}\right) \; \mathrm{d}r ,\\
 \theta_{h,x}^3(s) & := \frac{1}{h} \int_x^{x+h} 5 \frac{u(s)''(r)^2 u(s)'(r)}{(1+ u(s)'(r)^2)^\frac{7}{2}} \; \mathrm{d}r.
 \end{align*}
Then 
\begin{equation*}
f_{h,x}(t) = \frac{1}{h} \int_{x}^{x+h} u_0'''(r) \; \mathrm{d}r + \int_0^t \frac{-2}{(1+ u'(s)(x)^2)^\frac{5}{2}} f_{h,x}(s) \; \mathrm{d}s + \int_0^t \left( \sum_{i = 1}^{3} \theta_{h,x}^i(s)  \right) \; \mathrm{d}s.
\end{equation*}
One infers that $f_{h,x} \in W^{1,1}(0,T)$ and satisfies a differential equation like \eqref{eq:ODE} with 
\begin{equation*}
\alpha(t) := - \frac{2}{(1+ u(t)'(x)^2)^\frac{5}{2}}, \qquad
\beta(t) :=\sum_{i = 1}^{3} \theta_{h,x}^i(s)  .
\end{equation*}
Now observe that for $s \in [0,T] $ 
\begin{equation}\label{eq:thetaest1}
|\theta_{h,x}^1(t)| = |A(t)| + \frac{1}{h}\int_{x}^{x+h} M_s(r) \; \mathrm{d}r \leq 5 \E(u_0) + 2J  =: J^* ,
\end{equation}
where $J$ is given in \eqref{eq:defA}. Furthermore,
\begin{align}\label{eq:thetaest2}
|\theta_{h,x}^2(s) | & \leq \frac{2}{h} \int_{x}^{x+h} |u(s)'''(r)| \left( \int_{x}^{x+h} \frac{5 |u(s)''(\rho)|\; |u(s)'(\rho)|}{(1+ u(s)'(\rho)^2)^\frac{7}{2}} d\rho \right) \; \mathrm{d}r \nonumber
\\ & = \frac{10}{h} \int_{[x,x+h]^2} \frac{|u(s)'''(r) | \; |u(s)''(\rho)| \; | u(s)'(\rho)|}{(1 + u(s)'(\rho)^2 )^\frac{7}{2}} d(r, \rho) \nonumber
\\ & \leq 10 \sqrt{ \int_{[x,x+h]^2} \frac{|u(s)'''(r) |^2 \; |u(s)''(\rho)|^2 \; | u(s)'(\rho)|^2}{(1 + u(s)'(\rho)^2 )^7} d(r, \rho)} \nonumber
\\ & \leq 10 \sqrt{\int_0^1 |u(s)'''(r)|^2 \; \mathrm{d}r} \sqrt{\int_{x}^{x+h}  \frac{  |u(s)''(\rho)|^2 \; | u(s)'(\rho)|^2}{(1 + u(s)'(\rho)^2 )^7} } \nonumber \\ & 
\leq 10 \sqrt{h} \sup_{ \widetilde{t} \in [0,T]} ||u(\widetilde{t}) '''||_{L^2} \sup_{ \widetilde{t} \in [0,T] } ||u(\widetilde{t}) '' ||_{L^\infty}   .
\end{align}
Using Lemma \ref{lem:ODE}, one finds that 
\begin{align*}
f_{h,x}(t) &  = \exp\left( - \int_0^t \frac{2}{(1+ u(s)'(x)^2)^\frac{5}{2}}\right) f_{h,x} (0)  \\ & \qquad  + \int_0^t \exp\left( - \int_s^t \frac{2}{(1+ u(w)'(x)^2)^\frac{5}{2}} \; \mathrm{d}w \right) \left( \sum_{i = 1}^{3} \theta_{h,x}^i(s) \right) \; \mathrm{d}s 
\end{align*}
and therefore, using 
$
f_{h,x}(0) = \frac{1}{h} \int_x^{x+h} u_0'''(r) \; \mathrm{d}r  ,
$
and \eqref{eq:thetaest1}, \eqref{eq:thetaest2} one has
\begin{align*}
& |f_{h,x}(t)|  \leq \frac{1}{h} \exp\left( - \int_0^t \frac{2}{(1+ u(s)'(x))^2 )^\frac{5}{2}}\; \mathrm{d}s\right) \int_x^{x+h} |u_0'''(r)| \; \mathrm{d}r  \\ & \; \;   + (J^* + 10 \sup_{ \widetilde{t} \in [0,T] } || u(\widetilde{t})'''||_{L^2}||u(\widetilde{t}) '' ||_{L^\infty}  \sqrt{h }  )   \int_0^t  \exp\left( - \int_s^t \frac{1}{(1 + u(w)'(x)^2)^\frac{5}{2} }   \; \mathrm{d}w \right)  \; \mathrm{d}s .
\\ & \;  \;  + \int_0^t  \exp\left( - \int_s^t \frac{1}{(1 + u(w)'(x)^2)^\frac{5}{2} }   \; \mathrm{d}w \right)|\theta_{h,x}^3(s)|  \; \mathrm{d}s.
\end{align*}
Observe that for each $x \in (0,1)$ and $s \in (0,1)$ 
\begin{equation*}
|\theta_{h,x}^3(s) | \leq 5 ||u(s)''||^2_{L^\infty} \leq  5\sup_{\widetilde{t} \in [0,T]}||u ( \widetilde{t}) ''||^2_{L^\infty}  
\end{equation*}
and because of continuity of the integrand 
\begin{equation*}
\lim_{h \rightarrow 0} |\theta_{h,x }^3 (s) | = 5 \frac{u(s)''(x)^2 |u(s)'(x)|}{(1+ u(s)'(x)^2)^\frac{7}{2}} \quad \forall x \in (0,1) \; \forall s \in (0,T) .
\end{equation*}
For the following we define 
$
S:= (1+ \sup_{\widetilde{t} \in [0,T]} ||u'(\widetilde{t})||_\infty^2 )^\frac{5}{2}.
$
Now the Lebesgue differentiation theorem and the dominated convergence theorem imply that for almost every $x \in (0,1)$
\begin{align}\label{eq:linfl1}
|u(t)'''(x)| & \leq \liminf_{h \rightarrow 0 } |f_{h,x}(t)| \nonumber \\ & \leq ||u_0'''||_\infty + J^* \int_0^t \exp \left(- \int_{s}^t \frac{1}{(1+ u(w)'(x)^2)^\frac{5}{2} }  \; \mathrm{d}w\right) \; \mathrm{d}s  \nonumber \\
& \quad + \int_0^t \exp \left(- \int_{s}^t \frac{1}{(1+ u(w)'(x)^2)^\frac{5}{2} }  \; \mathrm{d}w\right)\frac{5 u(s)''(x)^2 |u(s)'(x)|}{(1+ u(s)'(x)^2)^\frac{7}{2}} \; \mathrm{d}s 
\nonumber \\
& \leq  ||u_0'''||_\infty + J^* \int_0^t e^{-\frac{(t-s)}{S} } \; \mathrm{d}s+ 5 \int_0^t e^{-\frac{t-s}{S}} \frac{u(s)''(x)^2|u(s)'(x)|}{(1+ u(s)'(x)^2)^\frac{7}{2}}   \; \mathrm{d}s 
\nonumber \\ & \leq ||u_0'''||_\infty + J^* S + 5 \int_0^t e^{-\frac{t-s}{S}}  \frac{u(s)''(x)^2|u(s)'(x)|}{(1+ u(s)'(x)^2)^\frac{7}{2}} \; \mathrm{d}s .
 \end{align}
Integrating over $x$ and using Fubini's Theorem we obtain 
\begin{equation}
\int_0^1 |u(t)'''(x)| \; \mathrm{d}x \leq ||u_0'''||_\infty + J^* S +5 \int_0^t e^\frac{t-s}{S} \left( \int_0^1 \frac{u(s)''(x)^2|u(s)'(x)|}{(1+ u(s)'(x)^2)^\frac{7}{2}} \; \mathrm{d}x \right) \; \mathrm{d}s.
\end{equation} 
Estimating $|u(s)'(x)| \leq \frac{1}{2}( 1 + |u(s)'(x)|^2) $ and $1 \leq S$ we obtain 
\begin{align*}
||u(t)'''||_{L^1} & \leq ||u_0'''||_\infty + J^* S + \frac{5}{2} \int_0^t e^\frac{t-s}{S} \E(u(s)) \; \mathrm{d}s 
\\ & \leq  ||u_0'''||_\infty + J^* S + \frac{5}{2} \E(u_0) S 
\leq \left( ||u_0'''||_\infty + \frac{5}{2}\E(u_0) + J^* \right)S.
\end{align*}
Therefore, we have bounded $||u(t)'''||_{L^1}$ in terms of $S$. Since $u(t)''(0) = 0$ we find 
\begin{equation*}
||u(t)''||_{L^\infty} \leq \left( ||u_0'''||_\infty + \frac{5}{2}\E(u_0) + J^* \right)S.
\end{equation*}
Using this we can go back to \eqref{eq:linfl1} and find for almost every $x \in (0,1)$ 
\begin{align*}
|u(t)'''(x)| & \leq ||u_0'''||_\infty + J^*S + 5 \sup_{\widetilde{t} \in [0,t] }  ||u(\widetilde{t})''||^2_{L^\infty} \int_0^t e^\frac{t-s}{S} \; \mathrm{d}s 
\\ & \leq ||u_0'''||_\infty + J^* S + 5 S^3 \left( ||u_0'''||_\infty + \frac{5}{2} \E(u_0) + J^* \right)^2 \leq C S^3   
\end{align*}
for some $C =C(\delta, ||u_0'''||_\infty, \E(u_0) ) $. The claim follows. 
\end{proof}
\begin{remark}
The measurability discussion in \eqref{eq:mear} was really necessary the way it is presented in this article, see \cite[Discussion below Theorem 4.48]{Aliprantis} for details.  
\end{remark}
\begin{remark}
The $L^\infty$-estimate for the third derivative holds only true provided that the contact set does not come arbitrarily close to $0$ or $1$. For the elastic energy, one finds mild conditions on the obstacle which ensure exactly that, as we will discuss in the next two propositions. 
\end{remark}
\begin{prop}[A Standard Estimate for the Elastic Energy]\label{prop:cauch} 
 
Let $u \in H$ be arbitrary and let $G$ be the function defined in \eqref{eq:G}.
Then 
\begin{equation}\label{eq:cauch}
\E(u) \geq \sup_{x_1, x_2 \in (0,1), x_1 \neq x_2 }  \frac{(G(u'(x_1))- G(u'(x_2)))^2}{x_1 -x_2}.
\end{equation}
\end{prop}
\begin{proof}
Suppose that $x_1,x_2 \in (0,1)$ is such that $x_1 \neq x_2$. 
Observe that 
\begin{equation*}
\E(u) = \int_0^1 G'(u'(x))^2 u''(x)^2  \; \mathrm{d}x \geq \int_{x_1}^{x_2} G'(u'(x))^2 u''(x)^2 \; \mathrm{d}x
\end{equation*}
the claim follows easily using the Cauchy Schwarz inequality and the fundamental theorem of calculus. 
\end{proof}
\begin{remark}\label{rem:interestin}
For the rest of the article we define 
\begin{equation*}
c_0 := 2 ||G||_{\infty; \mathbb{R}} = 2 \int_0^\infty \frac{1}{(1+s^2)^\frac{5}{4}} \; \mathrm{d}s ,
\end{equation*}
where $G$ is given in \eqref{eq:G}. For the study of the time-independent problem, $c_0$ is an important constant since 
$
 \inf_{ u\in C} \E(u) \leq c_0^2, 
$
see \cite[Lemma 2.4]{Anna}. Note also that $G: \mathbb{R} \rightarrow \left( -\frac{c_0}{2}, \frac{c_0}{2} \right)$ is a diffeomorphism.  
\end{remark}
\begin{prop}[Trapping the Contact Set]

Suppose that $\psi \in C^0([0,1])$ is such that $\psi(0), \psi(1) < 0 $ and there is $\eta > 0 $ such that $\psi \in C^1([0, \eta]) \cap C^1([1- \eta, 1])$. Let $u_0 \in W^{3,\infty}(0,1)$ be such that $u_0''(0) = u_0''(1) = 0 $ and let $(u(t))_{t\geq 0} $ be the Obstacle Gradient Flow starting at $u_0$.  Then there is $\delta > 0 $  such that  $\{u(t) = \psi \} \subset [\delta, 1-\delta]$ for each $t > 0 $.
\end{prop}
\begin{proof}
Suppose that there is a sequence $\delta_n \rightarrow 0 $ and $t_n \subset [0, \infty)$  such that $u(t_n) (\delta_n) = \psi(\delta_n)$. Without loss of generality, one can assume that $\delta_n < \eta $ for each $n \in \mathbb{N}$. Note that since $u(t_n)-\psi$ attains its global minimum at $x = \delta_n$, we obtain $u(t_n)'(\delta_n) = \psi'(\delta_n)$ for each $n \in \mathbb{N}$.
Further, let $G$ be defined as in \eqref{eq:G}.
 Observe that then 
\begin{align*}
\E(u_0) & \geq \E(u(t_n)) \geq \int_0^{\delta_n} \frac{u(t_n)''(x)^2}{(1 + u(t_n)'(x)^2)^\frac{5}{2}} \; \mathrm{d}x \\ & \geq \frac{1}{\delta_n} \sup_{ \eta \in (0, \delta_n) }(  G(u(t_n)'(\eta)) - G( u(t_n)'(\delta_n) ) )^2 \\ &  = \frac{1}{\delta_n}  \sup_{ \eta \in (0, \delta_n) }(G(u(t_n)'(\eta))- G( \psi'(\delta_n)) )^2 . 
\end{align*}
Since $\delta_n \rightarrow 0$, we find that 
\begin{equation*}
\sup_{\eta \in (0,\delta_n) } |G(u(t_n)'(\eta)) - G(\psi'(\delta_n)) | \rightarrow 0 \quad (n \rightarrow \infty),
\end{equation*}
and thus by continuity of $G$,$\psi'$, 
\begin{equation*}
\limsup_{n \rightarrow \infty } \sup_{ \eta \in (0, \delta_n) }|G(u(t_n)'(\eta)) | \leq |G(\psi'(0))|.  
\end{equation*} 
Monotonicity of $G$ and $|G(w)|= G(|w|)$ implies 
\begin{equation*}
\limsup_{n \rightarrow \infty} G \left( \sup_{\eta \in (0,\delta_n)}| u(t_n)'(\eta)|  \right) \leq  G(|\psi'(0)|)   ,
\end{equation*}
which implies that 
\begin{equation*}
\limsup_{n\rightarrow \infty} \sup_{\eta \in (0,\delta_n)} |u(t_n)'(\eta)|\leq  |\psi'(0)| 
\end{equation*}
because of monotonicity and continuity of $G^{-1}$ in a neighborhood  $[0, G(|\psi'(0)|) + \epsilon]$ for a sufficiently small $\epsilon >0 $. Now observe that 
\begin{align*}
0 & = \lim_{n \rightarrow \infty } u(t_n)(0) = \lim_{n \rightarrow \infty} \left( u(t_n)(\delta_n) - \int_0^{\delta_n} u(t_n)'(\eta) d\eta  \right) \\ &  \leq \liminf_{n \rightarrow \infty} \left( \psi(\delta_n) + \delta_n \sup_{ \eta \in (0, \delta_n)} | u(t_n)'(\eta)| \right) = \psi(0)   ,
\end{align*}
a contradiction since $\psi(0) < 0 $. Therefore there exists $\delta_0 > 0 $ such that for each $t> 0 $,  $u(t)(x) > \psi$ for each $x \in (0,  \delta_0)$. Similarly one can show that there exists $\delta_1 > 0 $ such that $ u(t)(x)  > \psi $ for each $x  \in (1- \delta_1, 1)$. The claim follows choosing $\delta := \min\{ \delta_0, \delta_1 \} $.
\end{proof}

\begin{theorem}[Subconvergence Behavior of the Elastic Flow]\label{thm:subcon}
 
Let $\psi \in C^0([0,1])$ be such that $\psi(0),\psi(1) < 0 $ and there is $\eta > 0$ such that $\psi \in C^1([0, \eta]) \cap C^1([1-\eta ,1])$. Further, let $u_0 \in W^{3,\infty}(0,1) \cap H$ be such that $u_0''(0) = u_0''(1) = 0$. Let $(u(t))_{t \geq 0 }$ be the Obstacle Gradient Flow for the elastic energy with initial data $u_0$. Then one of the following is true: 
\begin{enumerate}
\item (Vertical parts at $\infty$) There is a sequence $\theta_n \rightarrow \infty $ such that 
\begin{equation*}
||u(\theta_n)'||_\infty \rightarrow \infty  
\end{equation*}
\item  (Subconvergence to a critical point) 
There is a sequence $\theta_n \rightarrow \infty$ and $u_\infty \in C$ such that $u(\theta_n) \rightarrow u_\infty $ in $W^{2,2}(0,1)$ and
$$\label{eq:varineq}
(\nabla \E(u_\infty), v- u_\infty) \geq 0 \quad \forall v \in C .\eqno{(\theequation)}
$$
\end{enumerate}
\end{theorem}
\begin{proof}
Suppose that $(1)$ does not hold true. Then $(0, \infty) \ni t \mapsto ||u(t)'||_\infty$ is bounded.  Because of \eqref{eq:abl}, $||\dot{u}(\cdot)|| \in L^2(0,\infty)$ and hence there is a sequence $\theta_n \rightarrow \infty$ such that $||\dot{u}(\theta_n)|| \rightarrow 0 $. Moreover $\theta_n$ can be chosen such that $(FVI)$ and the estimate in Proposition \ref{prop:linf} hold true, since these hold almost everywhere. From Proposition \ref{prop:linf} can be inferred that $ ||u(\theta_n)'''||_{L^\infty}$ is bounded. Since $u(t)''(0)= u(t)''(1) = u(t)(0) = u(t)(1) = 0 $ for all $t >0 $, see Theorem \ref{thm:navier}, we find that $ ||u(\theta_n)||_{W^{3,2}}$ is bounded.
 Therefore there exists a subsequence of $(\theta_n)_{n \in \mathbb{N}}$  which we do not relabel such that $(u(\theta_n))_{n \in \mathbb{N}} $ has a weak limit in $W^{3,2}(0,1)$. Let $u_\infty$ denote this weak limit. Since $W^{3,2}(0,1) \cap C$ is convex and closed in $W^{3,2}(0,1)$, we find that $u_\infty \in W^{3,2}(0,1) \cap C$. The compact embedding  $W^{3,2}(0,1) \cap H \hookrightarrow H$ shows also  that $u(\theta_n) \rightarrow u_\infty$ in $H$. We show finally that $u_\infty$ is also a critical point of the functional. Indeed, let $v \in C$ be arbitrary but fixed. Then 
 \begin{equation*}
  0 \leq \liminf_{n \rightarrow \infty} ( \dot{u}(\theta_n) , v- u(\theta_n) ) + ( \nabla \E(u(\theta_n) ), v- u(\theta_n) )  = (\nabla \E(u_\infty) , v - u_\infty). \qedhere
 \end{equation*}
\end{proof}
\begin{remark}
This behavior is consistent with the behavior of the static problem, see \cite{Mueller} and \cite{Anna}. Here it is shown that minimizers exist either as $W^{2,2}\cap W_0^{1,2}$-graphs or in a slightly larger class, roughly speaking, the class of graphs that may have vertical parts at the boundary.    
\end{remark}

\subsection{Examples for all possible asymptotics}

The rest of this article will be dedicated to the question whether case $(1)$ in Theorem \ref{thm:subcon} really occurs and whether there are conditions that ensure that case $(2)$ occurs. As \cite{Anna} points out, the minimization problem does not possess a solution in the class of symmetric functions for large obstacles.   

\begin{prop}[Symmetry of the Gradient]\label{prop:symgrad}   
Let $u \in H$. Then $\nabla \E(u(1 - \cdot)) = \nabla \E(u) (1- \cdot)$ .
\end{prop}
\begin{proof}
Fix $u \in H$ and let $\phi \in H$ be arbitrary. 
Observe that
\begin{equation*}
 \E(u + t \phi) = \E(u(1 - \cdot) + t \phi( 1- \cdot) ).
 \end{equation*}
One can take the derivative of both sides with respect to $t$ and evaluate at $t = 0$ to find that 
$
(\nabla \E(u) , \phi) = ( \nabla \E(u(1 - \cdot)), \phi(1- \cdot) ) .
$
Writing the right hand side as an integral and using the substitution rule, one obtains
\begin{equation*}
(\nabla \E(u) , \phi) = ( \nabla \E(u(1- \cdot) ) (1- \cdot) , \phi) \quad \forall \phi \in H.
\end{equation*}
Therefore $\nabla \E(u) =  \nabla \E(u(1- \cdot) ) (1- \cdot) $ and the claim follows. 
\end{proof}

\begin{lemma}[Symmetry Preservation]  \label{lem:sympres}
Suppose that $\psi \in C^0([0,1])$ is a symmetric obstacle, i.e. $\psi(x) = \psi(1-x)$ for all $x \in (0,1)$. 
Let $u_0 \in W^{3,\infty}(0,1) \cap H $ be such that $u_0''(0) = u_0''(1) = 0$ and $u_0(x) = u_0(1-x)$ for all $x \in (0,1)$. Let $(u(t))_{ t \geq 0 }$ be an Obstacle Gradient Flow. Then $u(t) (1- \cdot) = u(t) $ for all $t \geq 0 $. 
\end{lemma}
\begin{proof}
We show that $(u(t) (1- \cdot))_{t \geq 0 }$ is also an Obstacle Gradient Flow with intital datum $u_0$. Equality follows then from the uniqueness result in Proposition \ref{prop:Wellposed}. We check the conditions required in Definition \ref{def:coneflow}. Condition $(1)$ follows from the symmetry of $u_0$. For condition $(2)$ observe that for fixed $t \geq 0 $ one has $ u(t)(1- \cdot) \geq \psi(1- \cdot) = \psi$  because of symmetry of $\psi$. Therefore $u(t) (1- \cdot) \in C$ for all $ t \geq 0 $. It is straightforward to check that $u(t)(1-\cdot) \in W^{1,2}((0,T),H)$ with weak time derivative coinciding with $\dot{u}(t)(1- \cdot)$ almost everywhere. To verify $(3)$ we check $(FVI)  $. Let $v \in C$ be arbitrary and $t > 0 $ be such that $(FVI)$ holds. Using the substitution rule and Proposition \ref{prop:symgrad}  
\begin{align*}
& (\dot{u}(t) (1-\cdot) , v- u(t) (1- \cdot) ) + ( \nabla \E(u(t)(1-\cdot) ), v- u(t)(1-\cdot)) \\ 
& = (\dot{u}(t) , v(1- \cdot) - u(t) )  + ( \nabla \E(u(t)(1-\cdot))(1-\cdot) , v (1- \cdot) - u(t) ) 
\\ & = (\dot{u}(t) , v(1- \cdot) - u(t) )  + ( \nabla \E(u(t)(1-\cdot))(1-\cdot) , v (1- \cdot) - u(t) ) 
\\ & = (\dot{u}(t) , v(1- \cdot) - u(t) )  + ( \nabla \E(u(t)) , v (1- \cdot) - u(t) ) \geq 0 ,
\end{align*}
 since $v(1- \cdot ) \in C$ because of symmetry of $\psi$. The claim follows.
\end{proof}

\begin{remark}
The observation that symmetry is preserved by the flow is of special interest in the context of higher order variational problems, where symmetry of minimizers is often difficult to obtain. This is due to the lack of a maximum principle. 
\end{remark}

\begin{cor}[A Condition Ensuring Subconvergence]\label{cor:4.16}
Suppose that $\psi\in C^0 ([0,1]) \cap C^1([0, \eta ]) \cap C^1 ( [1- \eta , 1])$ for some $\eta > 0 $ and $\psi = \psi(1- \cdot)$. Let $u_0 \in W^{3,\infty}(0,1)\cap H$ satisfy $u_0''(0)= u_0''(1)= 0 $ and be symmetric with
$
\E(u_0) < c_0^2.
$
Then the Obstacle Gradient Flow starting at $u_0$ subconverges, that is case $(2)$ in Theorem \ref{thm:subcon} applies. 
\end{cor} 
\begin{proof}
To show that case $(2)$ in Theorem \ref{thm:subcon} applies, we exclude case $(1)$.
Let $(u(t))_{t \geq 0 } $ be the Obstacle Gradient Flow with initial datum $u_0$. Fix $x \in (0,\frac{1}{2})$.  Then using \eqref{eq:cauch} with $x_1 = x, x_2 = 1- x $ we obtain 
\begin{align*}
\E(u_0) &  \geq \E(u(t)) \geq \frac{1}{1-2x} (G(u(t)'(x))- G(u(t)'(1-x)))^2 \\ &  \geq (G(u(t)'(x))- G(u(t)'(1-x)))^2
\end{align*}
Since according to Lemma \ref{lem:sympres} $u(t) = u(t)(1- \cdot)$ we find that $u(t)'(1-x) = - u(t)(1- \cdot)' (x) = - u(t)'(x) $ and hence 
$
 \E(u_0) \geq  4 G(u(t)'(x) )^2 . 
$
Since $G$ is an odd  function, one has
\begin{equation}
G(| u(t)'(x)| ) = |G(u(t)'(x)) | \leq \frac{1}{2}\sqrt{\E(u_0)} < \frac{c_0}{2}
\end{equation} 
In particular, since $G$ is monotone and $G: (-\infty, \infty) \rightarrow ( -\frac{c_0}{2}, \frac{c_0}{2} )$ is a diffeomorphism, see Remark \ref{rem:interestin}, we find 
\begin{equation*}
|u(t)'(x)| \leq G^{-1} \left( \frac{\sqrt{\E(u_0)}}{2} \right) \quad \forall x \in \left[0, \nicefrac{1}{2} \right]  .
\end{equation*}
Since $u(t)'(x) = - u(t)'(1-x)$ we find 
\begin{equation*}
\sup_{t \in [0, \infty) } ||u(t)'||_{L^\infty(0,1)} \leq  G^{-1} \left( \frac{\sqrt{\E(u_0)}}{2} \right),
\end{equation*}
which excludes case $(1)$ in Theorem \ref{thm:subcon}.   
\end{proof}

\begin{remark}\label{rem:funkschar}
\end{remark}
\vspace{-0.63cm}
 \hspace{1.68cm} Of course it could happen that a symmetric function $u_0$ like in Corollary \ref{cor:4.16} does not exist. The following Lemma is to convince the reader that for small obstacles, such $u_0$ exists. The appropriate smallness condition is the  same as in \cite[Lemma 4.2]{Anna}. For the condition, an important comparision function is needed: 
\begin{equation*}
U_0(x) := \begin{cases}
\frac{2}{c_0 \sqrt[4]{1 + G^{-1} ( \frac{c_0}{2} - c_0 x )^2 }} & x \in (0,1), \\ 0  & x = 0,1 .
\end{cases} 
\end{equation*}  
\begin{wrapfigure}{r}{5.55cm}
\centering
\includegraphics[width=5.0cm, height=2.9cm]{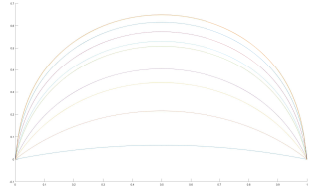}  
\caption{Some $u_c$ for $c=\tfrac{1}{2},\tfrac{3}{2},2.0,2.15,2.3,2.32,2.35,2.37,2.39$}
  \label{fig:amazon_result}
\end{wrapfigure} 
 For details on those see \cite{Djondjorov} and \cite{Deckelnick}.
 The reason why $U_0$ is so important is that $U_0$ is a graph reparametrization of one of Euler's free elastica, i.e. the critical points of  \eqref{eq:willmore}.
In the latter article, symmetric critical points of \eqref{eq:willmore} that possess a graph reparametrization are investigated. The authors show that all symmetric critical graphs are given by a family $(u_c)_{c \in (0,c_0)}$  that approximates $U_0$ from below, see Figure  \ref{fig:amazon_result}.
In particular, no symmetric critical graph can be found above $U_0$.


\begin{lemma}[Existence of Symmetric Graphs with Small Energy] 
Suppose that $\psi \in C^0([0,1])$ is symmetric. Let 
\begin{equation*}
C_{\mathrm{sym,reg}} := \{ w \in W^{3, \infty}(0,1) \cap H : w \geq \psi, w = w(1- \cdot), w''(0) = w''(1) = 0 \} .
\end{equation*}
 Then 
\begin{equation}\label{eq:energung}
\inf_{w \in C_{\mathrm{sym,reg}}  } \E(w) \leq c_0^2 .
\end{equation}
 If additionally $\psi(x) < U_0(x) $ for each $x \in [0,1]$, where $U_0$ is given in Remark \ref{rem:funkschar}, then there exists $v \in C_{\mathrm{sym,reg}}$ such that $\E(v) < c_0^2$. 
\end{lemma}
\begin{proof} The proof uses ideas from \cite[Lemma 2.4]{Anna} and \cite[Lemma 4]{Deckelnick} with just slight modifications. We start with the last part of the claim. To this, suppose that $\psi < U_0$. Note that according to  \cite[Lemma 2.3]{Anna}, $U_0$ is the uniform limit of a family of symmetric functions $(u_c)_{c \in (0,c_0)} \subset C^\infty([0,1]) \cap H $ that is increasing in $c$ and satisfies
$
\E(u_c) = c^2.
$
Because of continuity of $\psi$ and $U_0$ as well as uniform convergence of $(u_c)_{c \in (0,c_0)}$ there exists $\epsilon> 0 $ and $\overline{c} \in (0,c_0)$ such that $\psi(x) + \epsilon < u_{\overline{c}}(x) $ for each $x \in [0,1]$. Now take $\phi_n \subset C_0^\infty(0,1)$ symmetric such that $\phi_n \rightarrow u_{\overline{c}}''$ in $L^2(0,1)$. For $n \in \mathbb{N}$, let $v_n$ denote the unique (and symmetric) solution of 
\begin{equation*}
\begin{cases} 
v_n'' = \phi_n & x \in(0,1), \\ v_n = 0 & x = 0,1.  
\end{cases} 
\end{equation*}
Then $v_n \in H$ since by elliptic regularity  $v_n \in W^{2,2}(0,1) \cap W_0^{1,2}(0,1)$.
Now by definition of the norm in $H$ one has $|| v_n - u_{\overline{c}}||_H = || \phi_n - u_{\overline{c}} '' ||_{L^2} \rightarrow 0 $ and hence $v_n \rightarrow u_{\overline{c}}$ in $W^{2,2}(0,1)$ as $n \rightarrow \infty$. Since $W^{2,2}(0,1)$ embeds into $C^0([0,1])$ the convergence is also uniform and therefore there exists $n_0 \in \mathbb{N}$ such that for each $n \geq n_0$, $v_n > u_{\overline{c}} - \epsilon > \psi$ because of the choice of ${\overline{c}}, \epsilon$. Observe that 
$
\lim_{n \rightarrow \infty} \E(v_n) = \E(u_{\overline{c}}) = {\overline{c}}^2 < c_0^2  .
$
Therefore there exists also some $n_1 \in \mathbb{N}$ stuch that $\E(v_n) < c_0^2$ for each $n \geq n_1$. The claim follows taking $v = v_{\max\{n_0, n_1\}}$, which is $W^{3,\infty}$ and satisfies $v''(0) = v''(1) = 0$ as the second derivative of $v$ is compactly supported. It follows that $v \in C_{\mathrm{sym,reg}}$ which proves the last part of the claim. To show \eqref{eq:energung} one can use a similar approximation of the functions $v_\delta$ given in \cite[Equation (2.10)]{Anna} which exceed any obstacle for $\delta$ small enough and satisfy $\E(v_\delta) \leq c_0 + o(\delta)$.
\end{proof}
In the rest of the article we construct a special class of (large) obstacles for the flow can not subconverge.  
\begin{definition} [Cone Obstacle] 
Let $A > 0 $ and $a \in (0,\frac{1}{2})$. Define 
\begin{equation*}
\psi_{A,a}(x) := \begin{cases}
A(x-a) & x \in (0,\frac{1}{2}],
\\
A(1-x-a) & x \in (\frac{1}{2},1).
\end{cases} 
\end{equation*}
\end{definition}
\begin{remark}
For $\psi = \psi_{A,a}$, all the assumptions of Theorem \ref{thm:subcon} are satisfied with $\eta = \frac{1}{2}$ 
\end{remark}
\begin{lemma}[Energy and Contact Set]
Let $\psi= \psi_{A,a}$ for some $A  > 0$, $a \in (0, \frac{1}{2})$  and  let $u \in H$ be symmetric and such that $u \geq \psi $ almost everywhere. If there is $x_0 \neq \frac{1}{2}$ such that $u(x_0) = \psi(x_0)$, then
 \begin{equation}\label{eq:groen}
 \E(u) \geq 4G(A)^2 \min \left\lbrace \frac{1}{2a}, \frac{1}{1-2a} \right\rbrace.
\end{equation}  
\end{lemma}

\begin{proof}
Note that $u(x_0) = \psi(x_0)$ and symmetry of $u, \psi$ imply that $u(1- x_0) = \psi(1-x_0) $. Therefore one can assume without loss of generality that $x_0 \in (0, \frac{1}{2}) $. We distinguish two cases, namely $x_0 \geq a $ or $x_0 < a $. If $x_0 \geq a$ one obtains using \eqref{eq:cauch} with $x_1 = x_0$ and $x_2 = 1-x_0$ that 
\begin{equation*}
\E(u) \geq  \frac{(G(u'(x_0)) - G(u'(1-x_0)))^2 }{1 - 2x_0}.
\end{equation*}
Also note that $u'(x_0) = \psi'(x_0) = A $ and $u'(1-x_0) = \psi'(1-x_0) = - A$ which is due to the fact that $u- \psi$ is $C^1$ in a neighborhood of $x_0$  and attains its global minimum at $x_0$. Hence 
$
\E(u) \geq  \frac{4G(A)^2 }{1 - 2x_0} \geq \frac{4G(A)^2}{1- 2a } ,
$ the claim in this case.
Now suppose that $x_0 < a$. Then $u(x_0) = \psi(x_0)= A(x_0- a)  <0 $. Note that $u'(x_0) = \psi'(x_0) = A > 0 $ and since $u(x_0) < u(0)$ there has to be $\eta' \in (0,x_0)$ such that $u'(\eta') < 0 $. Because of the intermediate value theorem there exists $\eta \in (\eta',x_0)$ such that $u'( \eta) = 0 $. Hence $u'(x_0) = - u'(1-x_0) =A $ and $u'(\eta) = -u'(1-\eta) = 0 $. Using this and the Cauchy Schwarz inequality like in the proof of Proposition \ref{prop:cauch} we find
\begin{align*}
\E(u) & \geq \int_{\eta}^{x_0} \frac{u''(x)^2}{(1+ u'(x)^2)^\frac{5}{2}} \; \mathrm{d}x + \int_{1-x_0}^{1-\eta} \frac{u''(x)^2}{(1+ u'(x)^2)^\frac{5}{2}} \; \mathrm{d}x 
\\ & \geq \frac{(G(u'(\eta)) - G(u'(x_0)))^2 }{x_0 - \eta } + \frac{(G(u'(1- \eta)) - G(u'(1- x_0)) )^2 }{x_0 - \eta  } \\ & = \frac{2 G(A)^2 }{x_0 - \eta } \geq \frac{4 G(A)^2}{2x_0} \geq \frac{4G(A)^2}{2a}. \qedhere
\end{align*}
\end{proof}
\begin{cor}[Non-Subconverging Solutions] \label{cor:4.23}
Suppose that $\psi = \psi_{A,\frac{1}{4}}$ with  
\begin{equation}\label{eq:5.109}
A  > \max\left\lbrace G^{-1} ( \tfrac{c_0}{\sqrt{6}})  , \tfrac{8}{c_0},  4 \sup_{B > 0 }  \frac{1}{3}B \frac{\mathit{HYP2F1}( 1, \frac{3}{2} ; \frac{7}{4} , -B^2) }{\mathit{HYP2F1}(\frac{1}{2},1, \frac{3}{4}, -B^2)}  \right\rbrace,
\end{equation}
where $\mathit{HYP2F1}$ denotes the hypergeometric function in \cite[Definition 2.1.5]{Andrews}. 
Let  $u_0 \in H \cap W^{3,\infty}(0,1)$ be symmetric such that $u_0''(0) = u_0''(1) = 0$ and $\E(u_0) < \frac{4}{3}c_0^2$. Then there is no symmetric $u \in H \cap W^{3,2}(0,1)$ such that $u \geq \psi$, $u''(0) = u''(1) = 0 $, $\E(u) \leq  \E(u_0)$ and 
\begin{equation}\label{eq:variationalineq}
(\nabla \E(u) , v- u) \geq 0 \quad \forall \; v \in C.
\end{equation}
 In particular, for the Obstacle Gradient Flow with initial data $u_0$, case (2) in Theorem \ref{thm:subcon} is excluded.
\end{cor}
\begin{remark}
Note that $u_0$ as in the statement of Corollary \ref{cor:4.23} actually exists because of \eqref{eq:energung}. 
\end{remark} 
\begin{proof}[Proof of Corollary \ref{cor:4.23}] Notice that the supremum on the right hand side of \eqref{eq:5.109} is finite by \cite[Proof of Theorem 1.1, Lemma C.6]{Mueller}. Let $A$ be as in the statement and $a = \frac{1}{4}$. Further, we denote $\psi = \psi_{A,a}$. Note that 
\begin{equation}\label{eq:contrapsi}
\psi(\tfrac{1}{2}) = A ( \tfrac{1}{2} - a) >  \max\left\lbrace \tfrac{1}{4} G^{-1} ( \tfrac{c_0}{\sqrt{6}} ) ,  \tfrac{2}{c_0},   \sup_{B > 0 }  \frac{1}{3}B \frac{\mathit{HYP2F1}( 1, \frac{3}{2} ; \frac{7}{4} , -B^2) }{\mathit{HYP2F1}(\frac{1}{2},1, \frac{3}{4}, -B^2)}  \right\rbrace
\end{equation}
 and assume for a contradiction that there exists $u \in H \cap W^{3,2}(0,1)$ symmetric such that $u \geq \psi$, $\E(u) < \frac{4}{3}c_0^2 $, $u''(0)=u''(1)= 0$ and $(\nabla \E(u), v - u ) \geq 0 $ for each $v \in C$.  We distinguish cases analyzing the coincidence set $I:= \{ u = \psi \} $. If $I = \emptyset$ then one obtains from \eqref{eq:variationalineq} with the techniques of \cite[Theorem 6.9 Section 2]{Stampacchia} that $\nabla \E(u) \equiv 0 $. In this case, \cite[Lemma 4]{Deckelnick} implies that there exists $c \in (0,c_0)$ 
\begin{equation*}
u(x) =  \frac{2}{c \sqrt[4]{1 + G^{-1} \left( \frac{c}{2} - cx \right)^2 }} - \frac{2}{c\sqrt[4]{1 + G^{-1} \left( \frac{c}{2}\right)^2}} \quad x \in [0,1].
\end{equation*}
for some $c \in (0,c_0)$. Given this, \cite[Corollary 3]{Deckelnick}  and \eqref{eq:contrapsi} imply that 
\begin{equation*}
u \left( \tfrac{1}{2} \right) < \tfrac{2}{c_0} \leq A \left( \tfrac{1}{2} - a \right) = \psi \left(\tfrac{1}{2} \right)  ,
\end{equation*}
a contradiction to $u \geq \psi$. Now suppose that $I$ contains some point $x_0 \neq \frac{1}{2}$. Then because of \eqref{eq:groen} and $A \geq G^{-1}( \frac{c_0}{
\sqrt{6}} )$ one has
\begin{equation*}
\E(u) \geq 4 G(A)^2 \frac{1}{\nicefrac{1}{2} } = 8 G(A)^2 \geq 8 \tfrac{c_0^2}{6} = \tfrac{4}{3} c_0^2, 
\end{equation*}
which is a contradiction to $\E(u) < \frac{4}{3}c_0^2 $.  The only possibility that remains is $I = \{ \frac{1}{2} \} .$ From here we proceed similar to \cite[Lemma 3.3, Proof of Theorem 1.5]{Mueller}. We first show that there is $C \in \mathbb{R}$ such that 
\begin{equation}
\frac{v'(x)}{(1+ u'(x)^2)^\frac{5}{4}} = C \quad  \textrm{a.e. in } \left(0, \nicefrac{1}{2} \right)  ,  
\end{equation}
\begin{equation}\label{eq:symmi}
\frac{v'(x)}{(1+ u'(x)^2)^\frac{5}{4}} = - C \quad  \textrm{a.e. in } \left( \nicefrac{1}{2} , 1 \right),
\end{equation}
where $v(x) := \frac{u''(x)}{(1+ u'(x)^2)^\frac{5}{4}}.$ Indeed, for $\phi \in C_0^\infty(0, \frac{1}{2})$ and $|\epsilon| < \epsilon_0 $ small enough we have $u + \epsilon \phi \in C$. Therefore
 \begin{align*}
0 & = (\nabla \E(u), \phi) = \int_0^1 \frac{2u'' \phi''}{(1+ u'^2)^\frac{5}{2}} \; \mathrm{d}x - \frac{5 u''^2 u' \phi'}{(1+u'^2)^\frac{7}{2}} \; \mathrm{d}x  \\ & = \int_0^1 \frac{2v\phi'' }{(1+ u'^2)^\frac{5}{4}} \; \mathrm{d}x- \frac{5v^2u' \phi'}{(1+u'^2)} \; \mathrm{d}x  
\\ & = \left[ \frac{2v \phi'}{(1 + u'^2)^\frac{5}{4}} \right]_0^1 - \int_0^1 \phi' \left( \frac{2v' }{(1+u'^2)^\frac{5}{4}} - \frac{5 v u' u''}{(1 + u'^2)^\frac{9}{4}} \right) \; \mathrm{d}x -\int_0^1 \frac{5v^2u' \phi'}{(1+u'^2)} \; \mathrm{d}x .
\end{align*} 
Using $v(0) = \frac{u''(0)}{(1+ u'(0)^2)^\frac{5}{4}}= 0 $, $v(1) = \frac{u''(1)}{(1+ u'(1)^2)^\frac{5}{4}} = 0 $ and 
$
\frac{v u' u''}{(1 + u'^2)^\frac{9}{4}} = \frac{v^2u'}{1+ u'^2}$ almost everywhere
we obtain 
\begin{equation*}
0  = \int_0^1 \frac{2 v'\phi'}{(1+ u'^2)^\frac{5}{4}} \; \mathrm{d}x.
\end{equation*}
Hence there is $C \in \mathbb{R}$ such that 
\begin{equation*}
 \frac{v'(x)}{(1+ u'(x)^2)^\frac{5}{4}}= C \quad \textrm{a.e. in } \left( 0, \nicefrac{1}{2} \right)  .
\end{equation*}
Equation \eqref{eq:symmi} follows by symmetry of $u$. If now $C = 0$ then $v \equiv  const.$ on $(0,1)$, and since $v(0) = v(1) = 0 $, we have $v  \equiv 0 $. As a result $u$ is a line, which is a contradiction to the fact that $u \geq \psi$. Therefore $C \neq 0 $. We show next that $C < 0$. Indeed, $v'(x) = C(1+u'(x)^2)^\frac{5}{4}$ and $v(0) = 0$ imply 
\begin{equation*}
v(x) = C \int_0^x (1+ u'(s)^2)^\frac{5}{4} \; \mathrm{d}s
\end{equation*}
and hence using the definition of $v$,
\begin{equation}\label{eq:concavity}
u''(x) = C (1 + u'(x)^2)^\frac{5}{4} \int_0^x (1 + u'(s)^2)^\frac{5}{4} \; \mathrm{d}s. 
\end{equation}
Integrating once we find that 
\begin{equation*}
u'(x) - u'(0) = \frac{C}{2} \left( \int_0^x (1+ u'(s)^2)^\frac{5}{4} \; \mathrm{d}s \right)^2 .
\end{equation*}
Hence $u'$ is either increasing or decreasing in $(0, \frac{1}{2})$ depending on the sign of $C$. If $u'$ is increasing then $u'(x) \leq u'(\frac{1}{2}) = 0$ for each $x \in (0, \nicefrac{1}{2})$ which implies that $u$ is decreasing in $(0, \nicefrac{1}{2})$. But this makes $u \geq \psi$ impossible. A contradiction. Hence $u'$ is decerasing and therefore $u$ is concave on $(0, \nicefrac{1}{2})$. In particular $u'' \leq 0$ and together with \eqref{eq:concavity}, $C< 0 $. Equation \eqref{eq:concavity} reveals moreover that $u'' < 0$ on $(0,\nicefrac{1}{2})$. By symmetry, $u'' < 0 $ on $(0,1)$, so $u$ is strictly concave on $(0,1)$. Once this is shown, one can compute exactly like in \cite[Lemma 3.3, Theorem 1.1]{Mueller} that
\begin{equation*}
u\left( \tfrac{1}{2} \right) \leq \sup_{B > 0 }  \frac{1}{3}B \frac{\mathit{HYP2F1}( 1, \frac{3}{2} ; \frac{7}{4} , -B^2) }{\mathit{HYP2F1}(\frac{1}{2},1, \frac{3}{4}, -B^2)} .
\end{equation*}
We sketch the arguments for the sake of convenience of the reader. Multiplying $v'(x) = C(1+ u'(x)^2)^\frac{5}{4}$ by $v(x)$ we obtain 
 \begin{equation*}
 v(x) v'(x)  = C u''(x) \quad  \; a.e.\;  0 \leq x \leq \frac{1}{2}.
\end{equation*}  
Integrating we obtain that
\begin{equation*}
v(x)^2 = C (u'(x) - u'(0) ) \quad 0 \leq x \leq \frac{1}{2}
\end{equation*}
and therefore, since $u'$ is strictly decreasing and $u'' < 0$ on $(0,1)$ the only possible sign option is 
\begin{equation*}
 u''(x) = C  \sqrt{u'(0) - u'(x) } (1+ u'(x)^2)^\frac{5}{4}. 
\end{equation*}
In particular, dividing by $ \sqrt{u'(0) - u'(x) } (1+ u'(x)^2)^\frac{5}{4}$ and integrating one obtains 
\begin{equation}\label{eq:5.125}
\int_{u'(0)}^{u'(x)} \frac{1}{\sqrt{u'(0) - z} (1+ z^2)^\frac{5}{4}} \; \mathrm{d}z = Cx .
\end{equation}
Since by symmetry $u'(\frac{1}{2}) = 0 $ one has 
\begin{equation}\label{eq:inteconst}
\frac{C}{2} = \int_{u'(0)}^{0}  \frac{1}{\sqrt{u'(0) - z} (1+ z^2)^\frac{5}{4}} \; \mathrm{d}z .
\end{equation}
Define $F_0(s) := \int_{u'(0)}^{s} \frac{1}{\sqrt{u'(0)- z }(1+ z^2)^\frac{5}{4}} \; \mathrm{d}z $ for $s < u'(0)$.  Then  $F_0$ is increasing and has a $C^1$ inverse $F_0^{-1}$ that satisfies  $u'(x) =F_0^{-1}(Cx)$ according to \eqref{eq:5.125}. Moreover, $F_0(0) = \frac{C}{2}$ and $F_0(u'(0)) = 0$.  Hence 
\begin{equation*}
u(\tfrac{1}{2}) = \int_0^\frac{1}{2} F_0^{-1} (Cx) \; \mathrm{d}x = \frac{1}{C} \int_0^\frac{C}{2} F_0^{-1}(w) \; \mathrm{d}w = \frac{1}{C}\int_{F_0^{-1}(0)}^{F_0^{-1}(\nicefrac{C}{2})} s F_0'(s) \; \mathrm{d}s 
\end{equation*}
Using \eqref{eq:inteconst}, $F_0^{-1}( \nicefrac{C}{2} ) = 0 $ and $F_0^{-1} (0) = u'(0)$ one has 
\begin{equation*}
u(\tfrac{1}{2} ) = \frac{\int_{u'(0)}^{0} \frac{z}{\sqrt{u'(0) - z} (1+ z^2)^\frac{5}{4}} \; \mathrm{d}z}{2 \int_{u'(0)}^{0} \frac{1}{\sqrt{u'(0) - z} (1+ z^2)^\frac{5}{4}} \; \mathrm{d}z}
\end{equation*}
Using \cite[Lemma C.5]{Mueller}  and \eqref{eq:contrapsi} we find 
\begin{equation*}
u( \tfrac{1}{2} ) = \frac{1}{3} u'(0) \frac{\mathit{HYP2F1}(1, \frac{3}{2} ; \frac{7}{4}, -u'(0)^2 )}{\mathit{HYP2F1}(\frac{1}{2}, 1; \frac{3}{4}, -u'(0)^2)} \leq \sup_{B > 0 }  \frac{1}{3}B \frac{\mathit{HYP2F1}( 1, \frac{3}{2} ; \frac{7}{4} , -B^2) }{\mathit{HYP2F1}(\frac{1}{2},1, \frac{3}{4}, -B^2)} < \psi( \nicefrac{1}{2} )
\end{equation*}
and one obtains a contradiction again. The claim follows.
\end{proof}

\begin{remark}
A computer assisted calculation for the expression in \eqref{eq:contrapsi} shows that 
\begin{equation*}
\tfrac{1}{4} G^{-1} ( \tfrac{c_0}{\sqrt{6}} ) <  \sup_{B > 0 }  \frac{1}{3}B \frac{\mathit{HYP2F1}( 1, \frac{3}{2} ; \frac{7}{4} , -B^2) }{\mathit{HYP2F1}(\frac{1}{2},1, \frac{3}{4}, -B^2)}  = \tfrac{2}{c_0} = U_0 \left( \tfrac{1}{2} \right) \approx 0.84 .
\end{equation*} 
Then the quantity in \eqref{eq:contrapsi} coincides with $U_0(\tfrac{1}{2})$, which is also the highest value a symmetric critical graph curve of \eqref{eq:willmore} can attain, see Remark \ref{rem:funkschar}. This can be seen as a sharpness result for Corollary \ref{cor:4.16} in the following sense: If $\psi < U_0$ then the flow subconverges for symmetric initial data with small energy. On the contrary, \eqref{eq:contrapsi} implies that we can find examples of obstacles that exceed $U_0$ 'only by a little' with the property that the flow starting at symmetric initial data of small energy cannot subconverge.
\end{remark}

\appendix 

\section{Technical Proofs} \label{sec:App}

\begin{proof}[Proof of Proposition \ref{prop:standard}]
Let $(u_n)_{n \in \mathbb{N}} \subset C$ be a minimizing sequence for $\Phi_{v, \tau}$. Note that then 
\begin{equation*}
||u_n- v|| = \sqrt{ 2\tau(\Phi_{v,\tau}(u_n) - \E(u_n)) } \leq \sqrt{ 2\tau(\Phi_{v,\tau}(u_n) - \alpha) }
\end{equation*}
which implies that $(u_n)_{n \in \mathbb{N}}$ is bounded. Therefore $(u_n)_{n \in \mathbb{N}}$ possesses a weakly convergent subsequence which we call $(u_n)_{n \in \mathbb{N}}$ again for the sake of simplicity. Let $w \in H$ denote its weak limit. Since $C$ is weakly closed we infer that $w \in C$. Since $\Phi_{v,\tau}$ is weakly lower semicontinuous as the sum of two weakly lower semicontinous functionals (see Assumption \ref{ass:main} and \cite[Proposition 3.5(iii)]{Brezis}), we obtain 
\begin{equation*}
\Phi_{v,\tau}(w) \leq \liminf_{n \rightarrow \infty } \Phi_{v, \tau} (u_n)  = \inf_{u \in C} \Phi_{v, \tau} (u) . 
\end{equation*}
Equation \eqref{eq:eulerlag} follows after an easy computation from the inequality  
\begin{equation*}
0 \leq \frac{d}{d\epsilon}_{\mid_{\epsilon=0}} \Phi_{v,\tau} (\epsilon u + (1- \epsilon)w) \quad \forall u \in C. \qedhere
\end{equation*} 
\end{proof}

\begin{proof}[Proof of Proposition \ref{prop:3.4}]
The Frechét differentiability and the formula for the Frechét derivative follows after a straightforward computation from
\begin{equation*}
\E(u) = \int_0^1 G'(u')^2 u''^2 \; \mathrm{d}x + \int_0^1 K(u') \; \mathrm{d}x \quad \forall u \in H. 
\end{equation*}
We continue verifying Assumption \ref{ass:main} and \ref{ass:growth}. 
The energy $\E$ is clearly bounded from below by $0$. 
To show weak lower semicontinuity, suppose that $u_n \rightharpoonup u $ in $H$. Then $u_n' \rightarrow u'$ uniformly on $[0,1]$. This implies that there is $R> 0 $ such that $||u'||_\infty, ||u_n'||_\infty \leq R $ for each $n \in \mathbb{N}$. Hence 
\begin{align*}
|K(u_n'(x)) - K(u'(x))| & \leq ||u_n' - u||_\infty \int_0^1 |K'(s u_n'(x) + (1-s) u'(x))| \; \mathrm{d}s \\ & \leq ||u_n' - u'||_\infty \sup_{z \in B_R(0)} |K'(z)|  \rightarrow 0  \quad (n  \rightarrow \infty).
\end{align*}
 Thus, since the convergence is uniform
\begin{equation*}
\int_0^1 K(u_n') \; \mathrm{d}x \rightarrow \int_0^1 K(u') \; \mathrm{d}x .
\end{equation*}
Additionally,
\begin{align*}
 \int_0^1 &G'(u') u''^2 \; \mathrm{d}x  = \lim_{n \rightarrow \infty} \int_0^1 G'(u')^2 u'' u_n'' \; \mathrm{d}x \\ & \leq \liminf_{n \rightarrow \infty} \int_0^1 G'(u')G'(u_n') u'' u_n'' \; \mathrm{d}x + \int_0^1 |G'(u')| \;  |G'(u') - G'(u_n')|  \; |u''| \; | u_n''|  \; \mathrm{d}x  
 \\ & \leq  \liminf_{n \rightarrow \infty} \sqrt{\int_0^1 G'(u')^2 u''^2 \; \mathrm{d}x } \sqrt{\int_0^1 G'(u_n')^2 u_n''^2 \; \mathrm{d}x } \\ & \quad \quad  +  ||G'(u') - G'(u_n')||_\infty ||u''||_{L^2} ||u_n''||_{L^2} \sup_{z \in B_R(0) } |G'(z)| \\ 
 & \leq\liminf_{n \rightarrow \infty}  \sqrt{\E(u) } \sqrt{\E(u_n)} \\ & \quad \quad  +   ||u' - u_n'||_\infty ||u''||_{L^2} ||u_n''||_{L^2} \sup_{s,x \in [0,1]} |G''(su'(x) + (1-s) u_n'(x))| \sup_{z \in B_R(0)} |G'(z)|
 \\ & \leq\liminf_{n \rightarrow \infty}  \sqrt{\E(u) } \sqrt{\E(u_n)} \\ & \quad \quad  +  ||u' - u_n'||_\infty ||u''||_{L^2} ||u_n''||_{L^2} \sup_{z \in B_R(0)} |G''(z)| \sup_{z \in B_R(0)} |G'(z)|
 \\ & = \sqrt{\E(u)} \sqrt{\liminf_{n \rightarrow \infty} \E(u_n)}.
\end{align*}
The lower semicontinuity follows. We proceed verifying Assumption \ref{ass:growth} and the rest follows immediately. Let $u,v \in H$. Denote by $C_P$ the operator norm of the embedding operator $\iota :  W^{2,2}(0,1)\cap W_0^{1,2} (0,1)  \hookrightarrow W^{1, \infty} (0,1)\cap W_0^{1,2}(0,1)$. Also define 
$B:= B_{||u'||_\infty + ||v'||_\infty}(0)$.
\begin{align*}
||\nabla \E(u) - \nabla \E(v) || &  = \sup_{ ||\phi|| \leq 1 }  \int_0^1 2 (u'' G'(u') ^2 - v'' G'(v')^2) \phi''  + 2 \int_0^1 (K'(u') - K'(v')) \phi'   \\  &\qquad \qquad \quad  +  2 \int_0^1  ( G'(u') G''(u') u''^2 - G'(v') G''(v') v''^2 ) \phi' 
\\  \leq  \sup_{ ||\phi|| \leq 1 }  & 2 ||\phi|| \; ||u''G'(u')^2 -v''G'(v')^2||_{L^2}\\ &\quad   + 2 \sup_{ z \in B } |K''(z)| \;   ||\phi'||_{L^\infty} ||u' - v'||_{L^\infty} \\ &  \quad + ||\phi'||_{L^\infty} ||u''||_{L^2}^2 || G'(u') G''(u') - G'(v') G''(v')||_{L^\infty} \\ & \quad + ||\phi'||_{L^\infty} \int_0^1  |G'(v') G''(v')| |u''^2 - v''^2| \; \mathrm{d}x.
\end{align*}
Note that for a.e $x \in (0,1)$
\begin{align*}
& |G'(u'(x)) G''(u'(x)) - G'(v'(x)) G''(v'(x))| \\ & \qquad  \leq \left\vert  \int_0^1 (G''^2 + G' G''') (s u'(x) + (1-s) v'(x)) \; \mathrm{d}s \right\vert ||u' -v'||_\infty  \\ & \qquad \leq \sup_{z \in B } |G''(z)^2 + G'(z) G'''(z) | \; ||u' - v'||_\infty   
\end{align*}
and 
\begin{equation*}
| u''^2 - v''^2 | = | (u''- v'')^2  +2 v'' (u'' - v'') | \leq |u'' - v''|^2 + 2 |v''| |u'' - v''| .
\end{equation*}
Therefore 
\begin{align*}
||\nabla \E(u) - \nabla& \E(v) ||  \leq  \sup_{ ||\phi|| \leq 1 }   2 ||\phi|| ( ||(u'' -v'')G'(u')^2||_{L^2} + ||u''||_{L^2} || G'(u')^2 - G'(v')^2||_{\infty} )\\  &\quad \qquad \qquad 
+ 2  ||\phi'||_{L^\infty} ||u' - v'||_{L^\infty} \sup_{ z \in B } |K''(z)| 
 \\ &  \qquad \qquad \quad + ||\phi'||_{L^\infty} ||u''||_{L^2}^2 ||u' - v'||_\infty\sup_{z \in B } |G''(z)^2 + G'(z) G'''(z) |   \\ & \quad\qquad \qquad  +  ||\phi'||_{L^\infty} \sup_{z \in B}|G'(z)G''(z) | \int_0^1   |u''^2 - v''^2| \; \mathrm{d}x \\
\leq  \sup_{ ||\phi|| \leq 1 }  & 2 ||\phi|| ( \sup_{z \in B} |G'(z)| ^2\;  ||(u'' -v'')||_{L^2} + ||u''||_{L^2} \sup_{z \in B} |2(G'G'')(z)|\; ||u' - v'||_\infty )\\  &\quad 
+ 2 \sup_{ z \in B } |K''(z)| \;  ||\phi'||_{L^\infty} ||u' - v'||_{L^\infty}
 \\ &  \quad + ||\phi'||_{L^\infty} ||u''||_{L^2}^2 \sup_{z \in B} |G''(z)^2 + G'(z) G'''(z) | \; ||u' - v'||_\infty  \\ & \quad +  ||\phi'||_{L^\infty} \sup_{z \in B}|G'(z)G''(z) | (||u'' - v''||_{L^2}^2 + 2 ||v''||_{L^2} \; ||u'' -v''||_{L^2}) \\
 \leq  \sup_{ ||\phi|| \leq 1 }  &  2 ||\phi|| ( \sup_{z \in B} |G'(z)|^2 \; ||(u'' -v'')||_{L^2} + ||u''||_{L^2} \sup_{z \in B} |2(G'G'')(z)| \; ||u' - v'||_\infty )\\  &\quad 
+ 2 \sup_{ z \in B} |K''(z)| \;  ||\phi'||_{L^\infty} ||u' - v'||_{L^\infty}
 \\ &  \quad + ||\phi'||_{L^\infty} ||u''||_{L^2}^2 \sup_{z \in B} |G''(z)^2 + G'(z) G'''(z) | \;  ||u' - v'||_\infty  \\ & \quad +  ||\phi'||_{L^\infty} \sup_{z \in B}|(G'G'')(z) | (||u''||_{L^2} + 3||v''||_{L^2}) ||u-v||.  
\end{align*}
Estimating all $L^\infty$ norms with $C_P ||\cdot||$ one obtains the claim with 
\begin{align*}
\zeta(r) :=&  2 \sup_{ B_{C_pr}(0)} |G'|^2 + 2r \sup_{B_{C_pr}(0)}|G'G''|  + 2C_p^2 \sup_{B_{C_pr}(0)} |K''| \\ & + C_p^2 r^2 \sup_{B_{C_pr}(0)} |G''^2 + G'G'''| + 3C_pr \sup_{B_{C_pr}(0)} |G'G''| .
\end{align*}
\end{proof}

\begin{proof}[Proof of Lemma \ref{lem:ind:rec}]
We prove the claim by induction over $l$. The case $l = 0 $ is clear. Suppose now the claim is true for each $l \in \{0,1,...,k\}$. We prove that it is also true for $l = k+1$. 
\begin{align*}
a_{k+1} & \leq (1+ Z_1 \tau) a_k + Z_2 \tau  \\ 
& \leq (1 + Z_1 \tau) \left( (1+ Z_1 \tau)^k a_0 + \frac{Z_2}{Z_1} ((1+Z_1 \tau)^k - 1)  \right) + Z_2 \tau 
\\ & = (1+ Z_1 \tau)^{k+1} a_0 + \frac{Z_2}{Z_1} \left((1 + Z_1 \tau)^{k+1} - (1+ Z_1 \tau)  \right) + Z_2 \tau 
\\ & =   (1+ Z_1 \tau)^{k+1} a_0 + \frac{Z_2}{Z_1} ( ( 1+ Z_1 \tau)^{k+1} - 1 ). \qedhere 
\end{align*} 
\end{proof}

\begin{proof}[Proof of Lemma \ref{lem:ODE}]
The product rule for Sobolev functions, see \cite[Theorem 4 in Section 4.2.2]{EvansGariepy} implies that 
\begin{equation}\label{eq:dingenskirchen}
h(t) := \exp \left( -\int_a^t \alpha(s) \; \mathrm{d}s \right) f(t) 
\end{equation}
is weakly differentiable with weak derivative 
\begin{equation*}
h'(t) = \exp \left( -\int_a^t \alpha(s) \; \mathrm{d}s \right) \beta(t) . 
\end{equation*}
Therefore $h \in W^{1,1}(a,b)$ and hence 
\begin{align*}
h(t) & = h(a) + \int_a^t  \exp \left( -\int_a^s \alpha(w) \; \mathrm{d}w \right) \beta(s) \; \mathrm{d}s \\&   = f(a) + \int_a^t  \exp \left( -\int_a^s \alpha(w) \; \mathrm{d}w \right) \beta(s) \; \mathrm{d}s. 
\end{align*}
 The claim follows together with \eqref{eq:dingenskirchen}.
\end{proof}
\bibliographystyle{amsplain}

\end{document}